\documentclass[a4paper,11pt]{article}
\usepackage[utf8]{inputenc}
\usepackage{graphicx, subcaption}
\usepackage{amsmath, amssymb, amsthm}
\usepackage{algorithm, algorithmic}
\usepackage{hyperref}
\usepackage{xcolor}
\usepackage{stmaryrd}
\usepackage{bbm}
\usepackage{enumitem}
\usepackage{pifont}

\hypersetup{
  colorlinks=true,
  linkcolor=teal,
  filecolor=blue,
  citecolor=cyan,
  urlcolor=violet,
}

\theoremstyle{plain}
\newtheorem{thm}{Theorem}[section]

\newtheorem{cor}[thm]{Corollary}

\newtheorem{lemma}[thm]{Lemma}
\newtheorem{pro}[thm]{Proposition}
\newtheorem{defi}[thm]{Definition}

\newtheorem{rem}[thm]{Remark}

\newtheorem{exas}[thm]{Examples}
\newtheorem{notas}[thm]{Notations}

\theoremstyle{definition}

\newcommand{\norm}[1]{\left\lVert#1\right\rVert}
\newcommand{\abs}[1]{\left\lvert#1\right\rvert}
\newcommand{\R}{\mathbb{R}}

\newcommand{\C}{\mathbb{C}}
\newcommand{\E}{\mathbb{E}}

\newcommand{\Ker}{{\rm Ker}}
\newcommand{\Img}{{\rm Im}}
\newcommand{\Sp}{{\rm Sp}}
\newcommand{\Vect}{{\rm Vect}}
\newcommand{\lip}{\rm Lip_1}

\renewcommand{\S}{\mathcal{S}}

\DeclareMathOperator{\rank}{rank}

\title{Optimistic    Gradient Descent Ascent  in Zero-Sum and General-Sum  Bilinear  Games}
\author{\'{E}tienne de Montbrun, J\'er\^{o}me Renault}
\date{\today}

\begin{document}
\maketitle
\begin{abstract}
 We study the convergence of  Optimistic Gradient Descent Ascent in unconstrained  bilinear games. In a first part,  we consider the   zero-sum case and extend previous results by Daskalakis {\it{et al.}}  \cite{daskalakis2017training},  Liang and Stokes \cite{liang2019interaction}  and others:  we  prove, for any payoff matrix,  the exponential convergence of OGDA to a saddle point and also provide a new, optimal,  geometric  ratio for the convergence. We also characterize the step sizes inducing convergence, and are able to deduce the optimal step size for the speed of convergence.   In a second part, we introduce OGDA for general-sum bilinear games: we   show that in an interesting class of games, either OGDA converges exponentially fast to a Nash equilibrium, or the payoffs for both players converge exponentially fast to $+\infty$ (which 
 might  be interpreted as    endogenous emergence of coordination, or cooperation, among players).  We also give sufficient conditions for convergence of OGDA to a Nash equilibrium. These conditions are used  to increase the speed of convergence of  a min-max problem involving a matrix $A$, by introducing a general-sum game using the   Moore-Penrose inverse matrix of $A$. This shows   for the first time, at our knowledge,   that general-sum games can   be used to optimally improve algorithms designed for  min-max problems.  We finally illustrate our results on   simple  examples of Generative Adversarial Networks. 
\end{abstract}
% TODO: Change for any payoff matrix because of Peng! --> What to do of it?

%Keywords: Optimistic Gradient Descent Ascent; Bilinear games; General-sum games; Nash equilibria
\section{Introduction}
Min-max optimization is receiving a lot of attention, due in particular to the popularity of generative adversarial networks (GANs, introduced in \cite{goodfellow2014generative}) and adversarial training (see for instance \cite{Madryadversarial}). In the original version of GANs, two neural networks are in competition: a generator that aims at  generating data as close as possible to some real data, and a discriminator whose goal is to differentiate between generated and true data. Because the discriminator wants to maximize a loss that the generator wants to minimize, GANs can be seen as a min-max, or zero-sum game,  problem.

Formally, consider a ``payoff'' function  $g:\mathcal{X}\times \mathcal{Y}\rightarrow \R$ and assume  we want to find  $y^*$ achieving $\min_{y \in \mathcal{Y}} \sup_{x\in \mathcal{X}} g(x,y)$. Assuming  the zero-sum game associated to $g$ has a Nash equilibrium (i.e., a saddle-point), the game-theoretic approach is to find an equilibrium, i.e.  $(x^*,y^*)$ in $\mathcal{X}\times \mathcal{Y}$ such that:
\begin{equation}\label{eq_NE} \forall x\in \mathcal{X}, \forall y \in \mathcal{Y}, g(x,y^*)\leq g(x^*,y^*)\leq g(x^*,y).\end{equation}

A natural algorithm to search for such  an equilibrium is the Gradient Descent-Ascent  algorithm (GDA): given a small parameter $\eta>0$ and a starting point $(x_0,y_0)$, define inductively the sequence:
\[\forall t\geq 0,\left\{ \begin{array}{ccccc}  x_{t+1}& =&x_t& +& \eta \; \frac{\partial g}{\partial x}(x_t,y_t), \\ y_{t+1} &=& y_t& -& \eta\;  \frac{\partial g}{\partial y}(x_t,y_t). \end{array}\right. \]
Most GANs currently use (stochastic) Gradient Descent Ascent   under the hood.
In some cases, GDA was acknowledged to converge {\it in average} to the solution of the min-max problem in several zero-sum games \cite{nedic2009subgradient}. 
However, it is  well known that even for bilinear  functions $g(x,y)=x^TAy$ with $\mathcal{X}=\mathcal{Y}=\R^n$, GDA may exhibit a cyclic  behavior and the last iterate may not converge (\cite{daskalakis2018limit} for example).  A nice variant of the GDA is the Optimistic Gradient  Descent-Ascent algorithm (OGDA): 
\begin{equation}\label{eqOGDA}
  \forall t\geq 0,\left\{ \begin{array}{ccccccc}  x_{t+1}& =&x_t& +& 2\eta \; \frac{\partial g}{\partial x}(x_t,y_t)&-&\eta \;  \frac{\partial g}{\partial x}(x_{t-1},y_{t-1}), \\ y_{t+1} &=& y_t& -& 2\eta\;  \frac{\partial g}{\partial y}(x_{t},y_{t})&+& \eta \; \frac{\partial g}{\partial y}(x_{t-1},y_{t-1}).  \end{array} \right. \end{equation}
This optimistic  variant has been introduced by Popov in \cite{popov1980modification} and rediscovered in the GAN literature (\cite{daskalakis2017training} for example). A challenge is then to determine when OGDA converges to an equilibrium of the game.

\vspace{0,5cm}

In this paper, we focus on the bilinear unconstrained case, where $\mathcal{X}=\R^n$, $\mathcal{Y}=\R^p$ and payoffs are bilinear.\\

We first consider in section \ref{section_bilinear} the zero-sum case where $g(x,y)=x^TAy$ for some matrix $A\in \R^{n\times  p}$. The difficulty here is not to  find a saddle-point of $g$ (this is fairly easy: take $x^*=y^*=0$) but to prove or disprove the convergence of OGDA to a saddle-point in this simple setup. In this context, Daskalakis {\it{et al.}}  \cite{daskalakis2017training}  proved the  last iterate  convergence of OGDA, in the sense that for each $\varepsilon>0$ one can find $\eta>0$ and $T$ large enough such that $(x_T,y_T)$ (as in (\ref{eqOGDA})) is $\varepsilon$-close to a Nash equilibrium. Liang and Stokes \cite{liang2019interaction}, as well as Mokthari {\it{et al.}}  \cite{mokhtari2020unified} proved that in the case of a square invertible matrix $A$, OGDA converges to $(0,0)$ at exponential speed.    Peng {\it{et al.}} \cite{peng2020training} extended the analysis to non square matrices.

We extend     in  Theorem \ref{thm_conv_OGDA_xAy} of  section \ref{subsec_xAy}  the   previous results of   the litterature. For any matrix $A$ and any initial conditions, OGDA does converge, and the limit  is  a  Nash equilibrium that we characterize (and is independent of $\eta$).  Moreover, the speed of convergence is exponential  and we exhibit the exact geometric ratio for the convergence,  as a function of the step size $\eta$ and the matrix $A$, improving the previous ratios obtained for square  invertible matrices. We also    characterize the range of step sizes where convergence holds,   extending the previous values considered for $\eta$.   We believe these results of last iterate convergence matter because they show a strong stability of the OGDA algorithm for any matrix: not only OGDA can be used to find ($\varepsilon$)-equilibria, but they will not exhibit  cyclic  behaviors within  the possibly large set of Nash equilibria.

In  section \ref{subsec_comments},   Theorem \ref{thm_conv_OGDA_xAy} is compared with the previous existing results, and  in section \ref{subsec_general_zs} Theorem \ref{thm_conv_OGDA_zs_general}    extends  the results of Theorem \ref{thm_conv_OGDA_xAy} to the bilinear case where $g(x,y)= x^TAy+b^Tx+c^Ty+d$. 

In section \ref{subsec_optimal_ratio},  we  compute as a function of $A$,   the step size $\eta$ optimizing  the geometric   ratio of convergence. Whereas the largest step size considered so far  in the literature was $\frac{1}{2 \sqrt{\mu_{\max}}}$(with $\mu_{\max}$ the largest eigenvalue of $AA^T$), it turns out that the optimal step size   is always  at least  $\frac{1}{2 \sqrt{\mu_{\max}}}$, and  greater than $\frac{1}{2 \sqrt{\mu_{\max}}}$ as soon as $AA^T$ has several nonzero eigenvalues. The optimal ratio varies  between $\frac{\sqrt{2}}{2}$ to $1$, and is characterized as a function of $\frac{\mu_{\min}}{\mu_{\max}}$, where $\mu_{\min}$ is the smallest  positive  eigenvalue of $AA^T$.

The proofs, in Appendix,  are based on a precise  spectral analysis of the linear system with variable $(x_{t+1},y_{t+1},x_t,y_t)$, one  difficulty being to control the angles between complex eigenspaces in order to obtain proper  rates of convergence.

% In subsection \ref{subsec_general_zs}, Theorem \ref{thm_conv_OGDA_zs_general} extends all the results of subsection \ref{subsec_conv_xAy} to any bilinear  payoff function $g$, i.e. we consider the  cases where $g(x,y)= x^TAy+b^Tx+c^Ty+d$. 

\vspace{0,5cm}

We then consider OGDA for general-sum bilinear games. We stick here to the case where $\mathcal{X}$ and $\mathcal{Y}$ are Euclidean spaces, and in section \ref{sec_nzs}  we introduce the natural extension of  OGDA for  bilinear  general-sum games. It is easy to see that OGDA need not   always converge to a Nash equilibrium  here, and  we introduce in Theorem \ref{thm_conv_DOGDA} a very simple variant of OGDA, called DOGDA (Double Optimistic Gradient Descent Ascent), which always  converges to a Nash equilibrium in a bilinear general-sum game $(x^TAy,x^TBy)$\footnote{The ``DOGDA trick'' is simply based on a strong specifity of the games considered, namely that $(x^*,y^*)$ is a Nash equilibrium of the game if and only if ($x^*$ is an optimal strategy for player 1 in the zero-sum game with matrix $-B$ and $y^*$ is an optimal strategy for player 2 in the zero-sum game with matrix $A$).}. We then come back to  OGDA and  give in Theorems \ref{prop_conv_OGDA_xAy_xBy} and \ref{thm_conv_OGDA_general_xAy_xBy}   sufficient conditions for convergence  to a Nash equilibrium, when  all eigenvalues of $B^TA$ are real negative.

In section \ref{improvingwgsg}, we come back to the problem of finding 
 $y^*$ achieving $\min_{y \in \R^p} \sup_{x\in \R^n} g(x,y)$, where $g(x,y)=x^TAy+b^Tx+c^Ty+d$ is bilinear. One can use the zero-sum results of section \ref{section_bilinear} and use OGDA for the zero-sum game with payoff functions $g$ for player 1 and $-g$ for player 2. We obtain the convergence of $(y_t)_t$ to an optimal $y^*$ at exponential speed with ratio in $[\frac{\sqrt{2}}{2},1)$depending on $\frac{\mu_{\min}}{\mu_{\max}}$. But one can also consider OGDA for the general sum-game with payoff function $g_1=g$ for player 1 and $g_2(x,y)=- x(A^\dagger)^Ty$, where $A^\dagger$ is the generalized (Moore-Penrose) inverse of $A$. It turns  out that Theorem \ref{thm_conv_OGDA_general_xAy_xBy} applies here, and for this OGDA as soon as 
$\eta\leq 1/2$ the sequence $(y_t)$ converges to the same optimal $y^*$, with a  geometric   rate  of   $\frac{\sqrt{2}}{2}$ when $\eta\simeq 1/2$. This shows  the interest of introducing OGDA for general-sum games:  the speed of convergence to a solution of a Min-Max problem can be improved if one considers a well suited general-sum game, instead of the natural  associated zero-sum game.  
 
The last part of section \ref{sec_nzs} considers OGDA for a particular class of payoff matrices, including in particular the case of common payoff games where $A=B$ (implying that there is a potential for coordination or cooperation among players), or more generally the cases where $B=\alpha A$ for some $\alpha\neq0$.  It is proved that for this class of games, either OGDA converges  to a Nash equilibrum,  or under OGDA the payoff for both players tends to $+\infty$ exponentially fast as $t\rightarrow +\infty$. We believe this result  is a meaningful extension of the convergence of OGDA, since both players having $+\infty$ payoffs can be seen as a  generalized Nash equilibrium  of the game and a very desirable outcome of the interaction. \\

In section \ref{sec_illustration} we illustrate our results on stylised versions of Generative Adversarial Networks. We first  recall the theoretical GANs and WGANs setups, and then illustrate our results  on linear instances of  Wasserstein GANs problem  in the spirit of    \cite{daskalakis2017training}. We show in particular on an example  how introducing a general-sum game using $B=-(A^\dagger)^T$ improves the speed of convergence of the OGDA. We also discuss limitations of the bilinear model with respect to real GANs applications. 

We finally  conclude and present a few possible research  directions   in section \ref{secconclusion}. The proofs are in the Appendix.

\vspace{1cm}

\noindent \textbf{Other Related Works}:

Besides \cite{daskalakis2017training},  \cite{liang2019interaction},  \cite{mokhtari2020unified} and 
\cite{peng2020training}, Zhang and Yu \cite{zhang2019convergence} also consider convergence of gradient methods in bilinear zero-sum games.  Gidel \textit{et al.}
  \cite{gidel2018variational} study GANs through variational inequalities, with a convergence ratio not better than  in  \cite{liang2019interaction}. 

In \cite{rakhlin2013optimization}, Rakhlin and Sridharan introduced an algorithm called Optimistic Mirror Descent (OMD), based on the Mirror Descent, and  showed that it is a no-regret algorithm. Optimistic Gradient Descent Ascent (OGDA) is a particular case of OMD, for which the regularizer used is the 2-norm. Another algorithm close to  OGDA is  the Extra-Gradient method (EG) introduced by Korpelevich in \cite{korpelevich1976extragradient}. Both OGDA and EG  can be seen as approximations of the Proximal Point (PP) method .
In \cite{mokhtari2020unified}, Mokhtari \textit{et al.} showed again  the exponential convergence of PP,  OGDA, and of EG, for  square  invertible bilinear games,  for a single  particular value of $\eta$, and with an exponential bound of convergence greater than ours.
Another article studying Extra-Gradient, including  a stochastic setting,  is the one of Hsieh \textit{et al.} \cite{hsieh2019convergence}.\\
%In \cite{daskalakis2017training}, Daskalakis \textit{et al.} showed that in bilinear games,  OGDA %converges $\mathcal{O}(\eta)$-close to a saddle point where $\eta$ is the step-size. 
%Later, Liang and Stokes \cite{liang2019interaction} improved this result by showing an exponential %convergence to a saddle point, for bilinear games with invertible and squared matrix for a fixed %step-size $\eta$ depending on the eigenvalues of the game.\\
%They gave another exponential bound of Extra-Gradient, but above all, they proved the convergence of Extra-Gradient in a stochastic setting.\\
Beyond bilinear games,  some papers consider OGDA for concave/convex games. In particular,  Daskalakis \textit{et al.} \cite{daskalakis2018limit} study the stability of fixed points of the dynamics. Non-concave/non-convex setting exhibit problematic  properties: \cite{hsieh2020limits}  shows there may exist  attractors of OGDA containing no stationary points, and  \cite{daskalakis2021complexity} shows that finding an approximate local minmax equilibrium is PPAD-hard. However, in \cite{diakonikolas2021efficient}, Diakonikolas {\it et al.} showed positive results in a new class of non concave/non convex minmax problems.

  Many studies where also done in a compact settings, that is when $\mathcal{X}$ and $\mathcal{Y}$  are compact sets. In \cite{daskalakis2018last}, Daskalakis \textit{et al.}  proved the convergence of an Optimistic version of Multiplicative Weight Updates on the simplex, for an appropriate value of the learning rate.
Later, Wei \textit{et al.} improved their results in \cite{wei2020linear}, by showing the convergence for any learning rate smaller than some constant. Finally, in a  recent paper \cite{anagnostides2022optimistic}, Anagnostides \textit{et al.} consider  the convergence of a projected version of OGDA for general-sum games, where after each iteration, a projection on the simplex is applied. \\

%Some papers go further, and study the behavior of OGDA in the general case of concave/convex games.
%It was in particular the case for the article of Mokhtari \textit{et al.} \cite{mokhtari2020unified}, %and of an article of Daskalakis \textit{et al.} \cite{daskalakis2018limit}.\\
%Another extension is the study of non-concave/non-convex games. Since Farnia's article %\cite{farnia2020gans}, we know that there may be no equilibria in non-concave/non-convex games showing %that we are bound to fail.
%Some other research were done on the subject, as \cite{hsieh2020limits}, that showed there exists %attractors of OGDA containing no stationary points. In \cite{daskalakis2021complexity} they showed %that finding an approximate local minmax equilibrium is PPAD-hard.
%However, in \cite{diakonikolas2021efficient}, Diakonikolas showed more positive results.\\

\begin{notas} \rm We use  the following standard notations:\\
  $\R_+$ denotes the set of non-negative reals and $\R_-$ the set of non-positive reals.
  $A^T$ denotes the transpose of a matrix $A$ in $\C^{n\times p}$. If  $A$ is square, $\Sp(A)$ is the set of complex eigenvalues of $A$, and $\rho(A)=\max\{|\lambda|, \lambda\in \Sp(A)\}$ is the spectral radius of $A$. 
  Given  column vectors $x$, $x'$  in $\R^n$ or more generally in $\C^n$, we use the Hermitian scalar product $\langle x,x'\rangle={x}^T \, \overline{x'}$ where $\overline{x'}$ is the complex conjugate of $x'$, and the   norm $\|x\|=\sqrt{\langle x,x\rangle}$. Given  two matrices $A$ and $B$ of the same size, we denote by $\S(A,B)$ the set $\Sp(B^TA) \cup \Sp(AB^T)$, and we simply write  $\S(A)$\footnote{Recall that $\S(A)\subset \R_+$, and $\S(A)=\{0\}$ if and only if $A=0$. If  $A$ and $B$ are matrices of the same size, then  $AB^T$ and $B^TA$ have  the same  nonzero eigenvalues.} for $\S(A,A)$.
$I_n$ denotes the identity matrix in $\R^{n\times n}$.

%$A^T$ denotes  the transpose of a matrix $A$. Given a square matrix $A$, $\Sp(A)$ is the set of complex eigenvalues of $A$, and %$\rho(A)=\max\{|\lambda|, \lambda\in \Sp(A)\}$ is the spectral radius of $A$. We use $\|x\|$ for the Euclidean norm, and  $\|A\|_{op}:=\sup\{\|Ax\|, %\|x\|\leq 1\}=\sqrt{\rho(A^TA)}$ for the operator norm of $A\in \R^{n \times p}$.

\end{notas}

% TODO: See if I can delete a line before, not to have this ugly pagination, with one line of a section at one place.

\section{Zero-Sum Bilinear Games}
\label{section_bilinear}

\subsection{Zero-Sum Games $x^TAy$}
\label{subsec_xAy}

 Let $A$ be a payoff matrix in $\R^{n \times p}$. We consider the zero-sum game where simultaneously player 1 chooses $x$ in $\R^n$, player 2 chooses $y$ in $\R^p$, and finally player 2 pays $x^TAy$ to player 1. Here $x$ and $y$ are seen as column vectors. This game has a value which is 0, and both players have optimal strategies:
 \[\max_{x \in \R^n} \inf_{y \in \R^p} x^T Ay=  \min_{y \in \R^p} \sup_{x \in \R^n}x^T Ay=0.\]
 The set of optimal strategies of player 1 is $\{x\in \R^n, x^TA=0\}=\Ker(A^T)$, and the set of optimal strategies of player 2 is $\{y\in \R^p, Ay=0\}=\Ker(A)$. So $(x,y)$ is a Nash equilibrium (or saddle-point) of the game  if and only if $x\in \Ker(A^T)$ and $y\in \Ker(A)$.

\begin{defi}The Gradient Descent Ascent is defined by:
\[\forall t\geq 0,\left\{ \begin{array}{ccccc}  x_{t+1}& =&x_t& +& \eta A y_t, \\ y_{t+1} &=& y_t& -& \eta A^T x_t, \end{array}\right. \]
where $\eta>0$ is a fixed parameter (the gradient step), and $(x_0,y_0)\in \R^n\times \R^p$ is the initialization.\end{defi}

Notice that the fixed points of the GDA are exactly the Nash equilibria of the game. 
It is well known  that the GDA does not converge  in general,  as one can see in the simple ``Matching Pennies''  case  where $n=p=1$ and  $A=(1)$,  so that  $x^TAy$ is simply $xy$.  In this case, we obtain:
\[\left\{ \begin{array} {ccc} 
  x_{t+1}&= &x_t+\eta y_t  \\
  y_{t+1} &= & y_t-\eta x_t \end{array}\right.,\]
  which implies $x_{t+1}^2+y_{t+1}^2=(x_t^2+y_t^2)(1+\eta^2)$, and $\|(x_t,y_t)\| \xrightarrow[t\to \infty]{}+\infty$ as soon as $(x_0,y_0)\neq(0,0)$. \\

To overcome this problem, 
\cite{rakhlin2013optimization} (see also  \cite{daskalakis2018limit})
  introduced ``optimism'' to GDA. 

\begin{defi}  The Optimistic Gradient Descent Ascent  (OGDA) is defined by:
\[\forall t\geq 0, \left\{ \begin{array}{ccccccc}  x_{t+1} &= & x_t &+ &2\eta A y_t& -& \eta A y_{t-1},\\ y_{t+1} &= & y_t &- &2\eta A^T x_t &+ &\eta A^T x_{t-1}, \end{array} \right.\]
where $\eta>0$ is a fixed parameter, and $(x_0,y_0,x_{-1},y_{-1})\in \R^n\times \R^p\times \R^n\times \R^p$ is the initialization.\end{defi}

The term of time $t-1$ is here to thwart the divergence of GDA, by bringing back toward the center the iterate of time $t+1$.
We can see empirically the importance of this ``optimistic'' addition in figure \ref{GDAvsOGDA} for the matching pennies problem where $A=(1)$.\\

\begin{figure}[h]
  \begin{subfigure}{0.5\textwidth}
  \includegraphics[width = \textwidth]{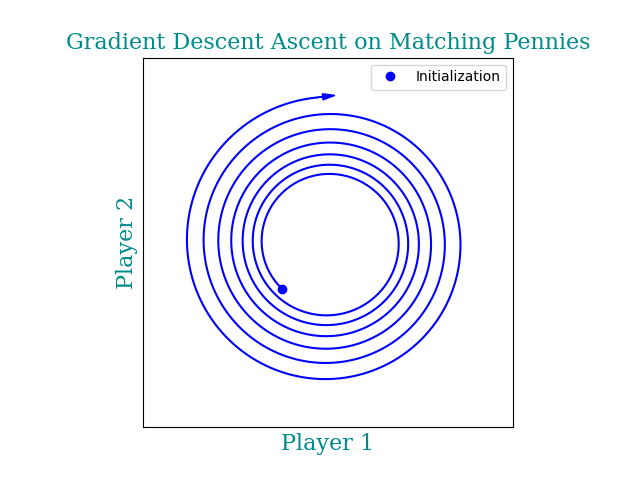}
  \end{subfigure}
  \begin{subfigure}{0.5\textwidth}
  \includegraphics[width = \textwidth]{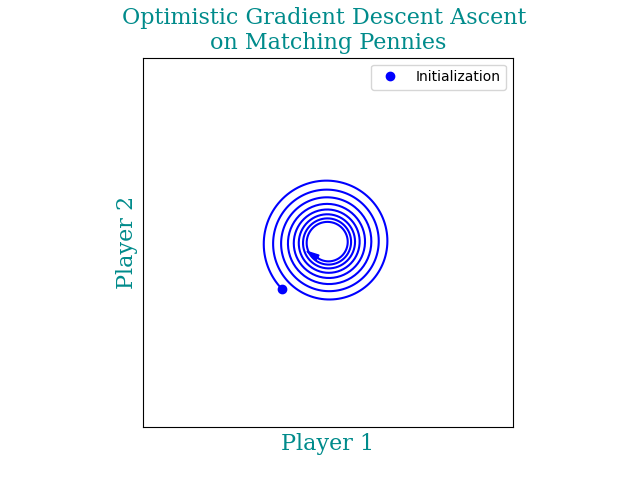}
  \end{subfigure}\\
  \caption{Comparison of GDA and OGDA for Matching Pennies}
  \label{GDAvsOGDA}
\end{figure}

We will see our optimistic GDA as a dynamical system in $\R^n\times \R^p\times \R^n\times \R^p$.  
 
\begin{defi}  Define the matrix 
\[ \Lambda = \begin{pmatrix} I_n & 2\eta A & 0 & -\eta A\\-2\eta A^T & I_p & \eta A^T & 0\\ I_n & 0 & 0 & 0\\ 0 & I_p & 0 & 0  \end{pmatrix} \in \R^{(n+p+n+p) \times (n+p+n+p)},\] and  for $t\geq 0$, let  $Z_t$ be the column vector ${Z_t =\left(\begin{array}{c} x_t\\ y_t\\ x_{t-1}\\y_{t-1}\end{array}\right) \in \R^n\times \R^p\times \R^n\times \R^p}$.
\end{defi}
The  matrix $\Lambda$ has already been introduced  in \cite{peng2020training}. We now have: 
  \begin{equation}Z_0\in \R^n\times \R^p\times \R^n\times \R^p  \;\;{\rm and}\;\; \forall t\geq 0,  Z_{t+1} = \Lambda  Z_t. \end{equation}
In the sequel we will denote by $I$ the identity matrix of size $2(n+p)$.

\subsection{Convergence of OGDA  for a zero-sum game $x^T Ay$}
\label{subsec_conv_xAy}

The  convergence of OGDA when $A=(1)$ is well known and can be shown as follows. We have
 \[ \Lambda = \begin{pmatrix} 1 & 2\eta  & 0 & -\eta  \\-2\eta & 1 & \eta & 0\\ 1 & 0 & 0 & 0\\ 0 & 1 & 0 & 0  \end{pmatrix},\]
with characteristic polynomial $\det(\Lambda-\lambda I)= \lambda^2(1-\lambda)^2+ \eta^2(1-2\lambda)^2.$ Assuming $\eta<1/\sqrt{3}$, easy computations show that  the  spectral radius of  $\Lambda$  is strictly smaller than 1.  Then $\Lambda^t\xrightarrow[t\to \infty]{} 0$ by Gelfand's theorem, so $(x_t,y_t)\xrightarrow[t\to \infty]{}(0,0)$.\\

We now  generalize this linear algebra approach to any matrix $A$, focusing not only on the convergence property but also on the characterization of the limit and on the speed of convergence.

%For this, we first need the following definition:

%\begin{defi}
%  For two matrices $A$ and $B$, we denote by $\S(A,B)$ the set $\Sp(B^TA) \cup \Sp(AB^T)$. To simplify, we use the notation
% $\S(A)$ to denote $\S(A,A)$.
%\end{defi}
\begin{thm}
  \label{thm_conv_OGDA_xAy}
Let $A\in \R^{n\times p}$ with $A\neq 0$  and $\eta>0$. We write $\mu_{\max}=\max\{\mu\in \R, \mu \in \Sp(AA^T)\}$, 
  $\mu_{\min}=\min\{\mu>0, \mu \in \Sp(AA^T)\}>0$, and  consider the  Optimistic Gradient Descent Ascent:
  \begin{align}
    \label{OGDA_algo}
    \forall t\geq 0,\;  \left\{
    \begin{array} {ccc} 
      x_{t+1} & = & x_t \;+\; 2\eta A y_t - \eta A y_{t-1},  \\
      y_{t+1} & = & y_t - 2\eta A^T x_t + \eta A^T x_{t-1}.
    \end{array}\right.
  \end{align}

If $\eta>\frac{1}{\sqrt{3\mu_{\max}}}$, there exist  initial conditions where the corresponding OGDA does not converge. In the sequel we assume that $$\eta<\frac{1}{\sqrt{3\mu_{\max}}}.$$
 Fix an  initial condition $(x_0, y_0,x_{-1},y_{-1})\in \R^n\times \R^p\times \R^n \times \R^p$, and denote by  $D$ be the distance from the initialization  to the set $\Ker(\Lambda-I)$. We have the following: \\

1) $(x_t, y_t)_t$ converges to the Nash equilibrium $\left(x_\infty, y_\infty\right)$, where  $x_\infty$ is the orthogonal projection of $x_0$ onto $\Ker(A^T)$ and $y_\infty$ is the orthogonal projection of $y_0$ onto $\Ker(A)$.\\
 
%  We  define $\lambda _{\max}\in (0,1)$ by $\lambda _{\max}=$ 
% $$\max\left\{{\bf{1}}_{\mu_{\min} \leq \frac{1}{4\eta^2}} \sqrt{\frac{1}{2}(1+\sqrt{1-4\eta^2 \mu_{\min}})}\; , {\bf{1}}_{\mu_{\max} \geq \frac{1}{4\eta^2}} \sqrt{2 \eta^2\mu_{\max} + \eta \sqrt{\mu_{\max}} \sqrt{4\eta^2 \mu_{\max}-1}}\right\}.$$
  
 2) If $\eta<\frac{1}{2\sqrt{\mu_{\max}}}$, the convergence is exponential and  for all $t\geq 0$,
\[ \|(x_t,y_t)-(x_\infty,y_\infty)\|\leq  C \; D\;  \lambda_{\rm max}^t  \;\]
with $\lambda_{\max}=\sqrt{\frac{1}{2}(1+\sqrt{1-4\eta^2\mu_{\min}})}$,  and $C=\sqrt{\frac{2}{1-\sqrt{\frac{1+5\eta^2\mu_{\max}}{2+ \eta^2\mu_{\max}}}}}
$. \\

3) If $\eta\geq \frac{1}{2\sqrt{\mu_{\max}}}$ we extend the definition of $\lambda _{\max}$ to  $\lambda _{\max}= \max\{\lambda_*, \lambda_{**}\}$, with
 \[\lambda_* = {\bf{1}}_{\mu_{\min} \leq \frac{1}{4\eta^2}} \sqrt{\frac{1}{2}(1+\sqrt{1-4\eta^2 \mu_{\min}})}\; \text{ and } \lambda_{**} =  {\bf{1}}_{\mu_{\max} \geq \frac{1}{4\eta^2}} \sqrt{2 \eta^2\mu_{\max} + \eta \sqrt{\mu_{\max}} \sqrt{4\eta^2 \mu_{\max}-1}}.\]
% $\eta\geq \frac{1}{2\sqrt{\mu_{\max}}}$  and 

\noindent 3a) If  $\frac{1}{4\eta^2}\notin \S(A)$, the convergence is exponential with ratio $\lambda _{\max}$ and  for all $t\geq 0$, we have
$\|(x_t,y_t)-(x_\infty,y_\infty)\|\leq  C \;  D\;  \lambda_{\rm max}^t  ,$
$${\it with}\;\; C=\max\left\{{\bf{1}}_{\mu_{\min} < \frac{1}{4\eta^2}}  \sqrt{\frac{2}{1-\sqrt{\frac{1+5\eta^2\mu_{*}}{2+ \eta^2\mu_{*}}}}},{\bf{1}}_{\mu_{\max} > \frac{1}{4\eta^2}} \sqrt{\frac{2}{1-\sqrt{\frac{2+\eta^2\mu_{**}}{1+ 5\eta^2\mu_{**}}}}}\right\},$$
  $\mu_{*}=\max\{\mu \in \S(A), \eta \sqrt{\mu}<1/2\}$ and $\mu_{**}=\min\{\mu \in \S(A), \eta \sqrt{\mu}>1/2\}.$

\noindent 3b)  If $\eta\geq \frac{1}{2\sqrt{\mu_{\max}}}$  and $\frac{1}{4\eta^2}\in \S(A)$, then for every $\lambda>\lambda_{\max}$ there exists a constant $C'$ such that for all $t\geq 0$, $ \|(x_t,y_t)-(x_\infty,y_\infty)\|\leq  C' D \lambda^t.$
 \end{thm}

The proof of Theorem \ref{thm_conv_OGDA_xAy} is in section \ref{subsec_proof_thm_conv_OGDA_xAy} of the Appendix. The most delicate aspects  are the control of the angles between  eigenvectors of the matrix $\Lambda$  for  parts {\it 2)} and {\it 3a)}, and the analysis of the Jordan normal form of $\Lambda$  for part {\it 3b)}. 

\subsection{Comments on Theorem \ref{thm_conv_OGDA_xAy}}
\label{subsec_comments}

Several remarks are in order. \\

\begin{enumerate}[label=\alph*),wide]

\item Theorem \ref{thm_conv_OGDA_xAy} also holds for $A=0$ if we use the convention $\lambda_{\max}=0$ in this case.

  \item Both  $\lambda_*$ and $\lambda_{**}$ are in $[0,1)$ and $\lambda_{\max}\in [\frac{\sqrt{2}}{2}, 1)$. Part 3a) generalizes part 2) of the Theorem, since if $\eta<\frac{1}{2 \, \sqrt{\mu_{\max}}}$  we have ${\bf{1}}_{\mu_{\min} < \frac{1}{4\eta^2}}=1$,  ${\bf{1}}_{\mu_{\max} \geq \frac{1}{4\eta^2}}=0$ and $\mu_{*}=\mu_{\max}$.

% \item  Notice that if $\eta$ i small enough so that $\eta<\frac{1}{\sqrt{4\mu_{\max}}}$, then $1/(4\eta^2)> \mu_{\max}\geq \mu_{\min}$ and  $\lambda_{\max}=\sqrt{\frac{1}{2}(1+\sqrt{1-4\eta^2\mu_{\min}})}$.

  \item $\Ker(\Lambda-I)$ is  the set $\left\{\left(\begin{array}{c} x\\ y\\ x\\y\end{array}\right), (x,y) \;{\rm Nash}\;{\rm equilibrium}\right\}$. Since it contains $0$, $D\leq \|(x_0,y_0,x_{-1},y_{-1})\|=\|Z_0\|$.

  \item  Notice that the  speed of convergence is independent of the dimension: $\lambda_{\rm max}$, as well as $C$ and $D$ from parts 2) and 3a) of the Theorem,  only depends on $\S(A)$ and $\eta$, and not  on $n$ nor $p$.

  \item The step-size $\eta$ is fixed here, and does not evolve with time.  If $\eta<\frac{1}{2 \, \sqrt{\mu_{\max}}}$, the ratio of convergence $\lambda_{\rm max}$ is better when $\eta$ is large, so that having $\eta$ too small is not a good option for large $t$. We  elaborate on the optimal choice of $\eta$  in subsection \ref{subsec_optimal_ratio}. 

  % \item In the particular case where $\Ker(A)=\{0\}$, one  can show that  convergence to  $0$  holds as soon as $0<\eta< \frac1{ \, \sqrt{3}  \sqrt{\mu_{\max}}}$. Here all eigenvalues of $\Lambda$ have modulus strictly less than 1, so $\Lambda^t\xrightarrow[t \to \infty]{}0$ by Gelfand's formula.\\
  
  \item When looking at the expression of the limit $(x_\infty, y_\infty)$ with regard to the initialization $(x_0, y_0, x_{-1}, y_{-1})$, one should notice that $x_\infty$, somehow  surprisingly, only depends  on $x_0$, and that $y_\infty$ only depends  on $y_0$.
  Thus, $y_0$, $y_{-1}$ and $x_{-1}$ will only have an influence of the curve $(x_t)_t$, but not on its limit.

\item Our ratio $\lambda_{\max}$ is indeed optimal for any matrix $A\neq 0$.
  Indeed,  Lemma \ref{lemma_optimal_value} in the Appendix shows  that there exists a real vector $Z_0 \in \R^{n+p+n+p}$ and a constant $c>0$ such that $\norm{Z_t} > c\lambda_{\max}^t$.
  
  This shows that our result improves on the existing literature.
  We still do a tour below (see g), h), i) and j)) of the main related  papers, to see how Theorem \ref{thm_conv_OGDA_xAy} compare with their results.
  This tour is summarized in Table \ref{table_convergence}, where we only mention our results related to the case where $\eta <\frac{1}{2\sqrt{\mu_{\max}}}$.

  \begin{table}[h]
  \hspace*{-2cm}
  \begin{tabular}{|c|c|c|c|c|}
  \hline
  Reference & Convergence Speed & Exponential Parameter & Constant & Learning Rate \\
  \hline
  \cite{daskalakis2017training} & Polynomial &  \ding{55} & \ding{55} & $\eta < \frac{\mu_{\min}^2}{3}$ \\
  \hline
  \cite{liang2019interaction} & Exponential & $\exp(\frac{-\eta^2\mu_{\min}}{2})$ & 4 & $\eta = \frac{1}{\sqrt{8 \mu_{\max}}}$ \\
  \hline
  \cite{peng2020training} & Exponential & $\sqrt{\frac{1}{2}\left(1+\sqrt{1-\eta^2 \mu_{\min}} \right)}$ & Unbounded & $\eta \leq \frac{1}{2\sqrt{\mu_{\max}}}$ \\
  \hline
  \cite{mokhtari2020unified} & Exponential & $\sqrt{1 -\frac{\mu_{\min}\eta^2}{4}}$ & 1 & $\eta = \frac{1}{\sqrt{40 \mu_{\max}}}$ \\
  \hline
  This paper & Exponential & $\sqrt{\frac{1}{2}\left(1+\sqrt{1-4\eta^2 \mu_{\min}} \right)}$ & $C=\sqrt{\frac{2}{1-\sqrt{\frac{1+5\eta^2\mu_{\max}}{2+ \eta^2\mu_{\max}}}}}$ & $\eta <\frac{1}{2\sqrt{\mu_{\max}}}$ \\
  \hline
  \end{tabular}
  \caption{Comparison table of the convergence rate}
  \label{table_convergence}
\end{table}

  \item \cite{daskalakis2017training} was the first paper to study last iterate convergence of OGDA in saddle point problems. 
  They consider the case of a matrix $A$ with $\mu_{\max}\leq 1$ and $0<\eta< \frac{1}{3} \mu_{\min}^2$, where $\mu_{\min}>0$ is  the smallest positive eigenvalue of $A^T A$.
  Assuming $x_{-1}=x_0$ and $y_{-1}=y_0$, and introducing  $\Delta_t=\|A^Tx_t\|^2+\|Ay_t\|^2$, they show that for $t\geq 2$, 
  \begin{equation}
    \label{eq_bound_daskalakis}
    \Delta_t\leq (1-\eta^2 \mu_{\min}^2) \Delta_{t-1}+16 \eta^3 \Delta_0,
  \end{equation}
  This does not prove convergence of $(\Delta_t)_t$ to 0 but shows that $\Delta_t$ is small for  large $t$ and small $\eta$.  \\
  Let us fix an $\epsilon>0$, and let $\eta$ be at the utmost of order $\epsilon^2$.\\
  Denote by $T$ the number of epochs needed for $\norm{(x_t, y_t)}$ to be smaller than $\epsilon$. Then:
  \begin{equation}
    \frac{16 \epsilon^2 \Delta_0}{\mu_{\min}^2}+ {(1-\epsilon^4 \mu_{\min}^2)}^{T-2}(\Delta_2-\frac{16 \epsilon^2 \Delta_0}{\mu_{\min}^2}) \leq \epsilon \iff T \geq \alpha \epsilon^{-4} \log \left(\frac{1}{\epsilon}\right)
  \end{equation}
  for some constant $\alpha$.
  Due to the term in $\epsilon^{-4}$, they have a polynomial convergence, which is worth than the actual exponential convergence.

  \item Let us now compare our result to the more recent work \cite{liang2019interaction} of Liang and Stokes.
  Theorem \ref{thm_conv_OGDA_xAy} holds for any initialization $(x_0,y_0,x_{-1},y_{-1})$, for any matrix   and for any small enough value of $\eta$,   not only for a particular value of $\eta$  and  squared invertible matrices, as in the paper of Liang and Stokes. % TODO their proof should work for any \eta smaller than there bound, and not for a precise eta.
  Consider for instance the particular case where $\eta^2\mu_{\max}=1/8$, studied in  \cite{liang2019interaction} for  the particular case of $A$ square invertible.
  We still denote by $\mu_{\min}>0$ the smallest eigenvalue of $A^T A$.
  Liang and Stokes obtained:
  \begin{equation}
    \label{eq_bound_liang}
    \|(x_t,y_t)-(0,0)\|\leq  4 \sqrt{2}\; \lambda_L^t\; \max\{\|(x_0,y_0)\|, \|(x_1,y_1)\|\},
  \end{equation}
  with \[\lambda_L={\rm exp} \left(-\frac{\mu_{\min}}{16 \mu_{\max}}\right).\] 

  Since $\|(x_0, y_0,x_{-1},y_{-1})\|\leq \sqrt{2} \max\{\|(x_0,y_0)\|, \|(x_{-1},y_{-1}\|\}$, the constants in the above inequality (\ref{eq_bound_liang}) is just slightly bigger than the constant $C_\eta$ we find for $\eta=\frac{1}{\sqrt{8\mu_{\max}}}$.
  
  Moreover, by comparing the derivative of the functions $ f(x) = \sqrt{\frac{1}{2}(1+\sqrt{1-x/2})}$ and $g(x) = \exp\left(-\frac{x}{16}\right)$, we can show that $\lambda_{\max} = f \left(\frac{\mu_{\min}}{\mu_{\max}}\right)\leq \lambda_L = g \left(\frac{\mu_{\min}}{\mu_{\max}}\right)$ for all $\frac{\mu_{\min}}{\mu_{\max}} \in [0, 1]$, showing that in addition to  generalizing the convergence results of Liang and Stokes to all matrices, we also improve  their bound.\\

  \item Another exponential bound can be seen in \cite{peng2020training}. They showed that for any learning rate $\eta$,
  \[ \norm{(x_t, y_t) - (x_\infty, y_\infty)} \leq C_P \lambda_P^T \]
  where $\lambda_P = \sqrt{\frac{1}{2}+\frac{1}{2}\sqrt{1-\mu_{\min}\eta^2}}$.
  The  value of $\lambda_P$ is still higher than $\lambda_{\max}$ for any value of $\eta$ and of $A$, and their is  no control on the  constant $C_P$.

  \item In  \cite{mokhtari2020unified}, Mokhtari \textit{et al.} showed once again the convergence of OGDA for square invertible matrices and for a particular value of $\eta$.
For  $\eta$ such that $\eta^2 \mu_{\max} = 40$, they show:
   \[ \norm{(x_t, y_t) - (x_\infty, y_\infty)} \leq \lambda_M^t \hat{r}_0 \]
  where $\lambda_M = \sqrt{1-\frac{\mu_{\min}}{6400 \mu_{\max}}}$.
  Once again, it can be shown that for any value of $\frac{\mu_{\min}}{\mu_{\max}}$, $\lambda_{\max} \leq \lambda_M$.

% \item Here, we finally illustrate a part of  Theorem \ref{thm1}: the expression of the limit value of OGDA with respect to the initialization.
% To show this, we plotted several runs of OGDA on the game with matrix $A=\begin{pmatrix} 1 & -1\\ -1 & 1 \end{pmatrix}$.
% They can be seen on Figure \ref{plot_ogda} below.
% On the left image, we plotted the first coordinate of $x_t$ with the first coordinate of $y_t$, and on the right image, we plotted the second coordinate of $x_t$ with the second coordinate of $y_t$.
% As was proved in Theorem \ref{thm1}, we can see that in  each case, OGDA converges to the orthogonal projection of the initialization vector $(x_0, y_0)$ onto $\Ker(A^T)\times \Ker(A)$.\\
% Each  target point $(x_\infty, y_\infty)$ verifies $x_{\infty, 1} = x_{\infty, 2}$ and $y_{\infty, 1} = y_{\infty,2}$, because $x_\infty$ and $y_\infty$ are in the kernel of $A$.
% This can be found back in the plots: the targets, designated by the cross, are at the same place on the left and on the right figure.\\

% \begin{figure}[h]
% \begin{subfigure}{0.5\textwidth}
% \includegraphics[width = \textwidth]{figures/plots_ogda_xaxis.png}
%   \caption{$x_{t,1}$ with respect to $y_{t,1}$}
% \end{subfigure}
% \begin{subfigure}{0.5\textwidth}
% \includegraphics[width = \textwidth]{figures/plots_ogda_yaxis.png}
% \caption{$x_{t,2}$ with respect to $y_{t,2}$}
% \end{subfigure}\\
% \caption{Several Plots of OGDA}
% \label{plot_ogda}
% \end{figure}

  \item We finally illustrate a part of Theorem \ref{thm_conv_OGDA_xAy}: the expression of the limit value of OGDA with respect to the initialization.
  Here, we are interested in the independence of the target vector $(x_\infty, y_\infty)$ with regard to the vectors $(x_{-1},y_{-1})$.
  For this, we plotted in Figure \ref{fig_plot_init} several runs of OGDA for the game with matrix $A=\begin{pmatrix} 1 & -1\\ -1 & 1 \end{pmatrix}$.
  On the left image, we plotted the first coordinate of $x_t$ with the first coordinate of $y_t$, and on the right image, we plotted the second coordinate of $x_t$ with the second coordinate of $y_t$.
  For each of the run, the same vector $(x_0, y_0)$ was chosen, and different values were taken for $x_{-1}$ and $y_{-1}$.
  We can see that this change of value changes the dynamic of the system mostly at the beginning, but that all of the experiments still converge to the same target point, which only depends  on the vector $(x_0, y_0)$, as one can see from Theorem \ref{thm_conv_OGDA_xAy}. 
  
  \begin{figure}[h]
  \begin{subfigure}{0.5\textwidth}
  \includegraphics[width = \textwidth]{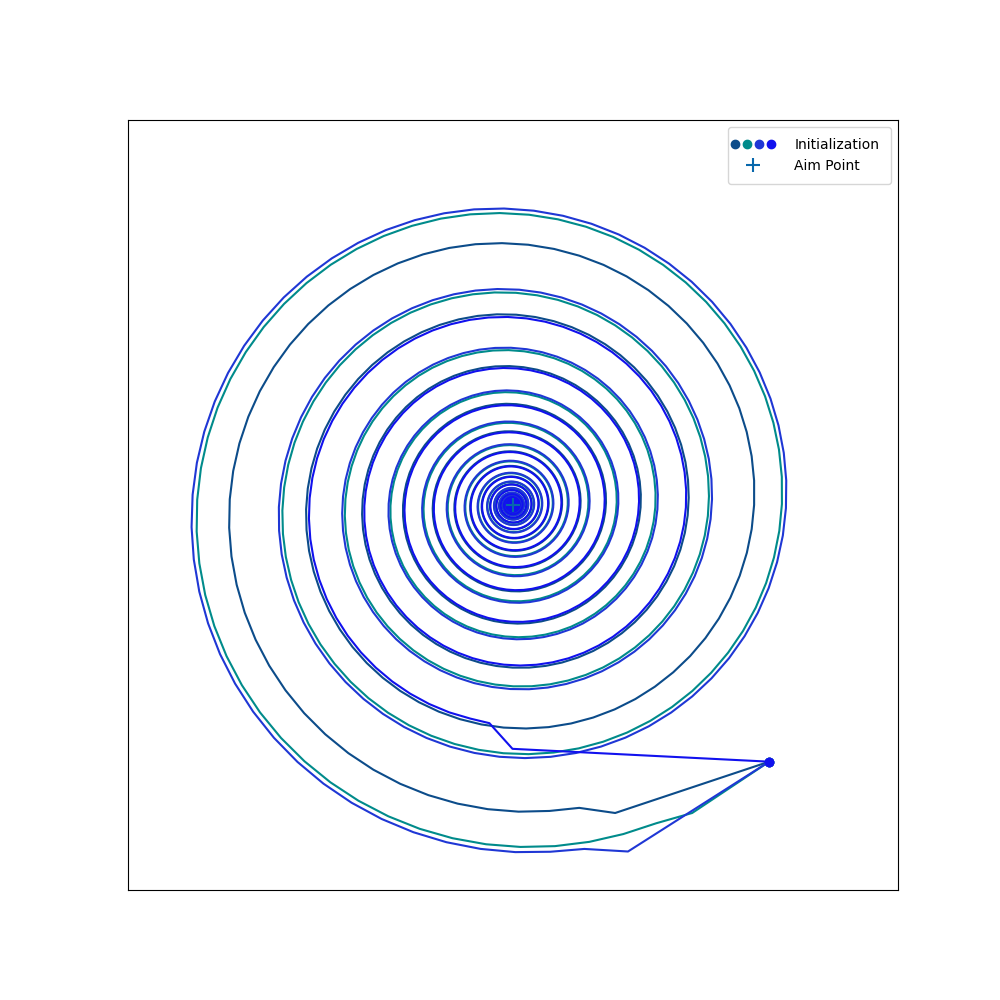}
  \caption{$x_{t,1}$ with respect to $y_{t,1}$}
  \end{subfigure}
  \begin{subfigure}{0.5\textwidth}
  \includegraphics[width = \textwidth]{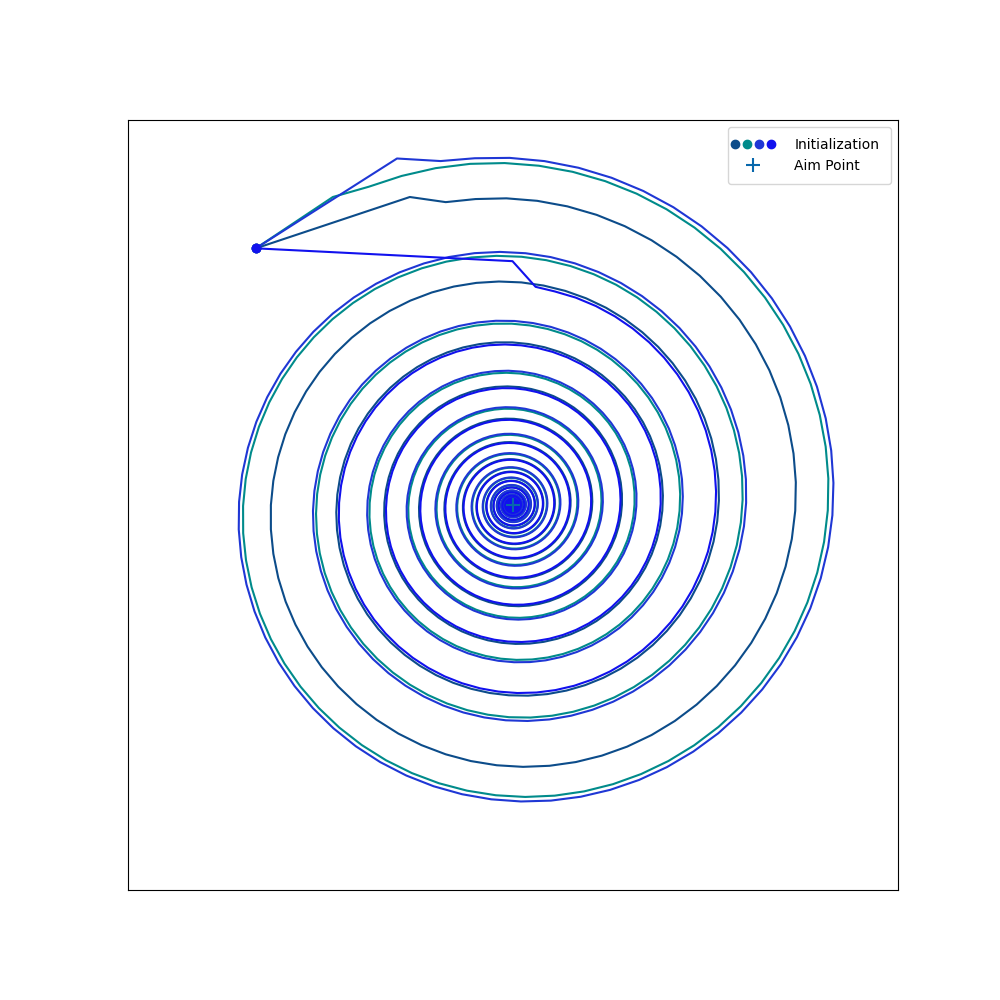}
  \caption{$x_{t,2}$ with respect to $y_{t,2}$}
  \end{subfigure}\\
  \caption{Several Plots of OGDA}
  \label{fig_plot_init}
  \end{figure}

\end{enumerate}

\subsection {Zero-sum Games $x^TAy+b^Tx+c^Ty+d$}
\label{subsec_general_zs}

Let us now consider the general bilinear  case, where we add to the  term $x^TAy$, a linear function in $(x,y)$. Fix  $A\in \R^{n \times p}$,  $b \in \R^n$, $c\in \R^p$  and $d \in \R$.
We look at  the zero-sum game where simultaneously player 1 chooses $x$ in $\R^n$, player 2 chooses $y$ in $\R^p$, and finally player 2 pays the quantity  $g(x,y)=x^TAy+b^Tx+c^Ty+d$  to player 1.
The Nash equilibria are the vectors $(x,y)$ such that $g(x,y)=\sup_{x'\in \R^n} g(x',y)=\inf_{y'\in \R^p}g(x,y')$, i.e.  such that $A^Tx+c = 0$ and $Ay+b=0$.  Here, there exists a Nash equilibrium if and only if $-b$ is in the range of $A$ and $-c$ is in the range of $A^T$. \\

The next result is the extension of Theorem \ref{thm_conv_OGDA_xAy} to this setup. The proof is in  Appendix \ref{subsec_proof_thm_conv_OGDA_zs_general}.

\begin{thm}
  \label{thm_conv_OGDA_zs_general}
  Let $A \in \R^{n\times p}\backslash\{0\}$, $b \in \R^n$, $c \in \R^p$, $d \in \R$ and  $\eta>0$. We write $\mu_{\max}=\max\{\mu\in \R, \mu \in \Sp(AA^T)\}$, 
  $\mu_{\min}=\min\{\mu>0, \mu \in \Sp(AA^T)\}$, and  consider the  Optimistic Gradient Descent Ascent:
  \begin{equation}
    \label{OGDAgeneralbili}
    \forall t\geq 0,\;  \left\{
      \begin{array} {ccc} 
        x_{t+1} & = & x_t \;+\; 2\eta (A y_t+b) - \eta (A y_{t-1}+b),  \\
        y_{t+1} & = & y_t - 2\eta (A^T x_t+c) + \eta (A^T x_{t-1}+c).
      \end{array}\right.
  \end{equation}

If the set ${\{(x,y) | A^Tx+c = 0, Ay+b=0\}}$ of Nash equilibria is empty, then $(x_t, y_t)_t$ diverges. 
If $\eta>\frac{1}{\sqrt{3\mu_{\max}}}$, there exist  initial conditions where the corresponding OGDA does not converge. In the sequel we assume that the set of Nash equilibria is not empty and that $\eta<\frac{1}{\sqrt{3\mu_{\max}}}.$

 Fix any  initial condition $(x_0, y_0,x_{-1},y_{-1})\in \R^n\times \R^p\times \R^n \times \R^p$, and let $D$ be the distance from $(x_0,  y_0, x_{-1}, y_{-1})$ to the set $\{(x,y,x,y)\in \R^n\times \R^p\times \R^n\times \R^p, A^Tx+c=0, Ay+b=0\}$. We have the following: \\

1) $(x_t, y_t)_t$ converges to the Nash equilibrium $\left(x_\infty, y_\infty\right)$, where  $x_\infty$ is the orthogonal projection of $x_0$ onto $\{x\in \R^n | A^Tx+c=0\}$, and $y_\infty$ is the orthogonal projection of $y_0$ onto $\{y \in \R^p | Ay+b=0\}$.\\

 2) If $\eta<\frac{1}{2\sqrt{\mu_{\max}}}$, the convergence is exponential and  for all $t\geq 0$,
\[ \|(x_t,y_t)-(x_\infty,y_\infty)\|\leq  C \; D\;  \lambda_{\rm max}^t  \;\]
with $\lambda_{\max}=\sqrt{\frac{1}{2}(1+\sqrt{1-4\eta^2\mu_{\min}})}$, and $C=\sqrt{\frac{2}{1-\sqrt{\frac{1+5\eta^2\mu_{\max}}{2+ \eta^2\mu_{\max}}}}}
$. \\

3) If $\eta\geq \frac{1}{2\sqrt{\mu_{\max}}}$ we extend the definition of $\lambda _{\max}$ to  $\lambda _{\max}= \max\{ \lambda_*, \lambda_{**}\}$, with
 \[\lambda_* = {\bf{1}}_{\mu_{\min} \leq \frac{1}{4\eta^2}} \sqrt{\frac{1}{2}(1+\sqrt{1-4\eta^2 \mu_{\min}})}\; \text{ and } \lambda_{**} = {\bf{1}}_{\mu_{\max} \geq \frac{1}{4\eta^2}} \sqrt{2 \eta^2\mu_{\max} + \eta \sqrt{\mu_{\max}} \sqrt{4\eta^2 \mu_{\max}-1}}.\]

\noindent 3a) If $\eta\geq \frac{1}{2\sqrt{\mu_{\max}}}$  and $\frac{1}{4\eta^2}\notin \S(A)$, the convergence is exponential with ratio $\lambda _{\max}$ and  for all $t\geq 0$,
$ \|(x_t,y_t)-(x_\infty,y_\infty)\|\leq  C \;  D\;  \lambda_{\rm max}^t  ,\;$
 $${\it with}\;\; C=\max\left\{{\bf{1}}_{\mu_{\min} < \frac{1}{4\eta^2}}  \sqrt{\frac{2}{1-\sqrt{\frac{1+5\eta^2\mu_{*}}{2+ \eta^2\mu_{*}}}}},{\bf{1}}_{\mu_{\max} \geq \frac{1}{4\eta^2}} \sqrt{\frac{2}{1-\sqrt{\frac{2+\eta^2\mu_{**}}{1+ 5\eta^2\mu_{**}}}}}\right\},$$
$\mu_{*}=\max\{\mu \in \S(A), \eta \sqrt{\mu}<1/2\}$ and $\mu_{**}=\min\{\mu \in \S(A), \eta \sqrt{\mu}>1/2\}.$

\noindent 3b)  If $\eta\geq \frac{1}{2\sqrt{\mu_{\max}}}$  and $\frac{1}{4\eta^2}\in \S(A)$, then for every $\lambda>\lambda_{\max}$ there exists a constant $C'$ such that for all $t\geq 0$, $ \|(x_t,y_t)-(x_\infty,y_\infty)\|\leq  C' D \lambda^t.$
\end{thm}

\subsection{Optimal geometric ratio of convergence}
\label{subsec_optimal_ratio}
Given a matrix $A\neq 0$, one can ask which $\eta\in (0, \frac{1}{\sqrt{3\mu_{\max}}})$ minimizes the geometric  ratio of convergence $\lambda_{\max}=$ 
  $$\max\left\{{\bf{1}}_{\mu_{\min} \leq \frac{1}{4\eta^2}} \sqrt{\frac{1}{2}(1+\sqrt{1-4\eta^2 \mu_{\min}})}\; , {\bf{1}}_{\mu_{\max} \geq \frac{1}{4\eta^2}} \sqrt{2 \eta^2\mu_{\max} + \eta \sqrt{\mu_{\max}} \sqrt{4\eta^2 \mu_{\max}-1}}\right\}.$$
  Let us introduce the variable  $\alpha=\frac{\mu_{\min}}{\mu_{\max}}\in (0,1]$. 
 Computations show that the optimal values for $\eta$ and $\lambda_{\max}$ are:
 
 $$\eta^*\sqrt{\mu_{\max}}=\sqrt{\frac{1}{32 \alpha}\left(3+6 \alpha-\alpha^2- (1-\alpha)\sqrt{(1-\alpha)(9-\alpha)}\right)}\in [\frac12, \frac{1}{2\sqrt{\alpha}}],$$
  
% $$4 {\eta^*}^2 \mu_{\max}=\frac{1}{8 \alpha}\left(3+6 \alpha-\alpha^2- (1-\alpha)\sqrt{(1-\alpha)(9-\alpha)}\right)\in [1,  \frac{1}{\alpha}],$$
  
  $$\lambda_{\max}^*=\sqrt{\frac{1}{2} \left( 1+ \sqrt{1-\frac{1}{8} \left(3+6 \alpha-\alpha^2 -(1-\alpha)\sqrt{(1-\alpha)(9-\alpha)}\right)}\right)}.$$
  
  For instance, if $A= \left(\begin{array}{ccc} 0& 0 & 0 \\ 0& 1 &0\\ 0&0&2\end{array}\right)$, $\mu_{\min}=1$, $\mu_{\max}=4$, $\eta^*\sqrt{\mu_{\max}}\simeq 0.5608$ and $\lambda_{\max}^*\simeq 0.956.$

When $\alpha$ varies from 0 to 1,  $\eta^*\sqrt{\mu_{\max}}$ decreases  from $\frac{1}{\sqrt 3}\simeq 0.577$ to $1/2$ and $\lambda_{\max}^*$ decreases from 1 to $\frac{\sqrt{2}}{2}$. We always have $\eta^*\sqrt{\mu_{\max}}\geq 1/2$, justifying our interest for the cases where $\eta \geq \frac{1}{2\sqrt{\mu_{\max}}}$ in Theorems \ref{thm_conv_OGDA_xAy} and \ref{thm_conv_OGDA_zs_general}.

Convergence is faster when $\alpha$ increases, and is optimal when $\alpha=1$, i.e. when $AA^T$ has a unique nonzero eigenvalue. It is for instance the case  when $A$ is (up to a multiplicative constant)  an orthogonal matrix. 
 
%  When $\alpha=1$, $\eta^*$ is simply $\frac{1}{2 \sqrt{\mu_{\max}}}$. 

\section{General-Sum Bilinear Games}
\label{sec_nzs}

We extend the analysis of section \ref{section_bilinear} to the class of   general-sum bilinear games. For simplicity we will only consider here values of $\eta$ corresponding to case {\it 2)}  of Theorems \ref{thm_conv_OGDA_xAy} and \ref{thm_conv_OGDA_zs_general}.

%$(x^TAy+b^Tx+c^Ty+d,x^TBy+e^Tx+f^Ty+g)$
\subsection{General-Sum Games $(x^TAy,x^TBy)$}
\label{subsec_xAy_xBy}

We consider here the case of a non zero-sum game given by two matrices $A$ and $B$ in $\R^{n\times p}$. Simultaneously player 1 chooses $x$ in $\R^n$ and  player 2 chooses $y$ in $\R^p$, then the payoff to player 1 is $x^TAy$ and the payoff to player 2 is $x^TBy$. By definition, $(x,y)\in \R^n \times \R^p$ is a  Nash equilibrium of the game if:
\[x^TAy = \max_{x'\in \R^n} x'^TAy \; \; {\rm and}\;\;  x^TBy= \max_{y'\in \R^p} x^TBy'.\]
The set of Nash equilibria is $\Ker(B^T)\times \Ker(A)$, which always contains $(0,0)$. In section \ref{section_bilinear} we have considered the case where $B=-A$. \\

We now introduce the OGDA for general-sum games:
\begin{defi}  The Optimistic Gradient Descent Ascent algorithm for the general-sum game $(x^TAy$, $x^TBy)$ is defined by:
  \begin{equation}
  \label{OGDA_nzs}
  \forall t\geq 0,\left\{
  \begin{array}{ccccccc}
  x_{t+1} &= & x_t &+& 2\eta A y_t &-& \eta A y_{t-1},\\
  y_{t+1}& = & y_t &+& 2\eta B^T x_t &-& \eta B^T x_{t-1},
  \end{array}\right.
  \end{equation}
where $\eta>0$ is a fixed parameter, and $(x_0,y_0,x_{-1},y_{-1})\in \R^{n}\times \R^p\times \R^n \times \R^p$ is the initialization.\end{defi}
Notice that if $(x_t,y_t)_t$ converges, the limit is a Nash equilibrium. We first observe  that there exist  general-sum games $(x^TAy,x^TBy)$ for which OGDA does not converge. 
A simple  case is when $A=B=(1)$. Assume for simplicity that $x_0=y_0$ and $x_{-1}=y_{-1}$. Then for all $t$,  $x_t=y_t$ and \[x_{t+1}=x_t(1+2\eta)-\eta x_{t-1}.\]
This can be written $x_{t+1}-x_t=\eta  x_t+\eta (x_t-x_{t-1})$, so that if $x_0>\max\{x_{-1},0\}$ then $x_{t+1}>\max\{x_{t},0\}$ for all $t$. It implies $x_{t+1}> x_t(1+\eta)$ for all $t$, so  $x_t\xrightarrow[t \to \infty]{}+\infty$.
We have  obtained:

\begin{pro}
  Convergence of  OGDA  may fail for general-sum games $(x^TAy,x^TBy)$.
  In other words, for any $\eta>0$, there exist some initialization $(x_0, y_0, x_{-1}, y_{-1})$ such that $(x_t, y_t)_t$ does not converge.
\end{pro}

 We  now present  a simple trick, a modification of OGDA which  ensures convergence to Nash equilibria in general-sum games when $\eta$ is small enough.
 
 The set of Nash equilibria of the game $(x^TAy$, $x^TBy)$ is $\Ker(B^T)\times \Ker(A)$.  The set  $\Ker(B^T)$ is the set of optimal strategies of player 1 in the zero-sum game $x^T(-B)y$ (where player 1 wants to minimize the payoff $x^TBy$) and $\Ker(A)$ is the set of optimal strategies of player 2  in the zero-sum game $x^TAy$ (where player 2 wants to minimize  $x^TAy$). As a consequence, the convergence of OGDA for zero-sum games gives here a simple algorithm converging to a Nash equilibrium of a general-sum game. 
 
 \begin{defi}
   \label{def_DOGDA}
   The Double Optimistic Gradient Descent Ascent (DOGDA) for the general-sum game $(x^TAy$, $x^TBy)$ is defined by:
   \begin{equation}
     \label{eq_DOGDA}
     \left\{
       \begin{array} {ccccc} 
         x_{t+1}& = & x_t &- 2 \eta B y'_t  &+ \eta B y'_{t-1},\\
         y'_{t+1} & = & y'_t &+2 \eta B^T x_t &- \eta B^T x_{t-1},\\
         x'_{t+1}& = & x'_t &+2 \eta A y_t  & - \eta A y_{t-1},\\
         y_{t+1} & = & y_t &-2 \eta A^T x'_t & +  \eta A^T x'_{t-1}.
       \end{array}\right.
   \end{equation}
   where $\eta>0$ is a fixed parameter, and $z_0:=(x_0,y'_0, x'_0, y_0, x_{-1},y'_{-1}, x'_{-1}, y_{-1} )\in {(\R^{n}\times \R^p)^4}$ is the initialization.
 \end{defi}

\begin{pro}
  \label{thm_conv_DOGDA}
  $\;$ {\bf Convergence  of DOGDA to a Nash Equilibrium}.

Let $A, B\in \R^{n\times p}$  and $0<\eta<\frac{1}{2 \, \sqrt{\mu_{\max}}}$, where $\mu_{\max}=\max\{\rho(A^TA), \rho(B^TB)\}$. Consider $z_0\in {(\R^{n}\times \R^p)^4}$ and  the sequence  $(x_t,y'_t,x'_t,y_t)$ induced by the DOGDA algorithm {\rm (\ref{eq_DOGDA})} with initialization $z_0$. 

Then $(x_t, y_t)$ converges to a Nash equilibrium $(x_\infty,y_\infty)$, $x_\infty$ is the orthogonal projection of $x_0$ on $\Ker(B^T)$ and $y_\infty$ is the orthogonal projection of $y_0$ on $\Ker(A)$. Moreover the convergence is exponential:  for all $t\geq 0$,
\[ \|(x_t,y_t)-(x_\infty,y_\infty)\|\leq  C \;  \lambda_{\rm max}^t \; \|z_0\|,\; \text{where}\; 
 C=\sqrt{2} \left(1-\sqrt{\frac{1+5\eta^2\mu_{\max}}{2+ \eta^2\mu_{\max}}}\right)^{-1/2}\] 
and $\lambda_{\max} =  \sqrt{\frac{1}{2}(1+\sqrt{1-4\eta^2\mu'_{\min}})},$
  with  $\mu'_{\min}=\min\{\ \mu>0, \mu \in \S(A) \cup \S(B)\}$
 and  the convention $\lambda_{\max}=0\; {\it  if}\; A=B=0$.
\end{pro}

\begin{proof} Directly comes from Theorem \ref{thm_conv_OGDA_xAy}, since DOGDA is the juxtaposition of 2 uncoupled OGDA dynamics.
\end{proof}

\begin{rem}  \rm One can replace $B$ with $-B$ and/or $A$ with $-A$ in the OGDA or DOGDA algorithms. This is a strong specificity of our bilinear games. We have chosen  (\ref{eq_DOGDA}) in definition \ref{def_DOGDA} because of the following remark \ref{rem_DOGDA_trick}. 
\end{rem}
 
\begin{rem}
  \label{rem_DOGDA_trick}
  \rm The DOGDA trick is due to the following property of our bilinear games: if $g_1$ denotes the payoff function $x^TAy$ for player 1 and $g_2$ is the payoff function $x^TBy$ for player 2, the Nash equilibria of the game $(g_1,g_2)$  are the couples $(x,y)$ where $x$ is optimal for player 1 in the zero-sum game $(-g_2,g_2)$ and $y$ is optimal for player 2 in the zero-sum game $(g_1,-g_1)$. This property will also hold for the  general-sum bilinear games of section \ref{subsec_extension_general}, so the DOGDA trick also applies there.  One can ask how often this property holds in general ?  In the case of general-sum  games with 2 players and 2 actions for each player played with mixed strategies, assuming that  the 8 real coefficients are a priori independently selected according to the uniform law on a given  compact interval with positive length, one can show that the probability to obtain a game where this property holds is $2/9\simeq 22\%$.
\end{rem}
 
\subsection{Sufficient conditions for convergence in games $(x^TAy,x^TBy)$}
\label{subsec_sufficient_xAy_xBy}

We come back to the analysis of the OGDA algorithm (\ref{OGDA_nzs}) and  partly  extend  the analysis of section \ref{section_bilinear} to the non zero-sum case. Define here the matrix:
 \[\Lambda_{A,B}= \begin{pmatrix} I_n & 2\eta A & 0 & -\eta A\\2\eta B^T & I_p & -\eta B^T & 0\\ I_n & 0 & 0 & 0\\ 0 & I_p & 0 & 0  \end{pmatrix} \in \R^{(n+p+n+p) \times (n+p+n+p)}.\]
For $t\geq 0$, let  again $Z_t$ be the column vector ${Z_t =\left(\begin{array}{c} x_t\\ y_t\\ x_{t-1}\\y_{t-1}\end{array}\right) \in \R^n\times \R^p\times \R^n\times \R^p}$. We have $\forall t\geq 0,  Z_{t+1} = \Lambda_{A,B}  Z_t$.

Define now for each $\mu \in \C$,  $$S^*(\mu)=\{ \lambda \in \C, \lambda^2(1-\lambda)^2 =\mu \eta^2(1-2\lambda)^2\},$$  and recall the notation $\S(A,B)=\Sp(B^TA) \cup \Sp(AB^T)$. One can show (see Appendix \ref{subsec_proof_prop_spectrum_xAy_xBy})

\begin{pro}
  \label{prop_spectrum_xAy_xBy}
 $\Sp(\Lambda_{A,B}) = \bigcup_{\mu \in \S(A,B)} S^*(\mu).$
\end{pro}

``Unfortunately'', one can show that for each $\mu\notin \R_-$ there exists $\eta_0>0$ such that for each $0<\eta<\eta_0$ there exists $\lambda$ in $S^*(\mu)$ with $|\lambda|>1$. So if there exists $\mu$ in $\S(A,B) \backslash\R_-$, there is no hope to prove that for $\eta>0$ small enough the OGDA will converge for any initialization. This is why  we will assume $\S(A,B) \subset \R_-$ in the following  Theorem \ref{prop_conv_OGDA_xAy_xBy}, which   gives  sufficient conditions for convergence of OGDA in general-sum games. It   includes the case where $B=-lA$ for some $l> 0$  but also, as we will see later, the case where $B=-(A^\dagger)^T$, where $A^\dagger\in \R^{p\times n}$ is the generalized Moore-Penrose inverse of $A$.

% OK pourtant The  proof of Proposition \ref{prop_spectrum_xAy_xBy} also provides useful expressions for the eigenspaces $E_\lambda=\Ker(\Lambda_{A,B}-\lambda I_{2(n+p)})$:  

%  For  $\lambda \neq 0, 1, $
% \begin{align*}
%  E_\lambda &= \left\{(x, y, x', y')\in \C^{n+p+n+p} | x=\lambda x', y = \lambda y', B^TA y' = \mu y', x' = \frac{(1-2\lambda)\eta A y'}{\lambda(1-\lambda)} \right\}\\
%   &= \left\{(x, y, x', y')\in \C^{n+p+n+p} | x=\lambda x', y = \lambda y', AB^T x' = \mu x', y' = \frac{(1-2\lambda)\eta B^T x'}{\lambda(1-\lambda)} \right\}
% \end{align*}
% and
% \[ E_0 = \left\{(0, 0, x', y')\in \C^{n+p+n+p} | x' \in \Ker(B^T), y' \in \Ker(A) \right\}, \]
% \[ E_1 = \left\{(x', y', x', y')\in \C^{n+p+n+p} | x' \in \Ker(B^T), y' \in \Ker(A) \right\}. \]

% \vspace{0.5cm}

\begin{thm}
  \label{prop_conv_OGDA_xAy_xBy}
  Let $A, B$ be in $\R^{n\times p}$ with $\S(A,B) \subset \R_-$. Let $\eta\in (0,\frac{1}{2 \, \sqrt{\mu_{\max}}})$, where $\mu_{\max}=\rho(B^TA) = \rho(AB^T)$ is  the largest eigenvalue of $B^TA$. Assume that ($A$ and $B$ are square invertible matrices), or that  $\Lambda_{A,B}$ is diagonalizable. \\
  
  Given  an initialization $(x_0, y_0,x_{-1},y_{-1})\in \R^n\times \R^p\times \R^n \times \R^p$, consider the  Optimistic Gradient Descent algorithm (\ref{OGDA_nzs}). Then:\\
  
  1) $(x_t, y_t)_t$ converges to a Nash equilibrium $\left(x_\infty, y_\infty\right)$.
  
  2)  $x_\infty$ is the linear projection of $x_0$ onto $\Ker(B^T)$ along $(\Ker(A^T))^\perp= \Img(A)$ and $y_\infty$ is the linear projection of $y_0$ onto $\Ker(A)$ along $(\Ker(B))^\perp= \Img(B^T)$.
  
  3) The convergence is exponential: there exists a constant $C>0$ verifying for all $t\geq 0$,
\[ \|(x_t,y_t)-(x_\infty,y_\infty)\|\leq  C \;\norm{(x_0, y_0,x_{-1},y_{-1})} \; \lambda_{\rm max}^t  ,\;\]
\begin{eqnarray*} \it{where}\;\;\; \lambda_{\max}&=&\max\{|\lambda|, \lambda \in \Sp(\Lambda_{A,B}), \lambda \neq 1\}<1,\; \\
\; &=&  \sqrt{\frac{1}{2}(1+\sqrt{1-4\eta^2\mu_{\min}})}, {\it with} \;\; \mu_{\min}=\min\{\  \mu , \mu>0, -\mu \in \S(A,B)\},
\end{eqnarray*}
 and  the convention $\lambda_{\max}=0$  if  $A=B=0$.
\end{thm}

% \begin{rem}
% Recall that $(\Ker(A^T))^\perp = \Img(A)$ and $(\Ker(B))^\perp = \Img(B^T)$.
% \end{rem}

\begin{exas}$\;$

1) Consider the case where $A=  \begin{pmatrix} 1 &  0  \\ 2 & 0\ \end{pmatrix}$ and $B= \begin{pmatrix} -2 &  -1  \\ 0 & 0\ \end{pmatrix}$. Then  $B^TA=\begin{pmatrix} -2 &  0  \\ -1 & 0\ \end{pmatrix}=B^T$ and  $AB^T=\begin{pmatrix} -2 &  0  \\ -4 & 0\ \end{pmatrix}$, so that  $\S(A,B)=\{0,-2\}\subset \R_{-}$ and $\mu_{\max}=\mu_{\min}=2$. For $\eta=0.1<\frac{1}{2 \, \sqrt{2}}$, one can check that $\Lambda_{A,B}$ is diagonalizable. Theorem \ref{prop_conv_OGDA_xAy_xBy} applies, so that for any initialization the sequence $(x_t)$, resp. $(y_t)$, converges to the projection of $x_0$, resp. $y_0$, onto the vertical axis $\{(0,z ), z  \in \R\}$ along $\{(z, 2z), z  \in \R\}$, resp. $\{(2z, z), z  \in \R\}$. The ratio of convergence is $\lambda_{max}=\sqrt{\frac12(1+\sqrt{0.92})} \simeq 0.9897$.

2) 
  Consider the case where $A= \begin{pmatrix} 1 &  1  \\ 1 & 1\ \end{pmatrix}$ and $B= \begin{pmatrix} 1 &  1  \\ -1 & -1\ \end{pmatrix}$. Then  $B^TA=\begin{pmatrix} 0 &  0  \\ 0 & 0\ \end{pmatrix}$, $AB^T=\begin{pmatrix} 2 &  -2  \\ 2 & -2\ \end{pmatrix}$ and $\S(A,B)=\{0\}\subset \R_{-}$. However one can show that $\Lambda_{A,B}$ is not diagonalizable, and $(\Lambda^t_{A,B})_t$  does not converge, so that for many initial conditions OGDA does not converge here.
\end{exas}

% \begin{rem}
% Let $A$ be the payoff matrix for a zero-sum game.
% There is a link between changing the learning rate of OGDA from $\eta$ to $\eta' = \frac{\eta}{m}$ for this game, and looking to OGDA for the  general-sum game $(A, -m^2A)$ with parameter $\eta$.
% Both algorithm will converge to the same points with the same convergence speed.\\
% Thus, if the learning rate is fixed for any reason, studying a well-modified general-sum game will improve the convergence rate.
% \end{rem}

\subsection{Extension to the general bilinear case}
\label{subsec_extension_general}
 We consider here the general case of a non zero-sum game given by matrices $A$, $B$ in $\R^{n\times p}$, vectors $b$, $e$ in $\R^n$, $c$, $f$ in $\R^p$, and constants $d$, $g$ in $\R$. Simultaneously  player 1 chooses $x$ in $\R^n$ and player 2 chooses $y$ in $\R^p$, then the payoff for player 1 is $$g_1(x,y)= x^TAy +b^Tx +c^Ty +d,$$ and the payoff for player 2 is $$g_2(x,y)=x^TBy+e^Tx+f^Ty+g.$$  In section \ref{subsec_general_zs} we have considered the case where $g_2=-g_1$. By definition, $(x,y)\in \R^n \times \R^p$ is a Nash equilibrium of the game if:
 $$g_1(x,y)=\max_{x' \in \R^n} g_1(x',y)\; {\rm and} \; g_2(x,y)=\max_{y' \in \R^p}g_2(x,y').$$ 
 Here, the set of Nash equilibria is $\{(x,y)\in \R^n \times \R^p, B^Tx+f=0, Ay+b=0\}$, and there exists a Nash equilibrium if and only if $-b$ is in the range of $A$ and $-f$ is in the range of $B^T$.

\begin{defi}  The Optimistic Gradient Descent Ascent algorithm for the general-sum game $(x^TAy +b^Tx +c^Ty +d, x^TBy+e^Tx+f^Ty+g)$ is given  by:
  \begin{equation} \forall t\geq 0,\left\{\begin{array}{ccccccc} x_{t+1} &= & x_t &+& 2\eta (A y_t+b)  &-& \eta (A y_{t-1}+b),\\ y_{t+1}& = & y_t &+& 2\eta (B^T x_t+f) &-& \eta (B^T x_{t-1}+f), \end{array}\right.
    \label{OGDA_generalbili}
  \end{equation}
where $\eta>0$ is a fixed parameter, and $(x_0,y_0,x_{-1},y_{-1})\in \R^{n}\times \R^p\times \R^n \times \R^p$ is the initialization.\end{defi}
The fixed points of the OGDA are the Nash equilibria. Notice that the  constants $d$ and $g$ play no role here, so  we set without loss of generality  $d=g=0$.  We generalize now Theorem \ref{prop_conv_OGDA_xAy_xBy}.

\begin{thm}
  \label{thm_conv_OGDA_general_xAy_xBy}
  Let $A, B$ be in $\R^{n\times p}$,  $b$, $e$ in $\R^n$ and  $c$, $f$ in $\R^p$, with $\S(A,B) \subset \R_-$. Let $\eta\in (0,\frac{1}{2 \, \sqrt{\mu_{\max}}})$, where $\mu_{\max}=\rho(B^TA) = \rho(AB^T)$ is  the largest eigenvalue of $B^TA$. Assume that  ($A$ and $B$ are square invertible matrices)  or that  $\Lambda_{A,B}$ is diagonalizable, and assume that a Nash equilibrium exists, i.e. that $\{(x,y)\in \R^n \times \R^p, B^Tx+f=0, Ay+b=0\}\neq \varnothing$. \\
  
  Given  an initialization $(x_0, y_0,x_{-1},y_{-1})\in \R^n\times \R^p\times \R^n \times \R^p$, consider the  Optimistic Gradient Descent algorithm (\ref{OGDA_generalbili}). Then:\\
  
  1) $(x_t, y_t)_t$ converges to a Nash equilibrium $\left(x_\infty, y_\infty\right)$.
  
  2)  $x_\infty$ is the linear projection of $x_0$ onto $\{x\in \R^n, B^Tx+f=0\}$  along $(\Ker(A^T))^\perp= \Img(A)$ and $y_\infty$ is the linear projection of $y_0$ onto $\{y \in \R^p, Ay+b=0\}$ along $(\Ker(B))^\perp= \Img(B^T)$.
  
  3) The convergence is exponential: there exists a constant $C>0$ verifying for all $t\geq 0$,
\[ \|(x_t,y_t)-(x_\infty,y_\infty)\|\leq  C \;\norm{(x_0, y_0,x_{-1},y_{-1})} \; \lambda_{\rm max}^t  ,\;\]
where $\lambda_{\max}= \sqrt{\frac{1}{2}(1+\sqrt{1-4\eta^2\mu_{\min}})}$, $\mu_{\min}=\min\{\  \mu >0, -\mu \in \S(A,B)\}$  and  the convention $\lambda_{\max}=0$  if  $A=B=0$.
% \begin{eqnarray*} \it{where}\;\;\; \lambda_{\max}&=&\max\{|\lambda|, \lambda \in \Sp(\Lambda_{A,B}), \lambda \neq 1\}<1,\; \\
% \; &=&  \sqrt{\frac{1}{2}(1+\sqrt{1-4\eta^2\mu_{\min}})}, {\it with} \;\; \mu_{\min}=\min\{\  \mu , \mu>0, -\mu \in \S(A,B)\},
% \end{eqnarray*}
  
\end{thm}

 \subsection{Improving the zero-sum rate  of convergence with general-sum games} \label{improvingwgsg}
 
 Consider a zero-sum game $g(x,y)=x^TAy+ b^Tx+c^Ty+d$ with $A\neq 0$, and assume we are interested in  finding an  optimal strategy $y$ for player 2.  We write as before $\mu_{\max}=\max\{\mu>0, \mu\in \Sp(A^TA)\}$ and $\mu_{\min}=\min \{\mu>0, \mu\in \Sp(A^TA)\}$.
 
%  Applying Theorem \ref{thm_conv_OGDA_xAy} (or equivalently, Theorem \ref{prop_conv_OGDA_xAy_xBy} with $B=-A$) provides the convergence of OGDA to equilibria of the zero-sum game for $\eta<\frac{1}{2\sqrt{\mu_{\max}}}$, with a geometric  rate of convergence of:  $$\lambda_{\max}= \sqrt{\frac{1}{2}(1+\sqrt{1-4\eta^2\mu_{\min}}}),$$ with $\mu_{\max}=\max\{\mu>0, \mu\in \Sp(A^TA)\}$ and $\mu_{\min}=\min \{\mu>0, \mu\in \Sp(A^TA)\}$.

 \subsubsection{Using $B=-lA$ for fixed $\eta$}
 
 Assume $\eta$ is sufficiently small, in particular smaller than $\frac{1}{2\sqrt{\mu_{\max}}}$.  Theorem \ref{thm_conv_OGDA_zs_general}  provides the convergence of OGDA to equilibria of the zero-sum game,  with a geometric  rate of convergence of:  $$\lambda_{\max}= \sqrt{\frac{1}{2}(1+\sqrt{1-4\eta^2\mu_{\min}}}),$$ 
Consider now applying  OGDA to the general sum game $(g(x,y)), -l g(x,y))$  for some parameter $l> 0$. The matrix  $\Lambda_{A,-l A}$ is diagonalizable (as in the zero-sum case, see Lemma \ref{lemA6} in the appendix), and Theorem \ref{thm_conv_OGDA_general_xAy_xBy} applies. As in the zero-sum case,  $(y_t)$ converges to the linear  projection of $y_0$ to $\{y, Ay+b=0\}$ along $(\Ker(B))^\perp$.  The geometric rate is now:  
$$\lambda_{\max}(l)=\sqrt{\frac{1}{2}(1+\sqrt{1-4l \eta^2 \mu_{\min}}}),$$
which is smaller than $\lambda_{\max}$ for $l \geq 1$. So increasing $l$ given $\eta$ fixed can improve the convergence rate. This is valid as long as $4l \eta^2 \mu_{\max}<1$, and can approach the  rate of  $\sqrt{{\frac12}(1+ \sqrt{1-\alpha})}$. 

Keeping $\eta$ fixed and increasing $l$ so that $4l \eta^2 \mu_{\max}\simeq 1$, can be seen as an alternative to  keeping $l=1$ (zero-sum game) and increasing $\eta$ so that $4l \eta^2 \mu_{\max}<1$. This may  be interesting for instance if several matrix games are considered at the same time and one wants  the  parameter $\eta$ to be the same in  each case.

\subsubsection{Using $B=-(A^\dagger)^T$ to improve the convergence rate}
\label{subsec_Adagger}

  Consider the case of $g(x,y)=x^TAy$, with  $A=  \begin{pmatrix} 1 &  0  \\ 0 & 2\ \end{pmatrix}$.  Then $\S(A)=\{1,4\}$, $\mu_{\max}=4$ and $\mu_{\min}=1$.
  By Theorem \ref{thm_conv_OGDA_xAy}, applying OGDA with $B=-A$ gives a convergence rate, when $12 \eta^2<1$, of
 $\lambda _{\max}=$ 
  $$\max\left\{{\bf{1}}_{4\eta^2 \leq 1} \sqrt{\frac{1}{2}(1+\sqrt{1-4\eta^2 })}\; , {\bf{1}}_{16\eta^2 \geq 1}  \sqrt{ 8 \eta^2  + 2\sqrt{2}\eta \sqrt{16 \eta^2-1}}\right\}.$$
 which for optimal $\eta^*=\frac{1}{8 \sqrt{2}} \sqrt{ 71-3\sqrt{105}}\simeq 0.2804$, gives  the optimal rate $\lambda_{\max}^*=\sqrt{\frac{1}{2} \left( 1+ \sqrt{1-\frac{1}{8} \left(3+6 \alpha-\alpha^2 -(1-\alpha)\sqrt{(1-\alpha)(9-\alpha)}\right)}\right)}$ for $\alpha=1/4$, i.e. $\lambda_{\max}^*\simeq 0.956$.
  
  Define now $B= \begin{pmatrix} -1 &  0  \\ 0 & -\frac12\ \end{pmatrix}$. We can run OGDA on the general-sum game $(x^TAy,x^TBy)$ and Theorem \ref{prop_conv_OGDA_xAy_xBy} applies. Since $B^TA=-Id$, we obtain the convergence rate  of $\lambda_{\max}=\sqrt{\frac{1}{2}(1+\sqrt{1-4\eta^2})}$, but now under the condition  $4\eta^2<1$. This implies that if $\eta$ is chosen optimally close to $1/2$, we get the best possible  convergence rate 
  $$\lambda_{\max}^{**}=\sqrt{\frac12}\simeq 0.707.$$\\

  Let us come back to the general setup of Theorem \ref{thm_conv_OGDA_general_xAy_xBy}. If $\S(A,B)=\{0\}$, the only case where the theorem applies is the (non interesting) case when  $A=B=0$. So let us assume here that $\S(A,B)\neq \{0\}$. Define:
  $$\mu_{\min}=\min \{\mu>0, \mu \in \Sp(-AB^T)\}\;\; {\rm and}\;\;\mu_{\max}=\max \{\mu>0, \mu \in \Sp(-AB^T)\}.$$
  Theorem \ref{thm_conv_OGDA_general_xAy_xBy} implies that if $4\eta^2\mu_{\max}<1$, then the convergence rate is $\lambda_{\max}= \sqrt{\frac{1}{2}(1+\sqrt{1-4\eta^2 \mu_{\min}}})$. So for $\eta$ optimally chosen, the convergence rate is close to:
  $$\lambda_{\max}^*=\sqrt{\frac{1}{2}\left(1+\sqrt{1-\frac{\mu_{\min}}{\mu_{\max}}}\right)}.$$
  This rate is minimized whenever $\mu_{\min}=\mu_{\max}$, i.e. when $AB^T$ and $B^TA$ have a unique negative eigenvalue. This is in particular the case if $B^T=-A^{-1}$. If $A$ is not square or not invertible, we can use more  generally $B^T=-A^\dagger$, where $A^\dagger$ is the Moore-Penrose inverse of $A$. 

  \begin{lemma}
    \label{lemmepseudo}
    Let $A\in \R^{n\times p}$ and $B=-(A^\dagger)^T$. Then $\Sp(AB^T)\subset \R_-$ and for  $0<\eta<\frac{1}{2}$, the matrix $\Lambda_{A,B}$ is diagonalizable.
  \end{lemma}

  For instance if $A=  \begin{pmatrix} 1 &  0  \\ 0 & 2 \\0& 0 \end{pmatrix}$ we put $B=\begin{pmatrix} -1 &  0  \\ 0 & -\frac12 \\0& 0 \end{pmatrix}$. We can now  apply Theorem \ref{prop_conv_OGDA_xAy_xBy} and obtain:
  
  \begin{thm} \label{improvementdagger}
    Let $A\in \R^{n\times p}$, $b\in \R^n$ and $c\in \R^p$,   and $0<\eta\leq \frac{1}{2}$. Consider the  Optimistic Gradient Descent algorithm  (\ref{OGDA_generalbili}) for the payoffs $g_1(x,y)= x^T Ay+ b^Tx +c^Ty$, and $g_2(x,y)= x^T By +e^Tx+f^Ty$ with    $B=-(A^\dagger)^T$ and some $e\in \R^n$, $f\in \R^p$.
    
Then  $(y_t)_t$ converges to the orthogonal  projection of $y_0$  onto $\{y\in \R^p, Ay+b=0\}$, and  $(x_t)_t$ converges to the  orthogonal projection of  $x_0$  onto $\{x \in \R^n, (A^\dagger)^Tx=f\}$.   

Moreover, the convergence is exponential: if $\eta<1/2$ there exists a constant $C>0$ s.t. for  all $t\geq 0$,
$ \|(x_t,y_t)-(x_\infty,y_\infty)\|\leq  C \;\norm{(x_0, y_0,x_{-1},y_{-1})} \; \lambda_{\rm max}^t  ,\;$ where 
$$ \lambda_{\max}= \sqrt{\frac{1}{2}(1+\sqrt{1-4\eta^2})}.$$
And if $\eta=1/2$ then for every $\lambda>1/\sqrt{2}$ there exists $C>0$ s.t. for  all $t\geq 0$,
$ \|(x_t,y_t)-(x_\infty,y_\infty)\|\leq  C \;\norm{(x_0, y_0,x_{-1},y_{-1})} \; \lambda^t$.
 \end{thm}
The proof follows from Theorem \ref{thm_conv_OGDA_general_xAy_xBy} and Lemma \ref{lemmepseudo}, using the general properties  $\Ker(A^\dagger)=\Ker(A^T)={\Img(A)}^\perp$ and $\Ker(A)= {\Img(A^\dagger)}^\perp$ (see the Appendix for the last sentence of  Theorem \ref{improvementdagger}).  Here also, $\lambda_{\max}$ is the best possible rate for exponential convergence, and if $\eta\simeq \frac12$ we obtain the OGDA optimal  rate of $\sqrt{\frac12} \simeq 0.707$. Notice that we obtain exactly the same limit for $(y_t)_t$  as in Theorem \ref{thm_conv_OGDA_zs_general}  when $B=-A$, but the speed of convergence is much better here when we use the OGDA algorithm in a non zero-sum context with $B=-(A^\dagger)^T$. 

Notice also that if $f=A^\dagger x^*$ for some $x^*$ satisfying $A^Tx^*+c=0$, then $(x_t)_t$ also converges to the same limit as in Theorem \ref{thm_conv_OGDA_zs_general}, i.e. to the orthogonal projection of $x_0$ onto $\{x \in \R^n, A^Tx+c=0\}$. In particular if $c=0$, one can choose $f=0$ to ensure this property.

\subsection{``Cooperation'' induced by OGDA}

% TODO: Should I replace Cooperation with Coordination everywhere ?

Assume the game $(x^TAy, x^TBy)$ has  a ``potential for cooperation'',  in the specific sense  that there exist actions  $(x,y) \in \R^n \times \R^p$ such that both $x^TAy$ and $x^TBy$ are positive. Define the sequence $(x_t,y_t)_t$ by $(x_t,y_t)=t(x,y)$ for each $t$, then both $x_t^TAy_t \xrightarrow[t \to +\infty]{} +\infty$ and $x_t^TBy_t \xrightarrow[t \to +\infty]{} +\infty$.
Even though we are not reaching a Nash equilibrium, this  may  be seen as a desirable outcome of the interaction between the players, who aim at maximizing their payoffs.

A particular case  is  common-payoff games, for which $A=B$.
This includes the case where $A=B=(1)$, for which there is no hope to obtain convergence of OGDA to a Nash equilibrium of the game except if the vectors at initialization are already null.
However, we  can  reach an infinite payoff with OGDA, as was seen in section \ref{subsec_xAy_xBy}.
We now  generalize this property  to a larger group of matrices.\\

We use the usual version of OGDA for general-sum games, that is  algorithm (\ref{OGDA_nzs}), and  the matrix describing the dynamics  is $\Lambda_{A,B}$,  as introduced  in section \ref{subsec_xAy_xBy}. In the proof of Theorem \ref{thm_conv_OGDA_xAy},  we showed that a matrix $\Lambda_{A,-A}$ is always  diagonalizable.
In  the general case of matrices $A$, $B$, assuming that the spectrum of $B^TA$ is included in $\R$ is not sufficient to conclude that $\Lambda_{A,B}$ is diagonalizable for small $\eta$. However if $A=\alpha B$  for some non-zero   parameter $\alpha$, 
we will show that $\Lambda_{A,B}$ is diagonalizable for $\eta$ small enough. \\

We know from Proposition \ref{prop_spectrum_xAy_xBy} that:
$\Sp(\Lambda_{A,B}) = \bigcup_{\mu \in \S(A,B)} S^*(\mu),$ with  $S^*(\mu)  = \{ \lambda \in \C, \lambda^2(1-\lambda)^2 =\mu \eta^2(1-2\lambda)^2\}.$
   Define $E_\lambda$ as the set of  complex  eigenvectors of $\Lambda_{A,B}$ associated to an eigenvalue $\lambda$. 

\begin{lemma} \label{lem314} Assume that $\eta < \frac{1}{2 \sqrt{\mu_{\max}}}$, and that ($A=\alpha B$ for some real $\alpha\neq  0$), or ($A$ and $B$ are square  matrices in $\R^{n\times n}$ such that $B^TA$ is diagonalizable with $n$  distinct nonzero real eigenvalues).   Then 
$\C^{(n+p+n+p)}=\bigoplus_{\lambda \in \Sp(\Lambda_{A,B}) } E_\lambda.$
\end{lemma}
The proof is in the Appendix.\\

We now present our last theorem, related to  the convergence of OGDA to a Nash equilibrium or to infinite payoffs in the case of general non zero-sum games with payoffs $(x^TAy +b^Tx +c^Ty +d, x^TBy+e^Tx+f^Ty+g)$. \\

\begin{thm}
  \label{thm_conv_OGDA_general_coop}
 Let $A, B$ be in $\R^{n\times p}$,  $b$, $e$ in $\R^n$, $c$, $f$ in $\R^p$, and  $d$, $g$ in $\R$, with $\S(A,B) \subset \R$.
 Let $\eta\in (0,\frac{1}{2 \, \sqrt{\mu_{\max}}})$, where $\mu_{\max}=\rho(B^TA) = \rho(AB^T)$ is  the largest eigenvalue of $B^TA$.
 Assume that $\Lambda_{A,B}$ is diagonalizable in $\C$, and  that a Nash equilibrium exists, i.e. that $\{(x,y)\in \R^n \times \R^p, B^Tx+f=0, Ay+b=0\}\neq \varnothing$. \\
  
  Given  an initialization $(x_0, y_0,x_{-1},y_{-1})\in \R^n\times \R^p\times \R^n \times \R^p$, consider the  Optimistic Gradient Descent algorithm (\ref{OGDA_generalbili}). Then either $(x_t,y_t)_t$ converges exponentially fast to a Nash equilibrium, that is, an element of $\{(x,y)\in \R^n \times \R^p, B^Tx+f=0, Ay+b=0\}$, or the current payoffs $x_t^TAy_t+b^Tx_t+c^Ty_t + d $ and $x_t^TBy_t + e^Tx_t + f^Ty_t+g $ converge exponentially fast to $+\infty$. \\

 The assumptions $\S(A,B) \subset \R$ and $\Lambda_{A,B}$   diagonalizable    are in particular true in each of the  following cases: 

1)  $B=\alpha A$ for some real number $\alpha\neq 0$, 

2) $A$ and $B$ are square matrices in $\R^{n\times n}$ such that $B^TA$ is diagonalizable with $n$  distinct nonzero real eigenvalues. 
\end{thm}
Notice that case $1)$   includes the case of common payoffs games $A=B$, as well as the case of zero-sum games $B=-A$. \\

\section{Illustration: Generative Adversarial Networks}

\label{sec_illustration}
Algorithms to approximate probability distributions and generate new data from some unknown laws are more and more needed. Despite this necessity, the theory behind this problem is lacking, resolving in inconsistent performances. The main type of generative algorithms are Generative Adversarial Networks (GANs) introduced in \cite{goodfellow2014generative} by Goodfellow \textit{et al.}

\subsection{GANs and WGANs}

In GANs, two neural networks (a discriminator and a generator) are set in competition against each other. The goal of the generator is to generate new data as close as possible to the  true data.
And in  the original version of GANs,  the discriminator aims at recognizing true data from the created ones.  We will here  mostly consider  Wasserstein GANs, a very popular sort of GANs that was introduced by Arjovsky in \cite{arjovsky2017wasserstein} and uses the Wasserstein distance  between probability distributions $\nu$ and $\nu'$,  defined by
\[W_1(\nu, \nu') = \sup_{f \in \lip} \E_{\nu}[f] - \E_{\nu'}[f].\]\\
There is a true distribution $\nu^*$, and we (the generator) want to generate a distribution as close as possible to $\nu^*$.
Unfortunately, we do not have access to the true distribution, but only to an empirical approximation $\hat{\nu}$, via some samples generated according to $\nu^*$. We also cannot choose any distribution but are restricted to choose a parameter $\theta$, and 
 we  want to find the parameter  minimizing the Wasserstein distance between the generated distribution  and the empirical distribution $\hat{\nu}$.\\

More precisely,  the generator $G_\theta$ is a usual neural network  from a latent space $\R^m$ to the feature space $\R^d$.
From a Gaussian distribution $Z \sim \mathcal{N}(0, I_{m})$ it defines a probability measure $\nu_\theta$, which is the law of the random variable $G_\theta(Z)$.
In order to find the generated measure which is the closest to $\hat{\nu}$, we want to find the parameter $\theta$ minimizing the Wasserstein distance between $\hat{\nu}$ and $\nu_\theta$.
This is given by the problem
\[ \inf_\theta W_1(\hat{\nu}, \nu_\theta) = \inf_\theta \sup_{f \in \lip}  \E_{\hat{\nu}}[f] - \E_{\nu_\theta}[f] \]
where $f$ plays the role of the discriminator.\\

However, in practice, we cannot compute the supremum over $1$-Lipschitz function, thus we need to approximate this set using neural networks.
We parameterize the set of discriminators, and denote by $D_\beta$ the 1-Lipschitz neural network  from $\R^d$ to $\R$ with parameter $\beta$.
In practice, we want  the set $\{D_\beta | \beta\}$ to be dense in the set of 1-Lipschitz function (as shown possible, for instance, in Anil \textit{et al.} in \cite{anil2019sorting}).  \\ %TODO Quote also Deel-Lip ?

The WGAN problem we finally want to solve can thus be written as
\begin{equation} \inf_\theta \sup_\alpha \mathcal{L}(\theta, \beta)  \text{ where }  \mathcal{L}(\theta, \beta) = \E_{\hat{\nu}}[D_\beta] - \E_{\nu_\theta}[D_\beta] \end{equation}
% While in general this minimax problem is highly non-convex/non-concave, because the neural networks are non-convex with regard to their parameters, the loss is bilinear in $G_\theta$ and in $D_\beta$.\\

\subsection{First example} 

We now look at a simple illustrative example introduced  by Daskalakis \textit{et al.} in \cite{daskalakis2017training}, and see how  our theoretical bound of convergence looks like. 
We assume that the data follows a multivariate normal distribution of mean $v \in \R^d=\R^n$: $\nu^* = \mathcal{N}(v, I_d)$, where $v$ is unknown, and we make the simplifying assumption $\hat{\nu}=\nu^*$. 
We begin from a multivariate normal distribution $Z$ with  mean 0: $Z \sim \mathcal{N}(0, I_n)$, and study linear generators of the form $G_\theta(z) = z + \theta$.
For the set of discriminators, we use the very small set  of all $D_\beta: x \mapsto \langle \alpha, x\rangle$, where $\norm{\beta}_2 \leq 1$ so that $D_\beta$ is 1-Lipschitz.
 
The WGAN problem becomes:
\[ \inf_\theta \sup_\beta \E_{x \sim \mathcal{N}(v, I_n)}[\langle \beta, x\rangle] - \E_{z \sim \mathcal{N}(0, I_n)}[\langle \beta, z+\theta\rangle],\]
that is $$\inf_\theta \sup_\beta \langle\beta, v-\theta\rangle.$$
We are in the case $x^TAy + b^Tx + c^Ty + d$ of  section \ref{subsec_general_zs} with $A = -I_n \in \R^{n \times n}$ (hence $\mu_{\max}=1$), $b=v$, $c=0$ and $d=0$. $\theta$ is arbitrary in $\R^n$.  A priori we have the constraint $\|\beta\|_2\leq 1$, but allowing $\|\beta\|$ arbitrary does not change the value nor the optimal strategies of the generator, so we simply skip the constraint and consider any $\beta$ in $\R^d$. 
%  but fortunately  if $\beta_0$ respects $\norm{\beta_0}_2 \leq \frac{1}{C}$, then, for all $t$, $\norm{\beta_t}_2 \leq 1$, which means that we can drop this constraint as long as we choose well $\beta_0$.
Theorem \ref{thm_conv_OGDA_zs_general} 
  implies here that  for $\eta < \frac{1}{2}$, OGDA will converge exponentially fast, with ratio $\lambda_{\max}=\sqrt{\frac12(1+\sqrt{1-4\eta^2})}$.
Running OGDA on this example with $\eta=0.3$ gives us $\lambda_{\max} = \frac{3}{\sqrt{10}}\approx 0.949$ and $C \approx 2.197$.\\

% Theorem \ref{thm_conv_DOGDA} 
%  implies here that  for $\eta < \frac{1}{2\sqrt{\mu_{\max}}}$, OGDA will converge exponentially fast.
% Running OGDA on this example with $\eta=0.3 < \frac{1}{\sqrt{2}}$ gives us $\lambda_{\max} = \frac{3}{\sqrt{5}}\approx 1.34$ and $C \approx 3.45$.\\

\begin{figure}[h]
  \begin{subfigure}{0.5\textwidth}
  \includegraphics[width = \textwidth]{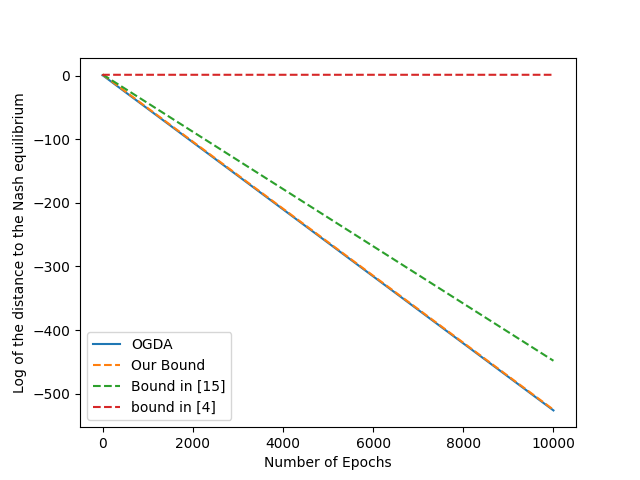}
  \caption{$\eta=0.3$}
  \end{subfigure}
  \begin{subfigure}{0.5\textwidth}
  \includegraphics[width = \textwidth]{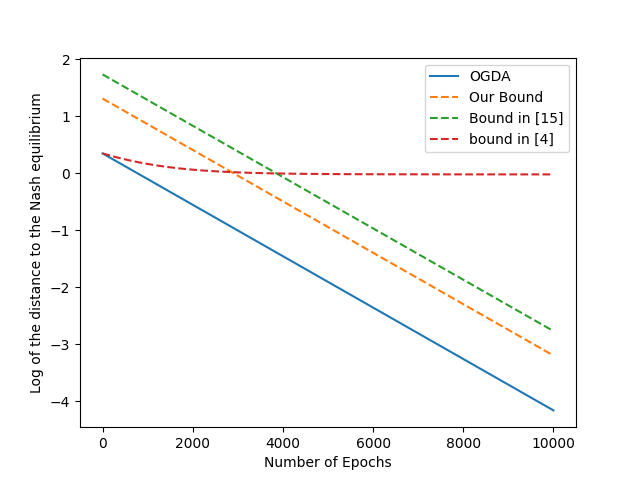}
  \caption{$\eta=0.03$}
  \end{subfigure}
  \caption{\centering Log plot of the distance to the Nash equilibrium and predicted bounds}
  \label{wgans_ogda}
\end{figure}

We plotted the graph of the distance of OGDA to the set of Nash equilibria with regard to time, on figure \ref{wgans_ogda}, where we took $n=2$ and $v = \begin{pmatrix} 3\\4 \end{pmatrix}$. In the  pictures, we see the exponential convergence of OGDA for different values of $\eta$.  We also plotted the different bounds that were found previously.

In the first picture, $\eta$ is close to the maximum possible level, so the convergence is really fast.
We have seen in section \ref{subsec_comments} that our ratio $\lambda_{\max}$ is the best  possible, and these figures mark well this fact: the plots of the log of the distance to the limit Nash equilibrium and of the log of the bound we found are parallel.
However, if $\eta=0.03$ our  constant $C$ is not optimal, as can be seen by the space between the two parallel lines.\\

For most of the different bounds   found in the literature, the higher $\eta$ is, the fastest the convergence occurs. 
This result can be seen in figure \ref{wgans_ogda}, as the bound given by the different papers, as well as the true speed of convergence of OGDA is faster for $\eta=0.3$ compared to $\eta=0.03$.
The previous papers gave results only for $\eta \leq \frac{1}{2\sqrt{\mu_{\max}}}$, and we know that, for this range of $\eta$, whatever the matrix $A$, the theoretical bounds on OGDA will be the fastest when $\eta$ equals $\frac{1}{2\sqrt{\mu_{\max}}}$.
What was shown here in section \ref{section_bilinear} is that OGDA will still converge for any $\eta \in \left[ \frac{1}{2\sqrt{\mu_{\max}}}, \frac{1}{\sqrt{3\mu_{\max}}} \right)$, and that for all matrices, the optimal $\eta$ is in this range.\\

In addition to this amelioration, a final step toward a better convergence rate has been  done in subsection \ref{subsec_Adagger}, where transforming the zero-sum game into a general-sum game using $B=-(A^\dagger)^T$, converging to the same limit point with higher speed greatly boost the convergence to a Nash equilibrium of the non-zero game. This allows us to have the best possible exponential ratio: $\lambda_{\max} = \frac{1}{\sqrt{2}}$ which improves a lot all previous theoretical rates. In the present case where $A=-I_n$, we already have $-A= -(A^\dagger)^T$ so there is no gain in changing the game.
   % TODO add figures for the Arxiv.

\subsection{Generalization} 
We generalize here the previous example, still keeping the linear structure. 
We assume that the data follows a multivariate normal distribution of mean $v \in \R^d$: $\nu^* = \mathcal{N}(v, I_d)$, where $v$ is unknown, and we make the simplifying assumption $\hat{\nu}=\nu^*$.  
Here $Z \sim \mathcal{N}(0, I_m)$, with $m<<d$ and study linear generators of the form $G_\theta(z) = \theta + A_1 z\in \R^d$, for some matrix $A_1\in \R^{d\times m}$. 
For the set of discriminators, we use the  set  of all $D_\beta: x \mapsto \langle A_2\beta, x\rangle$, for $\|\beta\|_2\leq 1$, with $A_2$ a fixed matrix in ${\R^{d\times d}}$ satisfying $\rho(A_2^TA_2)\leq 1$ so that $D_\beta$ is 1-Lipschitz for each $\beta$. 

The WGAN problem becomes:
\[ \inf_\theta \sup_\beta \E_{x \sim \mathcal{N}(v, I_d)}[\langle A_2\beta, x\rangle] - \E_{z \sim \mathcal{N}(0, I_m)}[\langle A_2 \beta, \theta + A_1 z \rangle],\]
that is $$\inf_\theta \sup_\beta \langle A_2\beta, v-\theta\rangle.$$
We are in the case   $x^TAy + b^Tx + c^Ty + d$ of  section \ref{subsec_general_zs} with $x=\beta$, $y=\theta$, $A = -A_2$, $b=A_2^Tv$, $c=0$ and $d=0$. $\theta$ is arbitrary in $\R^n$ and a priori we have the constraint $\|\beta\|_2\leq 1$, but again we skip the constraint which does not change the value nor the optimal strategy of the generator.
The set of Nash equilibria is then $\Ker(A_2)\times \{\theta \in \R^n, v-\theta\in \Ker(A_2^T)\}$.\\

We now consider the final amelioration given in  subsection \ref{subsec_Adagger}.  Assume $A_2^T\neq 0$ and define $\mu_{\max}=\max\{\mu, \mu \in \Sp(A_2^TA_2)\}\leq 1$, $\mu_{\min}=\min\{\mu>0, \mu \in \Sp(A_2^TA_2)\}$ and $\alpha=\frac{\mu_{\min}}{\mu_{\max}}$. Running OGDA with  $\eta<\frac{1}{  \sqrt{3\mu_{\max}}}$, Theorem \ref{thm_conv_OGDA_zs_general} ensures that $(\beta_t)$ converges to the orthogonal projection of $\beta_0$ onto $\Ker(A_2)$, and $(\theta_t)$ converges to the orthogonal projection of $\theta_0$ on $\{\theta \in \R^n, \theta \in v+ \Ker A_2^T\}$. The geometric ratio of convergence $\lambda_{\max}$ is  $\max\{\lambda_*, \lambda_{**}\}$, where
  $$\lambda_* = {\bf{1}}_{\mu_{\min} \leq \frac{1}{4\eta^2}} \sqrt{\frac{1}{2}(1+\sqrt{1-4\eta^2 \mu_{\min}})}\; \text{ and } \lambda_{**} = {\bf{1}}_{\mu_{\max} \geq \frac{1}{4\eta^2}} \sqrt{2 \eta^2\mu_{\max} + \eta \sqrt{\mu_{\max}} \sqrt{4\eta^2 \mu_{\max}-1}}.$$ And for  an optimal $\eta$, we obtain the  optimal rate of  $$\lambda_{\max}^*=\sqrt{\frac{1}{2} \left( 1+ \sqrt{1-\frac{1}{8} \left(3+6 \alpha-\alpha^2 -(1-\alpha)\sqrt{(1-\alpha)(9-\alpha)}\right)}\right)}\in [\frac{\sqrt{2}}{{2}},1).$$
  We can introduce $B=A_2^\dagger$ and run OGDA for the general-sum game $(x^T(-A_2^T)y+v^T A_2x, x^TBy)$. According to Theorem \ref{thm_conv_OGDA_general_xAy_xBy}, OGDA now converges to the same limit but with geometric ratio $\lambda'_{\max} = \sqrt{\frac12(1+\sqrt{1-4\eta^2})}$ which, for $\eta=1/2$, is $\frac{\sqrt{2}}{{2}}$, the best possible ratio. And this is true for all matrices $A_2$ and vector $v$.
  
  {\bf Example:} $-A_2=A=  \begin{pmatrix} 1 &  0  \\ 0 & 1/2\ \end{pmatrix}$ and $v=(1,1)$. In the zero-sum game, P1 maximizes $g_1(\beta,\theta)=-\beta A_2^T\theta+\beta A_2^Tv$, whereas P2 minimizes $g_1(\beta,\theta)$. Running OGDA for the zero-sum game, for $\eta<\sqrt{\frac{1}{3}}\simeq 0.577$ induces  the convergence of $(\beta_t, \theta_t)$ to $(0,v)$, with optimal geometric ratio of convergence $\lambda^*_{\max}\simeq 0.9538$ for $\eta^* \simeq 0.5608$. 
  Let us run OGDA on the general-sum game where player 1 still maximizes $g_1(\beta,\theta)$ and player 2 now maximizes $g_2(\beta,\theta)=\beta^TB \theta$, with $B= A_2^\dagger=\begin{pmatrix} -1 &  0  \\ 0 & -2\ \end{pmatrix}$. By Theorem \ref{improvementdagger}, $(\beta_t, \theta_t)$ still converges to the same limit $(0,v)$, and if $\eta \simeq 1/2$ the convergence ratio is now  the best possible  convergence rate  $\lambda_{\max}^{**}=\frac{1}{\sqrt{2}}\simeq 0.707.$\\
  The speed of convergence can be seen on the figure \ref{fig_dagger}.
  Using a general-sum game  with the matrix $A_2^\dagger$ greatly improves the speed of convergence of OGDA.

  \begin{figure}[h]
    \centering
    \includegraphics[width = 0.7\textwidth]{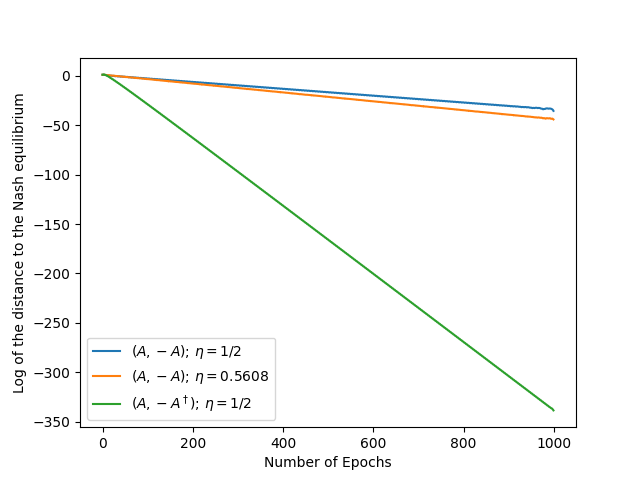}
    \caption{Acceleration with $A^\dagger$}
    \label{fig_dagger}
  \end{figure}

\subsection{Discussion and limitations}
Our results   concern bilinear games, which can be seen as  toy-models compared to the complexity of real WGANs. To go toward  more realism, several issues  could be considered.

In  practice, neural networks have millions of parameters, if not billions, and they suffer from the curse of dimensionality.
However, OGDA will not suffer more than GDA in this account, because OGDA, as GDA is a first order method with a single oracle call per iteration.
The only difference is that OGDA needs to store previous gradient, and thus needs more memory space than GDA.
For both of these algorithms, the speed of convergence to an equilibria is independent of the dimensions $n$ and $p$ in the case of bilinear games.

Another issue  is to take into account the stochastic character of the answer that we receive.
In practice, we won't exactly have access to $\E_{x \sim \mathcal{N}(v, I_d)}[\langle \alpha, x\rangle]$, nor to $\E_{z \sim \mathcal{N}(0, I_d)}[\langle \alpha, z+\theta\rangle]$, but only to empirical averages of data following the according distribution.
In \cite{daskalakis2017training}, Daskalakis \textit{et al.} already considered  this problem.
A deeper study on stochasticity for OGDA and Extra-Gradient was done by Hsieh \textit{et al.} in \cite{hsieh2020limits}.
They showed a geometric convergence of the expected distance to the saddle points for strongly concave/convex games, with a bound depending on the variance of the noise.

A further step is not to assume the unboundedness of the problem.
In  this article  like in many others, we assume  that the set of parameters is an Euclidean space: we are doing minimization and maximization on $\R^p$ and $\R^n$. But  in  reality, neural networks are computer-based, thus the set of possible neural-networks is far more complex. However, the  goal is to model the largest possible set of  functions, and the maxima and minima reached by the weights are very large, and will get larger if float maximum bit-size gets higher.

%so  assuming that  we are in unbounded  spaces may not be that much of a problem compared to the *bigger assumption that our losses are  bilinear.

A last important factor, and not the least of them,  is the fact that most of the GANs problem used in practice are non-concave/non-convex,  because neural networks do not have a convex dependence with their parameters. Convergence to equilibria for non-concave/non-convex games is currently far from being  well-understood \cite{daskalakis2022}.  \\

%Regarding local equilibria,  it was shown that any method needs super-polynomial time to compute an $\epsilon$-approximate local Nash equilibrium  \cite{daskalakis2021complexity}.\\

% Despite these limitations, several points we showed here are of importance.\\
% The first one, already stated in several paper, is the improvement we make using OGDA compared to GDA.

\section{Conclusion} \label{secconclusion}

We have proved   the exponential convergence of OGDA to a Nash equilibrium for any bilinear game, and also provided the best ratio for the convergence.
This implies  an important stability property for  the OGDA system as the number of steps increases. In case of multiple Nash equilibria, we  characterized the limit equilibrium as a function of the initialization, and gave a simple expression of it. 
We also improved  the range of convergence of OGDA with regard to the gradient step $\eta$ in the case of zero-sum games.
These results clarify and generalize all previous results of the literature on the topic.

We also extended the study to  general-sum games.  We   presented a simple trick to guarantee, via a variant of OGDA,  convergence to Nash equilibria in our general-sum games.   We   also showed  that in an important class of games (including the common payoffs case as well as the zero-sum case) either OGDA converges to a Nash equilibrium, or the payoffs of both players converge to $+\infty$. 

We gave sufficient conditions for OGDA to converge to a Nash equilibrium of a general-sum game, and showed how to increase the speed of convergence of  a min-max problem involving a payoff matrix $A$ by introducing a general-sum game  using the matrix $-(A^\dagger)^T$, i.e. minus the transpose Moore-Penrose inverse of $A$. This may be seen as a proof of concept that  non-zero sum games 
can be used to solve efficiently Minmax or zero-sum problems. Extending this approach to general, non necessarily bilinear, games, seems promising and is left for future research.

The bilinear games we have studied are widely used as a toy model for many applications, and we illustrated our results on   GAN examples. These  may  be used  as an  additional argument in favor of using the Optimistic version of GDA in GANs. Nevertheless, if one wants to go further in the application to GANs, there are  several factors that should be  taken into account.
One of them would be  the stochasticity of the payoffs. Some experiments were made in \cite{daskalakis2017training}, but no theoretical convergence speed is known, nor any bound on the proximity to a Nash equilibrium depending on the noise of the stochastic payoff.
A second and harder factor is the non-concave/non-convex properties  of real  instances of GANs.

%Several papers already consider this case, and lots of problem arise  in this scenario. % Quote the *papers, and the problem (PPAD convergence to local eps-NE?)

All in all, the general comprehension of OGDA  is improving,  and so are the applications to the stability of first-order methods  for  GANs. This paper aims at continuing this  dynamics, in giving more insights into zero-sum and non zero-sum  problems.

\newpage

\bibliography{biblio}
\bibliographystyle{abbrv}

\newpage

\appendix
\section{Appendix: Proofs}
\subsection{Proof of Theorem \ref{thm_conv_OGDA_xAy}}
\label{subsec_proof_thm_conv_OGDA_xAy}

We fix $A\in \R^{n\times p}\neq 0$ and $\eta>0$.

 \subsubsection{Convergence  to Nash equilibria for general $\eta>0$}
 \label{section_cvne}
 
Recalling the characteristic  polynomial of $\Lambda$ in the particular case where $A=(1)$, we introduce the following sets:
\begin{defi}
  Given $\mu \in \C$, we define $$S(\mu) = \{ \lambda \in \C, \lambda^2(1-\lambda)^2 +\mu \eta^2(1-2\lambda)^2=0\}.$$
\end{defi}
%$\mu$ begin fixed, let $\nu$ and $\delta$ be such that $\nu^2=-\mu$ and $\delta^2= 1- 4 \eta^2 \mu$. It is easy to see that $\lambda\in S(\mu)$ if and %only if \[\lambda=\frac12 (1 \pm 2 \eta \nu \pm \delta)\]
Notice that if $\eta=0$ or $\mu=0$, then $S(\mu)=\{0,1\}$. Simple computations imply the following lemma.

% \begin{rem}
% Let us denote by $q_A$ the characteristic polynomial of the matrix induced by GDA, for the bilinear problem with matrix $A$.
% In \cite{daskalakis2018limit}, they said that the characteristic polynomial of $\Lambda_A$ is equal to $(2X-1)^{n+p}q_A \left(\frac{X^2+X-1}{2X-1}\right)$, but they made a slight mistake, and the true result is $(2X-1)^{n+p}q_A \left(\frac{X^2+X-1}{2X-1}\right)$. It doesn't change their final results, but from there comes our different results for the eigenvalue.
% \end{rem}

\begin{lemma}
  \label{lem_expression_Smu}
  Assume $\mu$ is real non-negative. We write $X=\eta \sqrt{\mu}\geq 0$,  and let $\delta\in \C$ be such that $\delta^2=1-4X^2$. More precisely, if $4X^2\leq 1$ we use $\delta=\sqrt{1-4X^2}$ and if $4X^2\geq 1$ we use $\delta= i \sqrt{4X^2-1}$.

Then we can write $$S(\mu)=\{\lambda_1(\mu), \lambda_2(\mu), \lambda_3(\mu), \lambda_4(\mu)\},\; {\it with}$$
 $$\lambda_1(\mu)=\frac{1}{2}  (1+ \delta + 2 iX)=\overline{\lambda_3}(\mu), \;{\it and}\;\lambda_2(\mu)= \frac{1}{2}  (1- \delta + 2 iX)=\overline{\lambda_4}(\mu).$$

  0) If $X=0$, then $S(\mu)=\{0,1\}$.
  
1) If $0<X<1/2$, $S(\mu)$ has 4 elements and  $|\lambda_1(\mu)|^2=|\lambda_3(\mu)|^2=\frac12(1+\sqrt{1-4X^2})<1$, and $|\lambda_2(\mu)|^2=|\lambda_4(\mu)|^2=\frac12(1-\sqrt{1-4X^2})<\frac12.$

2) If $X=1/2$, $S(\mu)=\{\frac12(1+i), \frac12(1-i)\}$.

3) If $X>1/2$,  $S(\mu)$ has 4 elements and  $|\lambda_1(\mu)|^2=|\lambda_3(\mu)|^2=2X^2+X\sqrt{4X^2-1} $, and $|\lambda_2(\mu)|^2=|\lambda_4(\mu)|^2=2X^2-X\sqrt{4X^2-1}.$
\end{lemma}

Notice that for $X\geq 1/2$, the property $2X^2+X\sqrt{4X^2-1} <1$ is equivalent to $X< \frac{1}{\sqrt{3}}$.

The link between the eigenvalues of the matrices $A^TA$ and $AA^T$, and the eigenvalues of $\Lambda$ is the following.

\begin{lemma}
  \label{lem_eigenvalues_Lambda}
 \[\Sp(\Lambda) = \bigcup_{ \mu \in \S(A)} S(\mu)\]
%  In particular $\abs{\lambda} \leq 1$ for all $\lambda \in \Lambda$.
\end{lemma}

\begin{proof}
 For  $\lambda\in \C$ and $Z =(x^T,y^T,x'^T,y'^T)^T$ $\in \C^n\times \C^p\times \C^n\times \C^p$, we have: 
  \begin{align*}
  \Lambda Z = \lambda Z
  &\iff \left\{ \begin{matrix} x + 2\eta Ay - \eta Ay' = \lambda x\\ y -2\eta A^Tx + \eta A^Tx' = \lambda y \\ x = \lambda x' \\ y = \lambda y' \end{matrix}\right.\\
  &\iff \left\{ \begin{matrix} \lambda(1-\lambda)x' + (2\lambda -1)\eta Ay' = 0\\ \lambda(1-\lambda)y' - (2\lambda-1)\eta A^Tx' = 0 \\ x = \lambda x' \\ y = \lambda y' \end{matrix}\right.\\
  \end{align*}
\indent a) Assume $\Lambda Z=\lambda Z$ and $Z\neq 0$. Multiplying the first line by $(2\lambda-1)\eta A^T$ and the second line by $(2\lambda-1)\eta A$, we get:
  \begin{align}\left\{ \begin{matrix}\lambda^2(1-\lambda)^2y' + (2\lambda -1)^2\eta^2 A^TAy' = 0\\ \lambda^2(1-\lambda)^2x' + (2\lambda -1)^2\eta^2 AA^Tx' = 0\end{matrix}\right. \end{align}\\
Since $Z\neq 0$, we have $x'\neq 0$ or $y'\neq 0$, and $\lambda \neq 1/2$. This implies that $-\frac{\lambda^2(1-\lambda)^2}{(2\lambda-1)^2\eta^2}$ is an eigenvalue  of $A^TA$ or of $AA^T$. If we denote by $\mu$ such eigenvalue, then $\lambda\in S(\mu)$.  

b) Conversely, let $\mu$ be an eigenvalue of $A^TA$, and consider $\lambda\in S(\mu)$. Let $x'\neq 0$ be  such that $A A^T x'=\mu x'$ and consider $y'$ satisfying:\[\lambda(1- \lambda) y'=\eta (2\lambda-1) A^Tx'.\]
If  $\mu\neq 0$, then $\lambda\notin \{0,1\}$ and $y$ is uniquely defined; if $\mu=0$ then $A A^T x'=0$ so $A^Tx'=0$ and any $y'$ will do.
In each case, one can check that  $ \Lambda (\lambda x', \lambda y', x',y')^T =\lambda Z$, so that $\lambda$ is an eigenvalue of $\Lambda$. 
This also stands for $\mu$ an eigenvalue of $AA^T$.

a) and b) show that the set of eigenvalues of  $\Lambda$ is $\{ \lambda | \lambda \in S(\mu), \mu \in \S(A)\}$. \end{proof}

Combining Lemma \ref{lem_expression_Smu} and Lemma \ref{lem_eigenvalues_Lambda} yields; 

\begin{cor}
$$\max\{|\lambda|, \lambda \in \Sp(\Lambda)\backslash\{1\}\} <1 \; \Longleftrightarrow \forall \mu \in \S(A), \eta \sqrt{\mu} < {\frac{1}{\sqrt{3}}} \Longleftrightarrow \eta < {\frac{1}{\sqrt{3\mu_{\max}}}} $$
And if $\eta > {\frac{1}{\sqrt{3\mu_{\max}}}}$ we have $\max\{|\lambda|, \lambda \in \Sp(\Lambda) \}>1$, and for some intial conditions the OGDA diverges.
\end{cor}
In the sequel of the proof of Theorem \ref{thm_conv_OGDA_xAy}, we always assume that $$\eta < {\frac{1}{\sqrt{3\mu_{\max}}}}.$$ 
Consider $\mu$ in $\S(A)$. Since $A^TA$ and $AA^T$ are positive semi-definite, $\mu$ is a  non-negative real number. 
If $\mu \neq 0$, $S(\mu)\subset \{ \lambda \in \C,  |\lambda| <1\}$ by Lemma  \ref{lem_expression_Smu}. If $\mu=0$, then $S(\mu)=\{0,1\}$. By Lemma  \ref{lem_eigenvalues_Lambda}, we have:
\[\Sp(\Lambda)\subset \{1\}\cup \{\lambda \in \C, | \lambda|<1\}.\]
Let us  define $\lambda_{\max}= \max\{|\lambda|, \lambda \in \Sp(\Lambda)\backslash\{1\}\}<1$. One can check that $\lambda_{\max}=\max\{\lambda_*,\lambda_{**}\}$ with 
$\lambda_{*}= {\bf{1}}_{\mu_{\min} \leq \frac{1}{4\eta^2}} \sqrt{\frac{1}{2}(1+\sqrt{1-4\eta^2 \mu_{\min}})}$ and $\lambda_{**} = {\bf{1}}_{\mu_{\max} \geq \frac{1}{4\eta^2}} \sqrt{2 \eta^2\mu_{\max} + \eta \sqrt{\mu_{\max}} \sqrt{4\eta^2 \mu_{\max}-1}}.$
Notice that if $\eta \leq {\frac{1}{2\sqrt{\mu_{\max}}}}$, $\lambda_{\max}=\lambda_*$.\\

We now describe the eigenspaces of $\Lambda$. For each $\lambda$ in $\C$, we write $E_\lambda=\{Z \in \C^n\times \C^p \times \C^n\times \C^p, \Lambda Z =\lambda Z\}$. The proof of Lemma \ref{lem_eigenvalues_Lambda} easily gives: 

\begin{lemma}
  \label{lem_expression_eigenspaces}
Let $\lambda$ be in $\Sp(\Lambda)$, and  $\mu\geq 0$ in $\S(A)$ uniquely defined by  $\lambda^2(1-\lambda)^2 +\mu \eta^2(1-2\lambda)^2=0$.

If $\lambda=0$, then $\mu=0\in \S(A)$ and $E_0=\{(x,y,x',y')\in \C^{n}\times \C^p \times \C^n\times \C^p, x=0, y=0, Ay'=0, A^Tx'=0\}$, so 
$\dim(E_0)= \dim(\Ker(A)) + \dim(\Ker(A^T)).$

If $\lambda=1$, then $\mu=0\in \S(A)$ and $E_1=\{(x,y,x',y')\in \C^{n}\times \C^p \times \C^n\times \C^p, x=x', y=y', Ay'=0, A^Tx'=0\}$, so 
$\dim(E_1)= \dim(\Ker(A)) + \dim(\Ker(A^T)).$

If $\lambda \notin\{0,1\}$, then $\mu>0$ and 
 \[E_\lambda = \left\{(x,y,x',y')\in \C^{n}\times \C^p \times \C^n\times \C^p,  x = \lambda x' , y=\lambda y', A^TAy'= \mu y' , x'= \frac{(1-2\lambda)\eta }{\lambda(1-\lambda)}Ay'  \right\},\] 
 so 
 $\dim(E_\lambda)= \dim(\Ker(A^TA-\mu I)).$
\end{lemma}

 We can now sum  the dimension of the eigenspaces.
 \begin {lemma}\label{lemA6}
 $$ \sum_{\lambda \in Sp(\Lambda)} \dim(E_\lambda)
 =  2(n+p) -2  \dim(\Ker(A^TA-\frac{1}{4\eta^2} I)).$$
 \end{lemma}
 
 \begin{proof}
Recall that $\rank(A^TA)=\rank(A)=\rank(A^T)$  and, $A^TA$ being positive semi-definite, $\rank(A^TA)= \sum_{\mu >0} \dim(\Ker(A^TA-\mu I))$.  Recall also from Lemma  \ref{lem_expression_Smu} that $S(\mu)$ has 2 elements if $\mu=\frac{1}{4\eta^2}$.
$$
  \sum_{\lambda \in Sp(\Lambda)} \dim(E_\lambda)
 = \dim(E_0) + \dim(E_1) +  \sum_{\lambda \in Sp(\Lambda), \, \lambda\notin \{0,1\}} \dim(E_\lambda) $$
  $$= 2\dim(\Ker(A))+2\dim(\Ker(A^T)) +2 \dim(\Ker(A^TA-\frac{1}{4\eta^2} I))+ \sum_{\mu \in Sp(A^TA), \, \mu \neq 0, \frac{1}{4\eta^2}} 4\dim(\Ker(A^TA-\mu I)) $$
  $$  = 2\dim(\Ker(A))+2\dim(\Ker(A^T))+ 4\rank(A^TA) -2  \dim(\Ker(A^TA-\frac{1}{4\eta^2} I)) $$
$$  = 2(\rank(A) + \dim(\Ker(A))) + 2(\rank(A^T) + \dim(\Ker(A^T))) -2  \dim(\Ker(A^TA-\frac{1}{4\eta^2} I)) $$
 $$ = 2(n+p) -2  \dim(\Ker(A^TA-\frac{1}{4\eta^2} I)).$$
\end{proof}

We can already notice that $\frac{1}{4\eta^2}\notin\S(A)$ is equivalent to  $ \sum_{\lambda \in Sp(\Lambda)} \dim(E_\lambda)=2(n+p)$, i.e. is equivalent to the property that  $\Lambda$ is diagonalizable in $\C$.  The next lemma will be used to characterize the limit of the OGDA.

\begin{lemma}
  \label{lem_ortho}
  If $y$ is  an eigenvector of $A^TA$ associated to a nonzero eigenvalue, then $y$ is orthogonal to $\Ker(A)$. Similarly if $x$ is  an eigenvector of $AA^T$ associated to a nonzero eigenvalue, then $x$ is orthogonal to $\Ker(A^T)$.
\end{lemma}

\begin{proof}
  Let $\hat{y}$ be a vector in the kernel of $A$, and let $\mu \neq 0$ be the eigenvalue of $A^TA$ associated with $y$.
  Then:
  \[ \mu \langle y, \hat{y} \rangle = \langle  \mu y, \hat{y}\rangle = \langle A^TA y , \hat{y}\rangle = \langle Ay, A\hat{y} \rangle = \langle Ay, 0 \rangle = 0\]
  Hence, $\langle y, \hat{y} \rangle = 0$ for any $\hat{y}$ in $\Ker(A)$.
 The proof for eigenvectors of $AA^T$ is similar. 
\end{proof}

\subsubsection{Convergence  to Nash equilibria if  $\frac{1}{4\eta^2}\notin \S(A)$} \label{CVOK}

We assume here that $\frac{1}{4\eta^2}\notin \S(A)$, with $0<\eta<  {\frac{1}{\sqrt{3\mu_{\max}}}}$.
Here $\Ker(A^TA-\frac{1}{4\eta^2} I)$ is empty, so $\sum_{\lambda \in Sp(\Lambda)} \dim(E_\lambda)= 2(n+p)$ and 
$\Lambda$ is diagonalizable. Since $\max\{|\lambda|, \lambda \in \Sp(\Lambda)\backslash\{1\}\} <1$, the matrix $\Lambda^t$ converges as $t\to \infty$, and the convergence is exponential with ratio 
$\lambda_{\max}= \max\{|\lambda|, \lambda \in \Sp(\Lambda)\backslash\{1\}\}.$\\
 
Consider an initial condition $Z_0=(x_0,y_0,x_{-1},y_{-1})$.  $Z_0$ can  be uniquely written :
\[Z_0=\sum_{\lambda \in \Sp(\Lambda)} z_{ \lambda}, \; {\rm with \;} z_{\lambda}\in E_\lambda \; {\rm for \; all  \;} \lambda.\]
For all $t\geq 0$, $Z_t=(x_t, y_t,x_{t-1}, y_{t-1})=\Lambda^t Z_0=\sum_{\lambda \in \Sp(\Lambda)}\lambda^t z_{ \lambda}$ and $Z_t \xrightarrow[t\to \infty]{} Z_\infty:=z_{1}.$ We have obtained  the convergence of $Z_t$ to the projection of $Z_0$ onto  $E_1=\Ker(\Lambda-I)$ along $\bigoplus_{\lambda \neq 1} \Ker(\Lambda-\lambda I)$.
Since $Z_\infty \in E_1$,  $(x_t)_t$ converges to a limit $x_{\infty}$ in $\Ker(A^T)$ and $(y_t)_t$ converges to a limit $y_{\infty}$ in $\Ker(A)$. Moreover  there exists a constant $C'$,  only  depending on $A$ and $\eta$ through the matrix $\Lambda$, such that for all $t\geq 0$, $\|\Lambda
 ^t-\Lambda^{\infty}\|\leq C' \lambda_{\max}^t$. And  for all $t\geq 0$,
$ \|(x_t,y_t)-(x_\infty,y_\infty)\|\leq  C'  \lambda_{\max}^t.$

Now, let us write $z_\lambda$ as $({\hat{x}_\lambda}, {\hat{y}_\lambda}, {\hat{x}'}_\lambda, {\hat{y}'_\lambda})\in \C^n \times \C^p \times \C^n \times \C^p$ for each $\lambda$ in $\S(A)$.
We have $y_0= \sum_{\lambda \in \Sp(\Lambda)} \hat{y}_\lambda$, and $(y_t)_t$ converges to $y_\infty= \hat{y}_1\in \Ker(A)$. $\hat{y}_0=0$ and $\sum_{\lambda \neq \{0, 1\}} \hat{y}_\lambda$ is orthogonal to $\Ker(A)$ thanks to Lemma \ref{lem_ortho}, so we obtain 
that $\hat{y}_1 = y_\infty$ is the orthogonal projection of $y_0$ onto $\Ker(A)$.
In the same way, one can prove that $x_\infty$ is the orthogonal projection of $x_0$ onto $\Ker(A^T)$.

\subsubsection{Convergence  to Nash equilibria if  $\frac{1}{4\eta^2}\in \S(A)$}
We assume here that $\frac{1}{4\eta^2}\in \S(A)$, with $  {\frac{1}{2\sqrt{\mu_{\max}}}}\leq \eta<  {\frac{1}{\sqrt{3\mu_{\max}}}}$. 

  Consider  $\mu=\frac{1}{4\eta^2}>0$. By assumption, $\mu$ is an eigenvalue of $AA^T$ and $A^TA$, and $S(\mu)=\{\lambda_1, \lambda_2\}$ with $\lambda_1=\frac12(1+i)$ and $\lambda_2=\frac12(1-i)$. We know that $\dim(\Ker(\Lambda-\lambda_1I))= \dim(\Ker(\Lambda-\lambda_2I))=\dim(\Ker(A^TA-\mu I)), \;\; {\rm{and}}$
  $ \sum_{\lambda \in Sp(\Lambda)} \dim(E_\lambda)=2(n+p) -2  \dim(\Ker(A^TA-\frac{1}{4\eta^2} I))< 2(n+p),$
  so $\Lambda$ is not diagonalizable in $\C$.\\

 If $0\notin \S(A)$, then by Gelfand's theorem the OGDA converges to 0, and $\Ker(A)=\{0\}$.  To fix ideas we  now assume here that $0\in \S(A)$.\\

 The eigenvalues of $\Lambda$ are 1 and eigenvalues with modulus $<1$. 
 
 \begin{lemma} \label{lem123} The  algebraic multiplicty of $\lambda=1$ in the characteristic  polynomial of $\Lambda$ is $\dim E_1$. \end{lemma}
 
 \noindent {\bf Proof of Lemma \ref{lem123} }
  Let $Z=(x,y,x',y')\in \C^{n}\times \C^p \times \C^n\times \C^p$ be such that $(\Lambda -I)^2 Z=0$. Then $(\Lambda-I)Z\in E_1$, i.e.
  $2 \eta Ay-\eta Ay'=x-x'$, $ -2\eta A^Tx+\eta A^Tx'=y-y'$, $A^T(x-x')=0$ and $A(y-y')=0.$
  Then $Ay=Ay', A^Tx=A^Tx', \eta Ay'=x-x'$ and $-\eta A^Tx'=y-y'$. It implies $A^TA y'=0$, which  is equivalent to $Ay'=0$. Similarly $AA^Tx'=0$, so $A^Tx'=0$ and finally $Z\in E_1$. We have shown that $\Ker(\Lambda-I)^2= \Ker(\Lambda-I)$, so for each $l\geq 2$, $\Ker(\Lambda-I)^l= \Ker(\Lambda-I)$, and the algebraic multiplicty of $\lambda=1$  is $\dim E_1$. \hfill $\Box$
  
  \vspace{0.5cm}

  Looking at the Jordan normal form of $\Lambda$, we can conclude that $\Lambda^t$ converges to a projection matrix $\Lambda^\infty$.  For any initial condition $Z_0=(x_0,y_0,x_{-1},y_{-1})$, $Z_t$ converges to an element of $E_1$, so $(x_t)_t$ converges to a limit $x_{\infty}$ in $\Ker(A^T)$ and $(y_t)_t$ converges to a limit $y_{\infty}$ in $\Ker(A)$. Moreover for every $\lambda>\lambda_{\max}$ there exists a constant $C'$ such that for all $t\geq 0$, $\|\Lambda
 ^t-\Lambda^{\infty}\|\leq C' \lambda^t$. This implies\footnote{This also holds if $0\notin \S(A)$.} 
$ \|(x_t,y_t)-(x_\infty,y_\infty)\|\leq  C'  \lambda^t$ for all $t$.  To conclude, it remains to show that $x_\infty$, resp. $y_\infty$,  is the orthogonal projection of $x_0$, resp. $y_0$,  onto $\Ker(A^T)$.

  \begin{lemma} \label{lem124}
  $$\dim(\Ker(\Lambda-\lambda_1I)^2)= \dim(\Ker(\Lambda-\lambda_2I)^2)=2\dim(\Ker(A^TA-\mu I)).$$
 Moreover, if  $(x,y,x',y')\in  (\Ker(\Lambda-\lambda_1I))^2$, then $x$ and $x'$ are in $(\Ker(A^T))^\perp$ and $y$ and $y'$ are in $(\Ker(A))^\perp$.
  \end{lemma}
 
 \noindent {\bf Proof of Lemma \ref{lem124} } We write  the proof for $\lambda_1$ only. Consider $Z=(x,y,x',y')\in \C^{n}\times \C^p \times \C^n\times \C^p$. Computations show that $Z\in 
  (\Ker(\Lambda-\lambda_1I))^2$ if and only if:
 $$ \left\{ \begin{array} {ccc} 
  2 \eta A(2y-y')& =&i(2x-x') \\
 -2 \eta A^T(2x-x')& =&i(2y-y') \\
  \eta A(y-\lambda_1y')& =&\frac{i}{2} (x-\lambda_1x') \\
  \eta A^T(x-\lambda_1x')& =&-\frac{i}{2} (y-\lambda_1y').
 \end{array}\right.$$
 This implies $AA^T(2x-x')=\mu(2x-x')$ and $AA^T( x-\lambda_1x')=\mu(x-\lambda_1x')$. \\

 Conversely, choose independently any 2 vectors $x_1^*$ and $x_2^*$ in $\Ker(AA^T-\mu I)$, and define uniquely $x$, $x'$, $y$ and $y'$ by:
  $$ \left\{ \begin{array} {ccc} 
  2x-x' &=& x_1^*\\
 x-\lambda_1 x'& =& x_2^*\\
i(2y-y') & =& -2 \eta A^T(x_1^*) \\
  -\frac{i}{2} (y-\lambda_1y') & =& \eta A^T(x_2^*).
 \end{array}\right.$$
 Then one can check that $(x,y,x',y')$ belongs to $\Ker(\Lambda-\lambda_1 I)^2$.
 
 $\Ker(A^TA-\mu I)\subset \Ker(A)^\perp$ by Lemma A.6, so $x_1^*$ and $x_2^*$ belong to $\Ker(A)^\perp$. So $x$ and $x'$ also belong to the vector subspace  $\Ker(A^T)^\perp$, and  similarly $y$ and $y'$ belong to $\Ker(A)^\perp$.
 
 The map $(x_1^*,x_2^*)\mapsto (x,y,x',y')$ is linear one-to-one onto  $\Ker(AA^T-\mu I)\times \Ker(A^TA-\mu I)$, so  $\dim(\Ker(\Lambda-\lambda_1I)^2)= 2\dim(\Ker(A^TA-\mu I)).$ \hfill $\Box$
 
 \vspace{0.5cm}
 
Looking at the dimensions, we obtain that the algebraic dimension of both eigenvalues  $\lambda_1$ and $\lambda_2$ of $\Lambda$  is  $2\dim(\Ker(A^TA-\frac{1}{4\eta^2} I))$, and for $i=1,2$,  the characteristic subspace of $\lambda_i$ is $(\Ker(\Lambda-\lambda_iI)^2)$ for $i=1,2$.  Using the spectral decomposition of $\Lambda$, the initial vector 
 $Z_0$ can now be uniquely written:
$$Z_0=\sum_{\lambda \in \Sp(\Lambda)} z_{ \lambda},$$
  with  $z_{\lambda}\in E_\lambda$ for $\lambda\notin \{\lambda_1, \lambda_2\}$ and $z_{\lambda_i}\in \Ker(\Lambda-\lambda_iI)^2$ for $i=1,2$.
  
 $Z_t=\Lambda^t Z_0 \xrightarrow[t\to \infty]{} Z_\infty:=z_{1},$ so $(x_t)_t$ converges to the first component $\tilde{x}_1\in \Ker(A^T)$ of $z_1$. By  lemmas \ref{lem_ortho} and \ref{lem124}, $x_0-\tilde{x}_1$ belong to $(\Ker(A^T))^\perp$, so $\tilde{x}_1$ is the orthogonal projection of $x_0$ onto $\Ker(A^T)$. Similarly, $(y_t)_t$ converges to the orthogonal projection of $y_0$ onto $\Ker(A)$.\\
 
 This concludes the proofs of parts 1) and 3) of Theorem \ref{thm_conv_OGDA_xAy}.

 \begin{lemma}
   \label{lemma_optimal_value}
   Let $A \in \R^{n\times p}$ with $A \neq 0$, and $0 < \eta < \frac{1}{\sqrt{3 \mu_{\max}}}$.
   Then there exist $Z'_0 \in \R^{n+p+n+p}$ and a constant $c>0$ such that $\norm{Z'_t} > c\lambda_{\max}^t$.
 \end{lemma}

 \begin{proof}
   If $\lambda_{\max}=\lambda_*$,  consider $\lambda=\frac12(1+\sqrt{1-4\eta^2\mu_{\min}}+2i\eta \sqrt{\mu_{\min}})$, and lemmas \ref{lem_expression_Smu} and \ref{lem_eigenvalues_Lambda} will imply that $\lambda$ is an eigenvalue of $\Lambda$ with modulus $\lambda_{\max}<1$. Similarly, if $\mu_{\max} \geq \frac{1}{4\eta^2}$ and $\lambda_{\max}=\lambda_{**}$, consider $\lambda= \frac{1}{2}(1+i \sqrt{4 \eta^2 \mu_{\max}-1} +2i \eta \sqrt{\mu_{\max}})$, and similarly lemmas \ref{lem_expression_Smu} and \ref{lem_eigenvalues_Lambda} will imply that $\lambda$ is an eigenvalue of $\Lambda$ with modulus $\lambda_{\max}<1$.\\
     
   In both cases, consider $Z_0 \in \C ^{n+p+n+p} \neq0$ such that $\Lambda Z_0=\lambda Z_0$, we have for all $t$, $Z_t=\lambda^t Z_0\xrightarrow[t\to \infty]{}0$. So for all $t\geq 0$,
  \[\|(x_t,y_t,x_{t-1},y_{t-1})- 0\| = \lambda_{\max}^t \|Z_0\|.\]
  Now, the initialization $Z'_0$ that we should consider should have all of its coefficient real, while $Z_0$ may not be in $\R^{n+p+n+p}$.
  Let us prove the existence of a real vector $Z'_0$ verifying the lemma.\\

  For $\lambda$ an eigenvalue of $\Lambda$, let us denote by $F_\lambda$ the kernel $\Ker(\Lambda-\lambda I)^{\nu(\lambda)}$ where $\nu(\lambda)$ is the algebraic multiplicity of $\lambda$.
  Then, $\C^{n+p+n+p} = \bigoplus_{\lambda \in \Sp(\Lambda)} F_\lambda$.
  Remark that if $\Lambda$ is diagonalizable, $F_\lambda = E_\lambda$ for all $\lambda$.
  Because $\Vect_\C(\R^{n+p+n+p}) = \C^{n+p+n+p}$, there exist a vector $Z'_0 \in \R^{n+p+n+p}$ such that its projection on $F_{\lambda_{\max}}$ along $\bigoplus_{\lambda \in \Sp(\Lambda)\backslash\{\lambda_{\max}\}} F_\lambda$ is nonzero.
  Let us call $z_{\lambda_{\max}}$ this projection.
  Then, because $\lambda_{\max}$ is the biggest eigenvalue strictly smaller than 1, there exits a constant $c>0$ with
  \[\norm{\Lambda^t(Z'_0 - Z'_\infty)} \sim \norm{\Lambda^t z_{\lambda_{\max}}} > c\lambda_{\max}^t\]
  proving the wanted result.

  % Now, because $\Lambda$ is a $2(n+p) \times 2(n+p)$ matrix, we now there exist a basis $(v_1, \ldots v_{2(n+p)})$ of generalized eigenvectors, and assume without loss of generality that $v_1$ is an eigenvector of $\Lambda$ associated with the eigenvalue $\lambda_{\max}$.
  % Remark that if $\Lambda$ is diagonalizable, all of the $(v_i)_{1 \leq i \leq 2(n+p)$ are eigenvectors of $\Lambda$.\\

  %   Now let us prove by contradiction that there exist a vector $Z'_0$ in $\R^{n+p+n+p}$ such that its projection on $\Vect_\C(v_1)$ along $\Vect_\C(v_2, \ldots, v_{2(n+p)})$ is nonzero.\\
  %   Indeed, otherwise, we would have for all $e_i$, $e_i \in \Vect_\C(v_2, \ldots, v_{2(n+p)})$ where $(e_i)$ is the canonical base, and thus $C^{n+p+n+p} \in \Vect_\C(v_2, \ldots, v_{2(n+p)})$, which would mean that $v_1 \in \Vect_\C(v_2, \ldots, v_{2(n+p)})$, which is a contradiction with the fact that $(v_i)_{1 \leq i \leq 2(n+p)$ is a basis.

  %     Thus, there exist a vector $Z'_0 \in \R^{n+p+n+p}$ such that $Z'_0 = \kappa v_1 + v$ with $\kappa \neq 0$ and $v \in \Vect_\C(v_2, \ldots, v_{2(n+p)})$, and we have $\Lambda^t\Z'_0 \sim

 \end{proof}

\subsubsection{Speed of convergence when $0<\eta< {\frac{1}{ \sqrt{3\mu_{\max}}}}$ and $\frac{1}{4\eta^2}\notin \S(A)$ }
We prove here part 2) of Theorem \ref{thm_conv_OGDA_xAy}. Recall from subsection \ref{CVOK} that 
for all $t\geq 1$, \[Z_t-Z_\infty=\sum_{\lambda \in Sp(\Lambda), \lambda \neq 0,1} \lambda^t z_{\lambda}=\sum_{\mu\in \S(A)), \mu>0} \sum_{l=1}^4 \lambda_{l(\mu)}^t z_{\lambda_l (\mu)}.\]

The matrix $\Lambda$ is not diagonalizable in an orthogonal basis, and the vectors $(z_\lambda)$ are not orthogonal in general. We will show however that the  situation is close to it.\\

From Lemma  \ref{lem_expression_eigenspaces}, we have  for  $\mu>0$ in  $\S(A)$, if $l=1,2$, $E_{\lambda_l(\mu)}=  $
$$ \left\{(x,y,x',y')\in \C^{n}\times \C^p \times \C^n\times \C^p,  x = \lambda_l(\mu) x' , y=\lambda_l(\mu) y', A^TAy'= \mu y' , x'= \frac{-i}{\sqrt{\mu}} Ay' \right\},$$
and if $l=3,4$, $E_{\lambda_l(\mu)}= $  $$ \left\{(x,y,x',y')\in \C^{n}\times \C^p \times \C^n\times \C^p,  x = \lambda_l(\mu) x' , y=\lambda_l(\mu) y', A^TAy'= \mu y' , x'= \frac{i}{\sqrt{\mu}} Ay' \right\}.$$
For the case where $\mu=0$, we have 
$E_0 = \left\{(0, 0,x',y') \in \C^{n}\times \C^p \times \C^n\times \C^p, Ay'= 0 , A^Tx'= 0 \right\}$ and 
and
$E_1 = \left\{(x', y',x',y') \in \C^{n}\times \C^p \times \C^n\times \C^p, Ay'= 0 , A^Tx'= 0 \right\}$.

\begin{lemma}$\;$
  \label{lem_eigenspace}

a) For $\mu_1\neq \mu_2$  in $\S(A)$, the vector subspaces $\bigoplus_{l=1,2,3,4} E_{\lambda_l(\mu_1)}$ and $\bigoplus_{l=1,2,3,4} E_{\lambda_l(\mu_2)}$ are orthogonal.

b) For $\mu$ in $\S(A)\backslash\{0\}$, $l\in \{1,2\}$ and $l'\in \{3,4\}$, the vector subspaces $E_{\lambda_l(\mu)}$ and $  E_{\lambda_{l'}(\mu)}$ are orthogonal.
\end{lemma}

\begin{proof}
  a) Consider $z_1=(x_1,y_1,x'_1,y'_1)$ in $E_{\lambda_l(\mu_1)}$ and $z_2=(x_2,y_2,x'_2,y'_2)$ in $E_{\lambda_{l'}(\mu_2)}$. We have both: 
  \[\langle Ay'_1,Ay'_2\rangle\; =\; \langle y'_1,A^TAy'_2\rangle\; =\; \langle y'_1,\mu_2y'_2\rangle\; =\; \mu_2\langle y'_1,y'_2\rangle,\]
  \[ \langle Ay'_1,Ay'_2\rangle\; =\; \langle A^TAy'_1,y'_2\rangle\; =\; \langle\mu_1y'_1,y'_2\rangle\; =\; \mu_1\langle y'_1,y'_2\rangle.\]
  Since $\mu_1\neq \mu_2$, $\langle y'_1,y'_2\rangle=0$. Then in all cases $\langle y_1,y_2\rangle=\langle x'_1,x'_2\rangle=\langle x_1,x_2\rangle=0$, and finally $\langle z_1,z_2\rangle=0$.
 
 b) Consider $z_1=(x_1,y_1,x'_1,y'_1)$ in $E_{\lambda_l(\mu)}$ and $z_2=(x_2,y_2,x'_2,y'_2)$ in $E_{\lambda_{l'}(\mu)}$.
 \begin{eqnarray*}
 \langle x'_1,x'_2\rangle&=  \langle\frac{-i}{\sqrt{\mu}} Ay'_1,\frac{i}{\sqrt{\mu}} Ay'_2 \rangle&=  -\frac{1}{\mu}\langle Ay'_1,Ay'_2\rangle\\ \; & = -\frac{1}{\mu}\langle y'_1,A^TAy'_2\rangle&= -\frac{1}{\mu}\langle y'_1,\mu y'_2\rangle= -\langle y'_1,y'_2\rangle.
 \end{eqnarray*}
 Then $\langle x_1,x_2\rangle=-\langle y_1,y_2\rangle$, and $\langle z_1,z_2\rangle=\langle x_1,x_2\rangle + \langle y_1,y_2\rangle+\langle x'_1,x'_2\rangle+\langle y'_1,y'_2\rangle=0.$
\end{proof}

\begin{rem} \rm
  
A consequence of the previous  lemma is that for $t\geq 1$: 
\[\|Z_t-Z_\infty\|^2=\sum_{\mu\in \S(A), \mu>0} \left( \|\lambda_{1}(\mu)^t z_{\lambda_1 (\mu)}+\lambda_{2}(\mu)^t z_{\lambda_2(\mu)}\|^2 + \|\lambda_{3}(\mu)^t z_{\lambda_3(\mu)}+\lambda_{4}(\mu)^t z_{\lambda_4(\mu)}\|^2\right).\]
We now fix $\mu>0$ in $\S(A)$ and simply write $\lambda_l$ for $\lambda_l(\mu)$ and $z_l$ for $z_{\lambda_l}$.  \\

We  need to handle the terms $ \|\lambda_{1}^t z_1+\lambda_{2}^t z_2\|^2$ and $ \|\lambda_{3} ^t z_3+\lambda_{4}^t z_{_4 }\|^2.$ Consider the first term (the second one will be treated similarly), we want to write $ \|\lambda_{1}^t z_{1 }+\lambda_{2}^t z_{2}\|^2\leq C |\lambda_1|^{2t} \|z_1+z_2\|^2$, where $C$ is a constant independent of the dimension.  This would be easy if $z_1$ and $z_2$ were orthogonal, and $C=1$ would  do. This would be impossible if $z_1+z_2=0$ with $z_1\neq 0$, but this can not happen since $z_1$ and $z_2$ belong to different eigenspaces. The key idea  is that the ``angle'' between the subspaces $E_{\lambda_1}$ and $E_{\lambda_2}$ can not be too low and is, unformally speaking, close to $\pi/4$: think of  $\lambda_1$ close to 1, and of $\lambda_2$ close to 0 so that for some vector subspace $F$,  $E_{\lambda_1}$  can be seen as the set of vectors $\{(\alpha, \alpha), \alpha \in F\}$ whereas $E_{\lambda_2}$  can be seen as the set of vectors $\{(0, \alpha), \alpha \in F\}$. If as a thought experiment we imagine $F$ being  the real line, we would have the horizontal axis and the 45-degree line in the Euclidean plane.\\

\end{rem}

\begin{notas} 
$$C(\mu)=\left\{ \begin{array} {ccc}  \sqrt{\frac{2}{1-\sqrt{\frac{1+5\eta^2\mu_{}}{2+ \eta^2\mu_{}}}}} & \mbox{if} & \eta \sqrt{\mu}<1/2,\\
\sqrt{\frac{2}{1-\sqrt{\frac{2+\eta^2\mu_{}}{1+ 5\eta^2\mu_{}}}}}&  \mbox{if} & \eta \sqrt{\mu}>1/2\end{array}\right.$$
\end{notas}

\begin{lemma}  $z_l$ being in $E_{\lambda_l(\mu)}$ for each $l=1,2,3,4$, we have:

 a) if  $\eta \sqrt{\mu}<1/2$, 
 $$\langle z_1,z_2\rangle | \leq \sqrt{\frac{1+5\eta^2\mu}{2+\eta^2\mu}} \; \|z_1\| \|z_2\|
  \; {\rm and}\;  |\langle z_3,z_4\rangle | \leq \sqrt{\frac{1+5\eta^2\mu}{2+\eta^2\mu}}\;  \|z_3\| \|z_4\|,$$
  and if $\eta \sqrt{\mu}>1/2$, 
  $$\langle z_1,z_2\rangle | \leq \sqrt{\frac{2+\eta^2\mu}{1+5\eta^2\mu}} \; \|z_1\| \|z_2\|
  \; {\rm and}\;  |\langle z_3,z_4\rangle | \leq \sqrt{\frac{2+\eta^2\mu}{1+5\eta^2\mu}}\;  \|z_3\| \|z_4\|,$$
  b) $ \|\lambda_1^t z_1+ \lambda_2^t z_2\|^2\leq C(\mu)^2 |\lambda_1|^{2t} \|z_1+z_2\|^2$ and $|\lambda_3^t z_1+ \lambda_4^t z_2\|^2\leq C(\mu)^2 |\lambda_1|^{2t} \|z_3+z_4\|^2.$

%  If  $ \eta \sqrt{\mu}<1/2$, 
% $ \|\lambda_1^t z_1+ \lambda_2^t z_2\|^2\leq C_*^2 |\lambda_1|^{2t} \|z_1+z_2\|^2$ and $|\lambda_3^t z_1+ \lambda_4^t z_2\|^2\leq C_*^2 |\lambda_1|^{2t} \|z_3+z_4\|^2.$  
%  
%If $  \eta \sqrt{\mu}>1/2,$  $ \|\lambda_1^t z_1+ \lambda_2^t z_2\|^2\leq C_{**}^2 |\lambda_1|^{2t} \|z_1+z_2\|^2 $ and $ |\lambda_3^t z_1+ \lambda_4^t z_2\|^2\leq C_{**}^2 |\lambda_1|^{2t} \|z_3+z_4\|^2.$
% 
 
 \end{lemma}

\begin{proof}
a)  Write $z_1=(x_1,y_1,x'_1,y'_1)$ and $z_2=(x_2,y_2,x'_2,y'_2)$.

\begin{eqnarray*} \langle x'_1,x'_2\rangle&= \langle\frac{-i}{\sqrt{\mu}} Ay'_1,\frac{-i}{\sqrt{\mu}} Ay'_2\rangle&= \frac{1}{\mu}\langle Ay'_1,Ay'_2\rangle\\
\; &=\frac{1}{\mu}\langle y'_1,A^TAy'_2\rangle&=\frac{1}{\mu}\langle y'_1,\mu y'_2\rangle= \langle y'_1,y'_2\rangle,\end{eqnarray*}
 
 and $\langle x_1,x_2\rangle= \lambda_1\overline{\lambda}_2\langle x'_1,x'_2\rangle= \lambda_1\overline{\lambda}_2\langle y'_1,y'_2\rangle=\langle y_1,y_2\rangle.$ So \[\langle z_1,z_2\rangle=2(1+\lambda_1\overline{\lambda}_2)\langle y'_1,y'_2\rangle,\] and similarly 
one can show that  $\|z_1\|^2=2(1+|\lambda_1|^2) \|y'_1\|^2$  {\rm and} $ \|z_2\|^2=2(1+|\lambda_2|^2) \|y'_2\|^2.$
 We obtain:
 \[\langle z_1,z_2\rangle = \frac{(1+ \lambda_1 \overline{\lambda}_2)}{\sqrt{(1+|\lambda_1|^2) (1+|\lambda_2|^2)}}\;\; \frac{\langle y'_1,y'_2\rangle}{\|y'_1\|\|y'_2\|} \;\;\|z_1\| \|z_2\|.\]
 Computations show that:
 
 If $\eta \sqrt{\mu}<1/2$, then  ${\sqrt{(1+|\lambda_1|^2) (1+|\lambda_2|^2)}}=\sqrt{2+\eta^2\mu}$, and $|1+ \lambda_1 \overline{\lambda}_2|=\sqrt{1+5\eta^2\mu}.$ Since $ \frac{|\langle y'_1,y'_2\rangle |}{\|y'_1\|\|y'_2\|}\leq 1$ by Cauchy-Schwartz inequality, we obtain the upper bound  for $|\langle z_1,z_2\rangle|$. The proof is similar for $|\langle z_3,z_4\rangle|$.\\
 
 If $\eta \sqrt{\mu}>1/2$, then  ${\sqrt{(1+|\lambda_1|^2) (1+|\lambda_2|^2)}}=\sqrt{1+5\eta^2\mu}$, and $|1+ \lambda_1 \overline{\lambda}_2|=\sqrt{2+\eta^2\mu}.$ We obtain the upper bound  for $|\langle z_1,z_2\rangle|$, and the proof is similar for $|\langle z_3,z_4\rangle|$.\\
 
b) On the one-hand, 
\[\|\lambda_1^t z_1+ \lambda_2^t z_2\|^2\leq 2  \left(\|\lambda_1^t z_1\|^2+ \|\lambda_2^t z_2\|^2\right)\leq 2|\lambda_1|^{2t} \left( \|z_1\|^2 + \|z_2\|^2\right).\]
On the other hand, if $\eta \sqrt{\mu}<1/2$,
\begin{eqnarray*}\|z_1+z_2\|^2&=&\|z_1\|^2+\|z_2\|^2+ 2\; {\cal R}e(\langle z_1,z_2\rangle),\\
 &\geq &\|z_1\|^2+\|z_2\|^2 - 2 \sqrt{\frac{1+5\eta^2\mu}{2+\eta^2\mu}}\norm{z_1}\norm{z_2},\\
&\geq &\|z_1\|^2+\|z_2\|^2 - \sqrt{\frac{1+5\eta^2\mu}{2+\eta^2\mu}}(\|z_1\|^2+\|z_2\|^2),\\
&\geq &(\|z_1\|^2+\|z_2\|^2) \left(1-\sqrt{\frac{1+5\eta^2\mu}{2+\eta^2\mu}}\right).\end{eqnarray*}
Notice that $\eta^2\mu<1/4$ implies $1-\sqrt{\frac{1+5\eta^2\mu}{2+\eta^2\mu}}>0$. We obtain:
\[\|\lambda_1^t z_1+ \lambda_2^t z_2\|^2\leq C(\mu)^2 |\lambda_1|^{2t} \|z_1+z_2\|^2.\]
And if $\eta \sqrt{\mu}>1/2$,
\begin{eqnarray*}\|z_1+z_2\|^2&=&\|z_1\|^2+\|z_2\|^2+ 2\; {\cal R}e(\langle z_1,z_2\rangle),\\
 &\geq &\|z_1\|^2+\|z_2\|^2 - 2 \sqrt{\frac{2+\eta^2\mu}{1+5\eta^2\mu}}\norm{z_1}\norm{z_2},\\
&\geq &\|z_1\|^2+\|z_2\|^2 - \sqrt{\frac{2+\eta^2\mu}{1+5\eta^2\mu}}(\|z_1\|^2+\|z_2\|^2),\\
&\geq &(\|z_1\|^2+\|z_2\|^2) \left(1-\sqrt{\frac{2+\eta^2\mu}{1+5\eta^2\mu}}\right).\end{eqnarray*}
Notice that $\eta^2\mu>1/4$ implies $1-\sqrt{\frac{2+\eta^2\mu}{1+5\eta^2\mu}}>0$. We obtain:
\[\|\lambda_1^t z_1+ \lambda_2^t z_2\|^2\leq C(\mu)^2 |\lambda_1|^{2t} \|z_1+z_2\|^2.\]

The remaining inequalities  are proved similarly.
\end{proof}

\begin{notas} $\;$

 $C_*={\bf{1}}_{\mu_{\min} < \frac{1}{4\eta^2}} \sqrt{\frac{2}{1-\sqrt{\frac{1+5\eta^2\mu_{*}}{2+ \eta^2\mu_{*}}}}}$, with $\mu_{*}=\max\{\mu \in \S(A), \eta \sqrt{\mu}<1/2\}$, and 
 
$C_{**}={\bf{1}}_{\mu_{\max} \geq  \frac{1}{4\eta^2}} \sqrt{\frac{2}{1-\sqrt{\frac{2+\eta^2\mu_{**}}{1+ 5\eta^2\mu_{**}}}}}$, with $\mu_{**}=\min\{\mu \in \S(A), \eta \sqrt{\mu}>1/2\}.$

$C=\max\{C_*,C_{**}\}$. 
\end{notas}

We can now bound  $\|Z_t-Z_\infty\|^2$.\\
 \begin{eqnarray*}
\|Z_t-Z_\infty\|^2&=&\sum_{\mu\in \S(A), \mu>0} \left( \|\lambda_{1}(\mu)^t z_{\lambda_1 (\mu)}+\lambda_{2}(\mu)^t z_{\lambda_2(\mu)}\|^2 + \|\lambda_{3}(\mu)^t z_{\lambda_3(\mu)}+\lambda_{4}(\mu)^t z_{\lambda_4(\mu)}\|^2\right),\\
\; &\leq &\sum_{\mu\in \S(A), \mu>0} C(\mu) ^2 |\lambda_{1}(\mu)|^{2t} \left( \|  z_{\lambda_1 (\mu)}+  z_{\lambda_2(\mu)}\|^2 + \|  z_{\lambda_3(\mu)}+  z_{\lambda_4(\mu)}\|^2\right),\\ \;&\leq &C^2 \lambda_{\max}^{2t} \|Z_0\|^2.\end{eqnarray*}
 
And we get: 
\begin{equation}
  \label{eq_exp_rate}
  \|Z_t-Z_\infty\|\leq C\lambda_{\max}^{t}\|Z_0\|.
\end{equation}
\vspace{0.5cm}

We can now conclude part 2) of Theorem \ref{thm_conv_OGDA_xAy}. Take any Nash equilibrium $(x^*,y^*)$ in $\Ker(A^T)\times \Ker(A)$, and define for each $t\geq -1$:  \[x'_t=x_t-x^*\; {\rm and}\; y'_t=y_t-y^*.\]
The sequence $(x'_t,y'_t)$ is induced by  the OGDA starting at  $(x'_0,y'_0,x'_{-1},y'_{-1})$.
 $(x'_t,y'_t)_t$ converges to the  limit $(x'_\infty,y'_\infty)= (x_\infty,y_\infty)-(x^*,y^*)$, so  
$\|(x'_\infty,y'_\infty)-(x'_t,y'_t)\|=\|(x_\infty,y_\infty)-(x_t,y_t)\|.$
Using  inequality (\ref{eq_exp_rate}), we obtain:
$$\|(x_\infty,y_\infty)-(x_t,y_t)\|\leq C \lambda_{\max}^{t}\|Z'_0\|,$$ with $\|Z'_0\|=\|(x_0-x^*,y_0-y^*, x'_{-1}-x^*,y'_{-1}-y^*)\|.$ Letting $(x^*,y^*)$ vary in $\Ker(A^T)\times \Ker(A)$ concludes the proof of part 2) of Theorem \ref{thm_conv_OGDA_xAy}:  For all $t\geq 0$,
\[ \|(x_t,y_t)-(x_\infty,y_\infty)\|\leq  C \; D\;  \lambda_{\rm max}^t  ,\;\]
$\text{with}\; C=\sqrt{\frac{2}{1-\sqrt{\frac{1+5\eta^2\mu_{\max}}{2+ \eta^2\mu_{\max}}}}},
$
\noindent  $D$ is the distance from the initial condition  to the set $\Ker(\Lambda-I)$, and $\lambda_{\max}=\sqrt{\frac{1}{2}(1+\sqrt{1-4\eta^2\mu_{\min}})}.$\qed

\subsection{Proof of Theorem \ref{thm_conv_OGDA_zs_general}}
\label{subsec_proof_thm_conv_OGDA_zs_general}
\begin{proof}
 Notice that the fixed points of the OGDA are exactly the Nash equilibria of the game:
  \[  \left\{
  \begin{array} {ccc} 
  x^* & = & x^* \;+\; 2\eta (A y^*+b) - \eta (A y^*+b)  \\
  y^* & = & y^* - 2\eta (A^T x^*+c) + \eta (A^T x^*+c) \end{array}
  \right. \iff \left\{
  \begin{array}{cc}
  Ay^*+b& =0\\
  A^Tx^*+c&=0
  \end{array}\right.
  \]
  Thus, saying that the set of Nash equilibria is empty is equivalent to saying  that OGDA has no fixed point: if there is no Nash equilibria, OGDA diverge.\\

  In the sequel we assume that  the set of Nash equilibria is not empty, and fix a Nash equilibrium ($x^*,y^*)$. Define for each $t\geq -1$:  \[x'_t=x_t-x^*\; {\rm and}\; y'_t=y_t-y^*.\] Easy calculations show that:
  \[\forall t\geq 0,\;  \left\{
  \begin{array} {ccc} 
  x'_{t+1} & = & x'_t  \;+\; 2\eta Ay'_t - \eta Ay'_{t-1},  \\
  y'_{t+1}  & = & y'_t \;  - 2\eta A^T x'_t + \eta A^Tx'_{t-1}.
  \end{array}\right. \]
  So  the sequence $(x'_t,  y'_t)_t$ follows the OGDA algorithm of section \ref{section_bilinear} (see \ref{OGDA_algo}). Thus, according to Theorem \ref{thm_conv_OGDA_xAy}, the sequence  $(x'_t,  y'_t)_t$ converges to a Nash equilibrium $(x'_\infty,y'_\infty) \in \Ker(A^T)\times \Ker(A)$. As a consequence $(x_t,y_t)_t$ converges to the limit $(x_\infty,y_\infty)=(x'_\infty + x^*,y'_\infty+y^*)$ which is  a fixed point of the OGDA, hence a Nash equilibrium of the game. \\
  
  Denote by $\Pi$ the orthogonal projection onto $\Ker(A^T) \times \Ker(A)$.
  According to Theorem
  \ref{thm_conv_OGDA_xAy}, we have $(x'_\infty,y'_\infty)=\Pi (x'_0,y'_0)$, so:  
\[(x_\infty,y_\infty)= (x^*,y^*)+ \Pi(x_0,y_0) -\Pi(x^*,y^*) =  (x^*,y^*)+ \Pi\left((x_0,y_0) -(x^*,y^*)\right).\]
Thus $(x_\infty,y_\infty)$ is the orthogonal projection of $(x_0, y_0)$ onto the affine space $\Ker(A^T)\times \Ker(A) + (x^*, y^*)$, which is equal to ${\{ (x,y) | A^Tx+c=0, Ay+b = 0\}}$.\\

It remains to study the convergence speed. For all $t$, we have $\|(x_t,y_t)-(x_\infty,y_\infty)\|= \|(x'_t,y'_t)-(x'_\infty,y'_\infty)\|$ and we obtain from Theorem \ref{thm_conv_OGDA_xAy} the same speeds of convergence. If  $\frac{1}{4\eta^2}\notin \S(A)$  we have for all $t\geq 0$,
\[ \|(x_t,y_t)-(x_\infty,y_\infty)\|\leq  C \; D'\;  \lambda_{\rm max}^t, \] 
with  $C$ and $\lambda_{\max}$ as in Theorem \ref{thm_conv_OGDA_xAy},  and $D'$ the distance from $(x'_0,  y'_0, x'_{-1}, y'_{-1})$ to the set $\{(x,y,x,y)\in \R^n\times \R^p\times \R^n\times \R^p, A^T=0, Ay=0\},$ and 
 
Let us finally  notice that $D'\leq \|(x'_0,y'_0,x'_{-1},y'_{-1})\|$ $=$ $\|(x_0,y_0,x_{-1},y_{-1})-(x^*,y^*,x^*,y^*)\|$.  Since this is true for all Nash equilibria $(x^*,y^*)$, we obtain that $D'$ is not greater  than the distance from $(x_0,  y_0, x_{-1}, y_{-1})$ to the set $\{(x,y,x,y), A^Tx+c=0, Ay+b=0\}.$ 
\end{proof}

\subsection{Proof of Proposition \ref{prop_spectrum_xAy_xBy}}
\label{subsec_proof_prop_spectrum_xAy_xBy}

\begin{proof}
 For  $\lambda\in \C$ and $Z =(x^T,y^T,x'^T,y'^T)^T$ $\in \C^n\times \C^p\times \C^n\times \C^p$, we have: 
  \begin{align*}
  \Lambda Z = \lambda Z
  &\iff \left\{ \begin{matrix} x + 2\eta Ay - \eta Ay' = \lambda x\\ y +2\eta B^Tx - \eta B^Tx' = \lambda y \\ x = \lambda x' \\ y = \lambda y' \end{matrix}\right.\\
  &\iff \left\{ \begin{matrix} \lambda(1-\lambda)x' + (2\lambda -1)\eta Ay' = 0\\ \lambda(1-\lambda)y' + (2\lambda-1)\eta B^Tx' = 0 \\ x = \lambda x' \\ y = \lambda y' \end{matrix}\right.\\
  \end{align*}
  Assume $\Lambda Z=\lambda Z$ and $Z\neq 0$. Multiplying (from the left) the first line by $(2\lambda-1)\eta B^T$ and the second line by $(2\lambda-1)\eta A$, we get:
  \begin{align}\left\{ \begin{matrix}\lambda^2(1-\lambda)^2y' - (2\lambda -1)^2\eta^2 B^TAy' = 0\\ \lambda^2(1-\lambda)^2x' - (2\lambda -1)^2\eta^2 AB^Tx' = 0\end{matrix}\right. \end{align}\\
  Since $Z\neq 0$, we have $x'\neq 0$ or $y'\neq 0$, and $\lambda \neq 1/2$. This implies that $\frac{\lambda^2(1-\lambda)^2}{(2\lambda-1)^2\eta^2}$ is an eigenvalue  of $B^TA$ or  of $AB^T$. If we denote by $\mu$ such eigenvalue, then $\lambda\in S^*(\mu)$.
  This gives us the first inclusion.\\

  Conversely, let $\mu$ be an eigenvalue of $AB^T$, and consider $\lambda\in S^*(\mu)$.\\
  If $\mu \neq 0$, let $x'\neq 0$ be  such that $A B^T x'=\mu x'$ and consider $y'$ satisfying:
  \[\lambda(1- \lambda) y'=- \eta (2\lambda-1) B^Tx'.\]
  Because  $\mu\neq 0$, $\lambda\notin \{0,1\}$, and thus $y'$ is uniquely defined.
  Setting $Z_\lambda = (\lambda x', \lambda y', x', y')$ gives us an eigenvector of $\Lambda_{A, B}$ associated with $\lambda$.\\
  If $\mu=0$ then $A B^T x'=0.$
  If $B^Tx'=0$, then setting $Z_0 = (0, 0, x', 0)$ and $Z_1 = (x', 0, x', 0)$ will give us an eigenvector of $\Lambda_{A, B}$ associated with $0$ and $1$ respectively.
  Otherwise, if $B^Tx' \neq 0$, then we set $Z_0 = (0, 0, 0, B^Tx')$ and $Z_1 = (0, B^T x', 0, B^Tx')$ to have an eigenvector of $\Lambda_{A, B}$ associated with $0$ and $1$ respectively.
  We have obtained  that for each eigenvalue $\mu$ of $AB^T$, each $\lambda \in \S^*(\mu)$ is an eigenvalue of $\Lambda$. 
  Similarly, this also stands for all $\mu$ an eigenvalue of $B^TA$.\\

  Thus,  we proved that the set of eigenvalues of  $\Lambda$ is $\{ \lambda | \lambda \in S^*(\mu), \mu \in \S(A, B)\}$. 
  % Moreover, with the computations done in this proof, we can have the following formulas for the eigenspaces $E_\lambda$:\\
% If $\lambda \neq 0, 1$, then:\\
% \begin{align*}
% E_\lambda &= \left\{(x, y, x', y') | x=\lambda x', y = \lambda y', B^TA y' = \mu y', x' = \frac{(1-2\lambda)\eta A y'}{\lambda(1-\lambda)} \right\}\\
%   &= \left\{(x, y, x', y') | x=\lambda x', y = \lambda y', AB^T x' = \mu x', y' = \frac{(1-2\lambda)\eta B^T x'}{\lambda(1-\lambda)} \right\}
% \end{align*}
% and
% \[ E_0 = \left\{(0, 0, x', y') | x' \in \Ker(B^T), y' \in \Ker(A) \right\} \]
% \[ E_1 = \left\{(x', y', x', y') | x' \in \Ker(B^T), y' \in \Ker(A) \right\} \]
\end{proof}

\subsection{Proof of Theorem \ref{prop_conv_OGDA_xAy_xBy}}
 Let $A, B \in \R^{n\times p}$. The  proof of Proposition \ref{prop_spectrum_xAy_xBy} also provides useful expressions for the eigenspaces $E_\lambda=\Ker(\Lambda_{A,B}-\lambda I_{2(n+p)})$:  

 For  $\lambda \neq 0, 1, $
 \begin{align*}
 E_\lambda &= \left\{(x, y, x', y')\in \C^{n+p+n+p} | x=\lambda x', y = \lambda y', B^TA y' = \mu y', x' = \frac{(1-2\lambda)\eta A y'}{\lambda(1-\lambda)} \right\}\\
%  &= \left\{(x, y, x', y')\in \C^{n+p+n+p} | x=\lambda x', y = \lambda y', AB^T x' = \mu x', y' = \frac{(1-2\lambda)\eta B^T x'}{\lambda(1-\lambda)} \right\}
 \end{align*}
  and $E_0 = \left\{(0, 0, x', y')\in \C^{n+p+n+p} | x' \in \Ker(B^T), y' \in \Ker(A) \right\},$
  
  $E_1 = \left\{(x', y', x', y')\in \C^{n+p+n+p} | x' \in \Ker(B^T), y' \in \Ker(A) \right\}.$

 \vspace{0.5cm}
We start with some  lemmas. The first one follows from simple computations. 
% \begin{rem}
% We recall that for any matrix A, B, we have $(\Ker(A^T))^\perp = \Img(A)$ and $(\Ker(B))^\perp = \Img(B^T)$.
% \end{rem}

\begin{lemma}
  \label{lem_valeigs}
 Assume that $\S(A, B) \subset \R$, and  $\eta<\frac{1}{2 \sqrt{\mu_{\max}}}$, where $\mu_{\max}=\rho(B^TA) = \rho(AB^T)$.
  Then, for ${\mu \in \S(A, B)}$, there are three cases:\\
  - If $\mu<0$, define $\nu=i \sqrt{\abs{\mu}}$ and $\delta=\sqrt{1-4\eta^2\abs{\mu}}\geq 0$.
  Then $S^*(\mu)$ has exactly 4 elements, which are $\lambda^*_1(\mu)=\frac{1}{2}  (1+ \delta + 2 \eta \nu)$,  
  $\lambda^*_2(\mu)= \frac{1}{2}  (1- \delta + 2 \eta \nu),$
  $ \lambda^*_3(\mu)=\overline{\lambda^*_1}(\mu)$, 
  $\lambda^*_4(\mu)=\overline{\lambda^*_2}(\mu).$ 
  Their modulus are all strictly smaller than $1$.\\
  - If $\mu=0$, then $S^*(\mu)$ possess two elements: $S^*(\mu) = \{0, 1\}$.\\
  - If $\mu>0$, define $\nu=\sqrt{\mu}$ and $\delta=\sqrt{1+4\eta^2\nu^2}$. Then $S^*(\mu)$ has exactly 4 elements which are real, denoted: $\lambda^*_1(\mu)= \frac{1}{2} (1+2\eta \nu + \delta)>1$, $\lambda^*_2(\mu)= \frac{1}{2} (1+2\eta \nu - \delta)\in(0,1)$, $\lambda^*_3(\mu)= \frac{1}{2} (1-2\eta \nu + \delta)\in(0,1)$ and $\lambda^*_4(\mu)= \frac{1}{2} (1-2\eta \nu - \delta)\in(-1,0)$.
\end{lemma}

\begin{lemma}
  \label{lem_ortho_general}
  If $y$ is  an eigenvector of $B^TA$ associated to a nonzero eigenvalue, then $y\in \Img(B^T)$.
  If $x$ is  an eigenvector of $AB^T$ associated to a nonzero eigenvalue, then $x\in \Img(A)$.
\end{lemma}

\begin{proof}
  Let $\hat{y}$ be a vector in the kernel of $B$, and let $\mu \neq 0$ be the eigenvalue of $B^TA$ associated with $y$.
  Then:
  \[ \mu \langle y, \hat{y} \rangle = \langle  \mu y, \hat{y}\rangle = \langle B^TA y , \hat{y}\rangle = \langle Ay, B\hat{y} \rangle = \langle Ay, 0 \rangle = 0\]
  Hence, $\langle y, \hat{y} \rangle = 0$ for any $\hat{y}$ in $\Ker(B)$.
  This means that $y$ is in $(\Ker(B))^\perp = \Img(B^T)$.\\
  The proof for the second part of the lemma is similar.
\end{proof}

\begin{lemma}
  \label{lem_dsum}
  If $\Lambda_{A, B}$ is diagonalizable, then $\Ker(B^T)$ is an algebraic complement of $\Ker(A^T)^\perp = \Img(A)$ and $\Ker(A)$ is an algebraic complement of  $\Ker(B)^\perp = \Img(B^T)$.
\end{lemma}

\begin{proof}
  We will first prove by contraposition that $\Ker(B^T)$ and $\Img(A)$ are in direct sum.\\
  Assume in a first time that $\Ker(B^T)$ and $\Img(A)$ are not in direct sum.
  Then, there exist $\tilde{x} \in (\Ker(B^T) \cap \Img(A)) \backslash \{0\}$.
  Let $\tilde{y} \in \C^p$ such that $\tilde{x} = A\tilde{y}$.
  Then $\tilde{y} \notin \Ker(A)$ but $\tilde{y} \in \Ker(B^TA)$, because $B^TA\tilde{y} = B^T\tilde{x} = 0$ by definition of $\tilde{x}$.
  This implies that $\dim\Ker(B^TA) > \dim\Ker(A)$, and by the rank-nullity theorem, that $\rank(B^TA) < \rank(A)$.

  Thus,
  \begin{align*}
  \sum_{\lambda \in Sp(\Lambda_{A, B})} \dim(E_\lambda)
  &= \dim(E_0) + \dim(E_1) +  \sum_{\lambda \in Sp(\Lambda_{A, B}), \, \lambda\notin \{0,1\}} \dim(E_\lambda)  \\
  &= 2\dim(\Ker(A))+2\dim(\Ker(B^T)) + \sum_{\mu \in Sp(B^TA), \, \mu \neq 0} 4\dim(\Ker(B^TA-\mu I)) \\
  &\leq 2\dim(\Ker(A))+2\dim(\Ker(B^T))+ 4\rank(B^TA) \\
  &< 2\dim(\Ker(A))) +2\dim(\Ker(B^T)) + 2(\rank(A) + \rank(B^T)) \\
  &= 2(n+p).
  \end{align*}
  This shows that $\Lambda_{A, B}$ is not diagonalizable.\\
  By contraposition, we have shown that $\Ker(B^T)$ and $\Img(A)$ are in direct sum.
  Similarly $\Ker(A)$ and $\Img(B^T)$ are also in direct sum.\\

  Now, let us show that $\Ker(B^T)$ and $\Img(A)$ are complementary and that $\Ker(A)$ and $\Img(B^T)$ are complementary.\\
  Because $\Ker(B^T)$ and $\Img(A)$ are subset of $\C^n$, $\dim(\Ker(B^T)) + \dim(\Img(A)) \leq n$, and similarly, $\dim(\Ker(A)) + \dim(\Img(B^T)) \leq p$.
  By the rank-nullity theorem, we also know that $\dim(\Ker(A)) + \dim(\Img(A)) = p$ and $\dim(\Ker(B^T)) + \dim(\Img(B^T)) = n$.
  In order to have equality here, we need to have equality too in the previous inequalities.
  Thus $\dim(\Ker(B^T)) + \dim(\Img(A)) = n$ and $\dim(\Ker(A)) + \dim(\Img(B^T)) = p$.
  It implies that $\Ker(B^T) \oplus \Img(A) = \C^n$ and $\Ker(A) \oplus \Img(B^T) = \C^p$.
\end{proof}

 Assume that $\S(A,B)\subset \R_-$. Then, using Proposition \ref{prop_spectrum_xAy_xBy} and Lemma \ref{lem_valeigs}, we deduce  that for $\eta<\frac{1}{2 \, \sqrt{\mu_{\max}}}$ we have  $\Sp(\Lambda_{A,B})\subset \{1\}\cup \{\lambda \in \C, | \lambda|<1\}.$\\
 
 a) Assume $A$ and $B$ are square invertible. Then $\Sp(AB^T) = \Sp(B^TA)\subset \{\mu \in \R, \mu<0\}$, and for $\eta>0$ small enough $\Sp(\Lambda_{A,B})\subset \{\lambda \in \C, | \lambda|<1\}$.
 % Then $\Lambda_{A,B}^t\xrightarrow[t\to \infty]{} 0$ by Gelfand's theorem, so $(x_t,y_t)\xrightarrow[t\to \infty]{}(0,0)$. The result follows since $\Ker(A)=\Ker(B^T)=\{0\}$.\\
 Then $\norm{\Lambda_{A,B}^t}^{\frac{1}{t}}\xrightarrow[t\to \infty]{} \rho(\Lambda_{A, B}) < 1$ by Gelfand's theorem, so $\norm{\Lambda_{A, B}^t Z_0}$ is of order $\rho(\Lambda_{A, B})^t \norm{Z_0}$.
It proves that $(x_t,y_t)$ tends to $(0,0)$ exponentially fast as $t$ tends to infinity. The result follows since $\Ker(A)=\Ker(B^T)=\{0\}$.\\
 
 b) Let $Z_0 = (x_0, y_0, x_{-1}, y_{-1})$ and $Z_t = (x_t, y_t, x_{t-1}, y_{t-1}) = \Lambda_{A,B}^t Z_0$.
 
 Assume $\Lambda_{A,B}$ is diagonalizable. Since, in addition its spectrum is  included in $\{1\}\cup \{\lambda \in \C, | \lambda|<1\},$ the sequence $(\Lambda^t_{A,B}Z_0)_t$ converges exponentially fast, with speed $\max \{\abs{\lambda} | \lambda \in \Sp(\Lambda_{A, B}), \lambda \neq 1\}$, to the projection of $Z_0$ onto the eigenspace associated to the eigenvalue $1$, along the direct sum of all the other eigenspaces.
 Moreover, $(x_\infty, y_\infty, x_\infty, y_\infty) \in \Ker(\Lambda_{A,B} -I)$ implies that $(x_\infty, y_\infty)$ is in $\Ker(B^T) \times \Ker(A)$.
 This proves that $(x_t,y_t)_t$ converges exponentially fast to a Nash equilibrium.\\

 Since $\Lambda_{A, B}$ is diagonalizable,  $$\C^{(n+p+n+p)} = \bigoplus_{\lambda \in \Sp(\Lambda_{A, B})} \Ker(\Lambda_{A, B} - \lambda I).$$
  Let $Z_0 = (x_0, y_0,x_{-1},y_{-1})$ denote the initialization vector.
  We can decompose $Z_0$ as $\sum_{\lambda \in \Sp(\Lambda_{A, B})} Z_\lambda$ where each $Z_\lambda \in E_\lambda$ is itself decomposed as $Z_\lambda = (\hat{x}_\lambda, \hat{y}_\lambda, \tilde{x}_\lambda, \tilde{y}_\lambda)$.
  This implies that $x_0 = \sum_{\lambda \in \Sp(\Lambda_{A, B})} \hat{x}_\lambda$ and that $y_0 = \sum_{\lambda \in \Sp(\Lambda_{A, B})} \hat{y}_\lambda$.\\
  Yet, from the expression of the eigenspaces $E_\lambda$ seen in the proof of Proposition \ref{prop_spectrum_xAy_xBy}, we know that $\hat{x}_0 = 0$, $\hat{y}_0 = 0$, $\hat{x}_1 \in \Ker(B^T)$ and $\hat{y}_1 \in \Ker(A)$.
  Moreover, for each $\lambda\notin \{0,1\}$, $\hat{x}_\lambda$ is an eigenvector of $AB^T$ associated to a non-zero eigenvalue, and similarly $\hat{y}_\lambda$ is an eigenvector of $B^TA$ associated to a non-zero eigenvalue. Then Lemma \ref{lem_ortho_general} gives  that $\sum_{\lambda \in \Sp(\Lambda_{A, B})\backslash\{0, 1\}} \hat{x}_\lambda \in \Img(A)$ and $\sum_{\lambda \in \Sp(\Lambda_{A, B})\backslash\{0, 1\}} \hat{y}_\lambda \in \Img(B^T)$.\\
  According to Lemma \ref{lem_dsum},  $\Ker(B^T)$ and $\Img(A)$ are in direct sum, and thus $x_\infty = \hat{x}_1$ is the linear projection of $x_0$ onto $\Ker(B^T)$ along $\Img(A)$.
  Once again from Lemma \ref{lem_dsum}, $\Ker(A)$ and $\Img(B^T)$ are in direct sum, and thus $y_\infty = \hat{y}_1$ is the linear projection of $y_0$ onto $\Ker(A)$ along $\Img(B^T)$.\\
  
  In both cases a) and b), the convergence is exponential with rate given by $\max\{|\lambda|, \lambda \in \Sp(\Lambda_{A,B}), \lambda \neq 1\}$ that is $$|\lambda_1^*(\mu)|^2=\max\left\{\sqrt{\frac{1}{2}(1+\sqrt{1+4\eta^2 \mu})}, \mu<0, \mu \in \S(A, B) \right\}.$$
\qed

\subsection{Proof of Theorem \ref{thm_conv_OGDA_general_xAy_xBy}}
\begin{proof}
  We proceed as in the proof of Theorem \ref{thm_conv_OGDA_zs_general}. Fix a Nash equilibrium ($x^*,y^*)$. Define for each $t\geq -1$, $x'_t=x_t-x^*\; {\rm and}\; y'_t=y_t-y^*.$ Easy calculations show that:
  \[\forall t\geq 0,\;  \left\{
  \begin{array} {ccc} 
  x'_{t+1} & = & x'_t  \;+\; 2\eta Ay'_t - \eta Ay'_{t-1},  \\
  y'_{t+1}  & = & y'_t \;  - 2\eta B^T x'_t + \eta B^Tx'_{t-1}.
  \end{array}\right. \]
  So  the sequence $(x'_t,  y'_t)_t$ follows the OGDA algorithm \eqref{OGDA_nzs} of section \ref{subsec_xAy_xBy}. Thus, according to Theorem \ref{prop_conv_OGDA_xAy_xBy}, the sequence  $(x'_t,  y'_t)_t$ converges to the  limit  $(x'_\infty,y'_\infty)$, where  $x'_\infty$ is the projection of $x'_0$ on $\Ker(B^T)$ along ${\Ker(A^T)}^\perp$ and $y'_\infty$ is the projection of $y'_0$ on $\Ker(A)$ along ${\Ker(B)}^\perp$. As a consequence $(x_t,y_t)_t$ converges to the limit $(x_\infty,y_\infty)=(x'_\infty + x^*,y'_\infty+y^*)$ which is  a fixed point of the OGDA, hence a Nash equilibrium of the game. Moreover $x_\infty-x_0= x'_\infty-x'_0 \in {\Ker(A^T)}^\perp$, so $x_\infty$ is the projection of $x_0$ on $\{x\in \R^n, B^Tx+f=0\}$ along ${\Ker(A^T)}^\perp$. Similarly $y_\infty$ is the projection of $y_0$ on $\{y\in \R^p, Ay+b=0\}$ along ${\Ker(B)}^\perp$.
 
  And  part {\it 3)} of Theorem \ref{thm_conv_OGDA_general_xAy_xBy} holds since $\|(x_t,y_t)-(x_\infty, y_\infty)\|=\|(x'_t,y'_t)-(x'_\infty, y'_\infty)\|$. 
\end{proof}

\subsection{Proof of Lemma \ref{lemmepseudo}}

\begin{proof}
$AA^\dagger$ is the matrix of the orthogonal projection onto $\Img(A)=(\Ker(A^T))^\perp$. Since $AB^T=-AA^\dagger$, its spectrum is included in $\{0,-1\}$. Similarly, 
$A^\dagger A$ is the matrix of the orthogonal projection onto $\Img(A^T)=(\Ker(A))^\perp$. We obtain $\S(A,B)\subset\{0,-1\}\subset \R_-$.

  $\mu=0$ corresponds to the possible eigenvalues $0$ and $1$ of $\Lambda_{A,B}$, and  we have $\dim(E_0)=\dim(E_1)=\dim(\Ker(A))+\dim(\Ker(A^\dagger)$.

Regarding $\mu=-1$, for each $\lambda\in S^*(-1)$ we have $\dim E_\lambda=\dim (\Ker(I_n-AA^\dagger))=\dim (\Ker(I_p- A^\dagger A))$. $I_n-AA^\dagger$ is the matrix of the orthogonal projection on $\Ker(A^T)$, hence $\dim (\Ker(I_n-AA^\dagger))=n-\dim(\Ker(A^T))$. $I_p-A^\dagger A$ is the matrix of the orthogonal projection on $\Ker(A)$, hence $\dim (\Ker(I_p-A^\dagger A))=p-\dim(\Ker(A))$. \\

Assume $-1\in \S(A,B)$. Then $\mu_{\max}=1$ and for $\eta<\frac12$, $S^*(-1)$ has exactly 4 elements (see Lemma \ref{lem_valeigs}), and the sum of the dimensions of the eigenspaces of $\Lambda_{A,B}$ is:
$$2 \dim(\Ker (A))+ 2\dim (\Ker(A^T))+ 2(n-\dim(\Ker(A^T))+2(p-\dim(\Ker(A))=2(n+p),$$
we obtain that $\Lambda_{A,B}$ is diagonalizable.

Assume $\S(A,B)=\{0\}$. Then $A=0$, and $A^\dagger=0$. $\Lambda_{A,B}$ does not depend on $\eta$ and has eigenvalues 0 and 1. Since $\dim(E_0)=\dim(E_1)=n+p$, $\Lambda_{A,B}$ is diagonalizable. This is the trivial case where  the sequence $(x_t)_t$ remains constant, keeping the value of the initialization. 
\end{proof}

 \subsection{Proof of Theorem \ref{improvementdagger}.}

\begin{proof}
The beginning of the proof follows from Theorem \ref{thm_conv_OGDA_general_xAy_xBy} and Lemma \ref{lemmepseudo}, using the general properties  $\Ker(A^\dagger)=\Ker(A^T)={\Img(A)}^\perp$ and $\Ker(A)= {\Img(A^\dagger)}^\perp$. We only need to look at the case where $\eta=1/2$.

Then $\S(A,B)\subset \{0,-1\}$, and $\Sp(\Lambda_{A,B})\subset\left\{0,1, \frac12(1+i), \frac12(1-i)\right\}$. As in the proof of Lemma \ref{lem123} one can show that $\Ker(\Lambda_{A,B}-I)^2)= \Ker(\Lambda_{A,B}-I)=\{(x,y,x',y')\in \C^n \times \C^p\times \C^n\times \C^p, x=x', y=y',  A^\dagger x'=0, Ay'=0\}$. So the   algebraic multiplicity of $\lambda=1$ in the characteristic  polynomial of $\Lambda$ is $\dim E_1$, and looking at the Jordan normal form of $\Lambda_{A,B}$ we can conclude for the speed of convergence in the case where $\eta=1/2$.
\end{proof}

\subsection{Proof of Lemma \ref{lem314}}

\begin{proof}
Here $A=\alpha B$ with $\alpha\neq 0$. Thanks to Lemma \ref{lem_valeigs},   all eigenvalues of $\Lambda_{A,B}$ are real. 
Let $\mu\neq 0$ be in $\Sp(B^TA)$, there are 4 elements $\lambda$ in $S^*(\mu)$ and for each of them:
  \[E_\lambda = \left\{(x,y,x',y')^T\in \C^{n}\times \C^{p} \times \C^{n}\times \C^{p},  x = \lambda x' , y=\lambda y', A^TAy'=\frac{\mu}{\alpha} y' , x'= \frac{(1-2\lambda)\eta }{\lambda(1-\lambda)}Ay'  \right\}.\]
with  $\lambda^2(1-\lambda)^2 = \eta^2\mu(2\lambda-1)^2$. 
  Thus $\dim(E_\lambda) = \dim(\Ker(\alpha A^TA-\mu I))$. 

  Regarding $\lambda = 0$, we have $E_0 = \left\{(x,y,x',y'), x = 0, y= 0, Ay'= 0, B^Tx'= 0  \right\}$, so that 
  $\dim(E_0) = \dim(\Ker(A)) + \dim(\Ker(B^T))= \dim(\Ker(A)) + \dim(\Ker(A^T))$. And if  $\lambda = 1$, then 
 $E_1 = \left\{(x,y,x',y'), x = x' , y= y',  Ay'= 0 , B^Tx'= 0  \right\},$ and 
 again, $\dim(E_1) = \dim(\Ker(A)) + \dim(\Ker(A^T))$.
 
 We proceed exactly as in the proof of Lemma  \ref{lemA6} to obtain 
 
 \noindent  $\sum_{\lambda \in Sp(\Lambda_{A,B})} \dim(E_\lambda)=2(n+p)$. As a consequence $\Lambda_{A,B}$ is diagonalizable (notice that when  $\alpha>0$, we could also prove that $\Lambda_{A,B}$ is diagonalizable in $\R^{n+p+n+p}$).\\

Similarly, one can show that if $A$ and $B$ are  square matrices in $\R^{n\times n}$ such that $B^TA$ is diagonalizable with $n$  distinct nonzero real eigenvalues, then $\Lambda_{A,B}$ is diagonalizable .  \\

 Let $\mu_1, \ldots, \mu_n$ be the $n$ nonzero and distinct eigenvalues of $B^TA$. Because $A$ and $B$ are square matrices, $B^TA$ also is, and because it is invertible, $S(A,B) = \Sp(B^TA)$.
 According to Proposition \ref{prop_spectrum_xAy_xBy}, $\Sp(\Lambda_{A, B}) = \bigcup_{\mu \in S(A, B)} S^*(\mu)$.
 If we can prove that all of the $(S^*(\mu_i))_{1\leq i \leq n}$ are disjoints and of cardinal 4, then we would find $4n$ eigenvalues for $\Lambda_{A,B}$, showing that it is diagonalizable.\\
 Looking at Lemma \ref{lem_valeigs}, we can see that if $\mu>0$, $S^*(\mu)$ is included in $\R$ and if $\mu<0$, $S^*(\mu)$ is included in the ball of radius $1$ minus $\R$.
 Thus if $\mu_i>0$ and $\mu_j<0$ for some $i$ and $j$ smaller than $n$, $S^*(\mu_i) \cap S^*(\mu_j) = \varnothing$.\\

 In the first place, let us study the case where $\mu_i$ and $\mu_j$ are both positive.\\
 Let us see $\lambda_k^*$ with $k \in \{1, 2, 3, 4\}$ as a function of $\mu$.
 Then it can be seen, computing the derivative and the limit at $0$ and $+\infty$ of theses functions, that for positive $\mu$, $\lambda_1^*$ is increasing with values in $(1,+\infty)$, $\lambda_2^*$ is increasing with values in $(0,\frac12)$, $\lambda_3^*$ is decreasing with values in $(\frac12,1)$ and $\lambda_4^*$ is decreasing with values in $(-\infty,0)$.
 This shows that for any $\mu_i,\mu_j>0$, with $\mu_j\neq \mu_i$ the cardinal of $S^*(\mu_i)$ and of $S^*(\mu_j)$ is 4, and that these two sets are disjoint.\\

 Now, let us study the case where both $\mu_i$ and $\mu_j$ are negative.\\
 Here, the element of $S^*(\mu_i)$ and $S^*(\mu_j)$ are complex numbers.
 The image of the element in $S^*(\mu_i)$ lies in $\{-\sqrt{\abs{\mu_i}}, \sqrt{\abs{\mu_i}}\}$, thus for $\mu_i \neq \mu_j$, $S^*(\mu_i) \cap S^*(\mu_j) = \varnothing$.
 To finish, we just need to show that for negative $\mu$, the cardinal of $S^*(\mu)$ is 4.
 Yet, for such $\mu$, $\Im(\lambda_1^*(\mu))>0$ and $\Re(\lambda_1^*(\mu)) > \frac12$, $\Im(\lambda_2^*(\mu))>0$ and $\Re(\lambda_2^*(\mu)) < \frac12$, $\Im(\lambda_3^*(\mu))<0$ and $\Re(\lambda_3^*(\mu)) > \frac12$, and $\Im(\lambda_4^*(\mu))<0$ and $\Re(\lambda_4^*(\mu)) < \frac12$.
 Each eigenvalues are in four different regions of the complex plan: they are all different from each other.\\

 Finally, we have shown that the spectrum of $\Lambda_{A,B}$ is composed of $4n$ distinct eigenvalues, which implies that $\Lambda_{A,B}$ is diagonalizable.
 \end{proof}

\subsection{Proof of Theorem \ref{thm_conv_OGDA_general_coop}}

We know from Proposition \ref{prop_spectrum_xAy_xBy} that:
$\Sp(\Lambda_{A,B}) = \bigcup_{\mu \in \S(A,B)} S^*(\mu),$ with  $S^*(\mu)  = \{ \lambda \in \C, \lambda^2(1-\lambda)^2 =\mu \eta^2(1-2\lambda)^2\}.$
%   Define $E^*_\lambda$ as the set of real eigenvectors of $\Lambda_{A,B}$ associated to an eigenvalue $\lambda$. \\

\begin{proof}
  For the first part of the proof, let us assume that $b, c, d, e, f$ and $g$ are zero vectors.
  We will then use it to cope with the more general case where they can be nonzero.\\
  We have for each $t$, $Z_t=\Lambda_{A,B}^t Z_0$, with $Z_t$ the column vector $(x_t^T,y_t^T, x_{t-1}^T,y_{t-1}^T)^T$.
  Because $\Lambda_{A,B}$ is assumed to be diagonalizable, $Z_0$ can be uniquely written
  $Z_0=\sum_{\mu \in \S(A,B)}\sum_{\lambda \in S^*(\mu)} z_{0, \lambda}$, with $z_{0, \lambda}\in E_\lambda$ for each $\lambda$, where $E_\lambda$ is the eigenspace associated with $\lambda$.\\
  Then for each $t$, 
  $$Z_t=\sum_{\mu \in \S(A,B)}\sum_{\lambda \in S^*(\mu)} \lambda^t z_{0, \lambda}.$$
  Define $\rho_*=\max\{\abs{\lambda} \in \Sp(\Lambda_{A,B}), z_{0,\lambda}\neq 0\},$ and $\lambda_*$ an eigenvalue verifying $\abs{\lambda_*} = \rho_*.$\\

  If $\rho_*<1$, then $Z_t \xrightarrow[t\to \infty]{}  0$ and $(x_t,y_t)  \xrightarrow[t\to \infty]{} (0,0)$, which is a Nash equilibrium of the game.\\% This is just the same as what is told in section 2.https://fr.overleaf.com/project/624448802a2628abd5f70c83

  If $\rho_*=1$, looking at Lemma \ref{lem_valeigs}, we can see that $\lambda_*=1$ is the only possible option such that $\abs{\lambda_*} = 1.$ 
  Then $Z_t \xrightarrow[t\to \infty]{}  z_{0,1} \in E^*_1$. So $(x_t,y_t)_t$ converges to a vector $(x,y)$ with $Ay=0$ and $B^Tx=0$, that is a Nash equilibrium of the game. \\

  Assume finally that $\rho_*>1$, we will  prove that $x_t^TAy_t \xrightarrow[t\to \infty]{} +\infty$ and $x_t^TAy_t \xrightarrow[t\to \infty]{} +\infty.$
  Let $\lambda_*$ be such that $\abs{\lambda_*} = \rho_*$.
  According to Lemma \ref{lem_valeigs}, $\lambda_*$ is a positive real: $\lambda_*=\rho_*>1$, and $\lambda_* \in S^*(\mu_*)$, with $\mu_*>0$ the largest element of $\S(A,B)$. \\
  Consider $z_{0, \lambda_*}= (x,y,x',y')^T\in E_{\lambda_*}$, we have $z_{0, \lambda_*}\neq 0$, $x = \lambda_* x'$, $y=\lambda_* y'$, $B^TAy'= \mu_* y'$ and  $x'= \frac{(1-2\lambda_*)\eta }{\lambda_*(1-\lambda_*)}Ay'$. 
  Since $\lambda_*>1$, we not only have  $\lambda_* \in S^*(\mu_*)$, but we also have $\lambda_*(1-\lambda_*)=\sqrt{\mu_*} \eta (1-2\lambda_*)$.\\
  We obtain:
  \[\langle x',Ay'\rangle=\frac{1}{\sqrt{\mu_*}} \|Ay'\|^2>0, \;{\rm and}\; \langle x,Ay\rangle=\frac{\lambda_*^2}{\sqrt{\mu_*}} \|Ay'\|^2 >0,\]
  and 
  \[\langle B^Tx',y'\rangle=\sqrt{\mu_*} \|y'\|^2>0, \;{\rm and}\; \langle B^Tx,y\rangle=\lambda_*^2\sqrt{\mu_*} \|y'\|^2 >0.\]
 Now, $x_t \sim \lambda_*^t x$ and $y_t \sim \lambda_*^t y$ when $t\to \infty$, so the payoff $\langle x_t, Ay_t\rangle$ of player 1 is equivalent to $\lambda_*^{2t} \langle x, Ay\rangle$ and we get:
  \[\langle x_t, Ay_t\rangle\sim \lambda_*^{2t} \langle x, Ay\rangle \xrightarrow[t\to \infty]{} +\infty. \]
  Similarly,  
  \[\langle B^Tx_t, y_t\rangle\sim \lambda_*^{2t} \langle B^Tx, y\rangle \xrightarrow[t\to \infty]{} +\infty\].\\

  We now consider the case where  $b,c,d,e,f$ and $g$ are  arbitrary.
  We proceed as in the proof of Theorems \ref{thm_conv_OGDA_general_xAy_xBy} and \ref{thm_conv_OGDA_zs_general}. Fix a Nash equilibrium ($x^*,y^*)$. Define for each $t\geq -1$, $x'_t=x_t-x^*\; {\rm and}\; y'_t=y_t-y^*.$ Easy calculations show that:
  \[\forall t\geq 0,\;  \left\{
  \begin{array} {ccc} 
  x'_{t+1} & = & x'_t  \;+\; 2\eta Ay'_t - \eta Ay'_{t-1},  \\
  y'_{t+1}  & = & y'_t \;  - 2\eta B^T x'_t + \eta B^Tx'_{t-1}.
  \end{array}\right. \]
  So  the sequence $(x'_t,  y'_t)_t$ follows the OGDA algorithm \eqref{OGDA_nzs} of section \ref{subsec_xAy_xBy}.
  Now, according to the first part of the proof, there are two cases depending on $\rho_* = \max\{\abs{\lambda} | \lambda \in \Sp(\Lambda_{A,B}), z_{0,\lambda} \neq 0\}$ :\\
Either $\rho_* \leq 1$, in which case $(x'_t, y'_t)_t$ will converge to exponentially fast to a point $(x'_\infty, y'_\infty) \in \Ker(B^T) \times \Ker(A)$.
  As a consequence, $(x_t, y_t)_t$ converges to the limit $(x_\infty, y_\infty) = (x'_\infty + x^*,y'_\infty+y^*)$, which is a fixed point of OGDA, and thus a Nash equilibrium of the game.\\

  Or $\rho_* > 1$, in which case $x'_t \sim \lambda_*^t x'$ and $y'_t \sim \lambda_*^t y$.
  From the fact that $x_t = x'_t+x^*$ and $y_t = y'_t + y^*$, we have $x_t \sim \lambda_*^t x'$ and $y_t \sim \lambda_*^t y'$.
  We finally have for the payoffs at stage $t$:
  \[ x_t^TAy_t+b^Tx_t+c^Ty_t + d \sim \lambda_*^{2t} {x'}^TA{y'} + \lambda_*^t b^Tx' + \lambda_*^t c^Ty' + d \sim \lambda_*^{2t} {x'}^TA{y'} \xrightarrow[t\to \infty]{} +\infty,\] and 
  \[ x_t^TBy_t+e^Tx_t+f^Ty_t + g \sim \lambda_*^{2t} {x'}^TB{y'} + \lambda_*^t e^Tx' + \lambda_*^t f^Ty' + g \sim \lambda_*^{2t} {x'}^TB{y'} \xrightarrow[t\to \infty]{} +\infty\]
  because, as before, ${x'}^TA{y'}$ and ${x'}^TB{y'}$ are positive.
\end{proof}

\end{document}